\documentclass[aap,preprint]{imsart}
\usepackage[dvipsnames]{xcolor}

\RequirePackage[utf8]{inputenc}
\usepackage{mathrsfs,soul}
\usepackage{hyperref}
\usepackage{amssymb,amsthm,amsfonts,amsbsy,latexsym,dsfont}
\usepackage[textsize=small]{todonotes}       % includes TO DO LIST
\usepackage{graphicx}
\usepackage[numeric,initials,nobysame]{amsrefs}
\usepackage{colortbl}

\usepackage{enumerate}
\usepackage{graphicx}
\usepackage{subcaption}   

\usepackage{amsmath}

\newenvironment{enumeratei}{\begin{enumerate}[\upshape (i)]}{\end{enumerate}}
\newenvironment{enumeratea}{\begin{enumerate}[\upshape (a)]}{\end{enumerate}}

\usepackage{paralist}

\newenvironment{inparaenuma}{\begin{inparaenum}[\upshape \bfseries (a) ]}{\end{inparaenum}}

\numberwithin{equation}{section}
\numberwithin{figure}{section}
\numberwithin{table}{section}

\startlocaldefs

\newtheorem{thm}{Theorem}[section]
\newtheorem{lem}[thm]{Lemma}
\newtheorem{theorem}[thm]{Theorem}
\newtheorem{cor}[thm]{Corollary}
\newtheorem{corollary}[thm]{Corollary}
\newtheorem{prop}[thm]{Proposition}
\newtheorem{defn}[thm]{Definition}

\newtheorem{ass}[thm]{Assumption}

\newtheorem{lemma}[thm]{Lemma}
\newtheorem{constr}[thm]{Construction}

\theoremstyle{definition}
\newtheorem{rem}[thm]{Remark}

\newcommand{\ind}{\mathds{1}}
\newcommand{\eps}{\varepsilon}

\newcommand{\set}[1]{\left\{#1\right\}}

\newcommand{\equald}{\stackrel{\mathrm{d}}{=}}
\newcommand{\probc}{\stackrel{\mathrm{P}}{\longrightarrow}}

\newcommand{\convas}{\stackrel{\mathrm{a.s.}}{\longrightarrow}}

\newcommand{\stod}{\preceq_{\mathrm{st}}}
\newcommand{\sustod}{\succeq_{\mathrm{st}}}

\definecolor{uncblue}{rgb}{0.29, 0.61, 0.83}
\definecolor{banared}{rgb}{0,0,0}

\newcommand{\chr}[1]{\textcolor[rgb]{0,0,0}{{#1}}}
\newcommand{\chnr}[1]{\textcolor[rgb]{0,0,0}{{#1}}}
\newcommand{\chsb}[1]{\textcolor[rgb]{0,0,0}{{#1}}}
\newcommand{\sba}[1]{{{#1}}}
 \newcommand{\bbb}[1]{\textcolor[rgb]{0,0,0}{{#1}}}
\newcommand{\chiain}[1]{\textcolor[rgb]{0,0,0}{{#1}}}
\newcommand{\rev}[1]{\textcolor[rgb]{0,0,0}{{#1}}}

\def\qed{ \hfill $\blacksquare$}
\newcommand{\cB}{\mathcal{B}}\newcommand{\cC}{\mathcal{C}}
\newcommand{\cD}{\mathcal{D}}\newcommand{\cE}{\mathcal{E}}\newcommand{\cF}{\mathcal{F}}
\newcommand{\cG}{\mathcal{G}}\newcommand{\cH}{\mathcal{H}}

\newcommand{\cM}{\mathcal{M}}
\newcommand{\cR}{\mathcal{R}}
\newcommand{\cT}{\mathcal{T}}
\newcommand{\cV}{\mathcal{V}}

\newcommand{\vp}{\mathbf{p}}

\newcommand{\mvtheta}{\boldsymbol{\theta}}

\newcommand{\bL}{\mathbb{L}}
\newcommand{\bN}{\mathbb{N}}
\newcommand{\bR}{\mathbb{R}}

\newcommand{\bZ}{\mathbb{Z}}
\newcommand{\dA}{\mathds{A}}\newcommand{\dB}{\mathds{B}}

\newcommand{\dG}{\mathds{G}}

\newcommand{\dL}{\mathds{L}}

\newcommand{\sC}{\mathfrak{C}}

\DeclareMathOperator{\E}{\mathbb{E}}
\DeclareMathOperator{\pr}{\mathbb{P}}

\DeclareMathOperator{\var}{Var}

 \DeclareMathOperator{\BP}{BP}
 \DeclareMathOperator{\PA}{PA}

\newcommand{\sss}{\scriptscriptstyle}

\newcommand{\convd}{\stackrel{d}{\longrightarrow}}
\newcommand{\convp}{\stackrel{P}{\longrightarrow}}

\usepackage{fourier}
\usepackage{pdfsync}
\definecolor{aqua}{rgb}{0.0, 1.0, 1.0}
\definecolor{boo}{rgb}{1.0, 0.0, 1.0}

\newcommand{\ac}{{\sf AC} }
\newcommand{\bc}{{\sf BC} }
\newcommand{\age}{{\sf Age} }

\endlocaldefs

\begin{document}

\begin{frontmatter}

% "Title of the paper"
% \title{Fluctuation bounds for continuous time branching processes and \chr{evolution of} growing \chsb{tree} networks \chr{with a change point}}
\title{Fluctuation bounds for continuous time branching processes and \chr{evolution of} growing \chsb{trees} \chr{with a change point}}
\runtitle{CTBP and \chr{evolution of networks under change point}}

% indicate corresponding author with \corref{}
% \author{\fnms{John} \snm{Smith}\corref{}\ead[label=e1]{smith@foo.com}\thanksref{t1}}
% \thankstext{t1}{Thanks to somebody} 
% \address{line 1\\ line 2\\ printead{e1}}
% \affiliation{Some University}

\author[A]{\fnms{Sayan} \snm{Banerjee}\ead[label=e1]{sayan@email.unc.edu}},  
\author[A]{\fnms{Shankar} \snm{Bhamidi}\corref{}\ead[label=e2]{bhamidi@email.unc.edu}} \and 
\author[B]{\fnms{Iain} \snm{Carmichael}\ead[label=e3]{iain@berkeley.edu}}

\address[A]{Department of Statistics and Operations Research, UNC Chapel Hill \\ \printead{e1}, \\ \printead{e2}}
\address[B]{Department of Statistics, UC Berkeley \\ \printead{e3}}
%\affiliation{University of North Carolina, Chapel Hill}

% \address{\printead{Department of Statistics and Operations Research, UNC Chapel Hill, NC 27599}}
%\and
%\author{\fnms{???} \snm{???}\ead[label=e2]{???}}
%\address{\printead{e2}}
%\affiliation{???}

\runauthor{Banerjee, Bhamidi and Carmichael}

% TODO: maybe add a sentence noting the generality of our model e.g. "A special case of this model is preferential attachment, however, the model we consider is much more general (e.g. the model could be anti-preferential if f is decreasing)."
\begin{abstract}
\chsb{We consider dynamic random trees constructed using an attachment function $f: \bN \rightarrow \mathbb{R}_+$ where, at each step of the evolution, a new vertex attaches to an existing vertex $v$ in the current tree with probability proportional to $f(\text{degree}(v))$.
We explore the effect of a change point in the system; the dynamics are initially driven by a function $f$ until the tree reaches size $\tau(n) \in (0, n)$, at which point the attachment function switches to another function, $g$, until the tree reaches size $n$.
% the dynamics are driven by an initial function $f$ until the tree reaches size $\tau(n) \in (0, n)$, at which point the attachment function switches to $g$. 
Two change point time scales are considered, namely the \emph{standard model} where $\tau(n) = \gamma n$, and the \emph{quick big bang model} where $\tau(n) = n^\gamma$, for some $0<\gamma <1$.
In the former case, we obtain deterministic approximations for the evolution of the empirical degree distribution (EDF) in sup-norm and use these to devise a provably consistent non-parametric estimator for the change point $\gamma$. In the latter case, we show that the effect of pre-change point dynamics asymptotically vanishes in the EDF, although this effect persists in functionals such as the maximal degree. Our proofs rely on embedding the discrete time tree dynamics in an associated (time) inhomogeneous continuous time branching process (CTBP). In the course of proving the above results, we develop novel mathematical techniques to analyze both homogeneous and inhomogeneous CTBPs and obtain rates of convergence for functionals of such processes, which are of independent interest.} 

\end{abstract}


\end{frontmatter}

% AOS,AOAS: If there are supplements please fill:
%\begin{supplement}[id=suppA]
%  \sname{Supplement A}
%  \stitle{Title}
%  \slink[doi]{10.1214/00-AOASXXXXSUPP}
%  \sdatatype{.pdf}" 
%  \sdescription{Some text}
%\end{supplement}

%
%\iainnote{
%Perhaps insert after the first sentence in the abstract:
%
%``This model includes the special case of ``linear preferential attachment", but is much more general (e.g. the model could be anti-preferential if $f$ is decreasing)."}

\section{Introduction}
\label{sec:int}

Driven by the explosion in the amount of data on various real world networks, the last few years have seen the emergence of many new mathematical network models. Motivations behind these models are diverse \cite{albert2002statistical,newman2003structure,newman2010networks,bollobas2001random,durrett-rg-book,van2009random} including \begin{inparaenuma}
    \item extracting unexpected patterns in the network (e.g. community detection);
    \item \chnr{understanding} properties of dynamics (e.g. the spread of epidemics);
    \item understanding mechanistic reasons for the emergence of empirically observed properties of \chnr{real world systems}. 
\end{inparaenuma} 
\chnr{An important niche} is the setting of networks that evolve over time. In the context of probabilistic combinatorics, these models have been studied for decades in the vast field of \emph{recursive} trees, e.g. see \cite{mahmoud2008polya,bergeron1992varieties,flajolet2009analytic,drmota2009random} and the references therein.

To fix ideas, consider the general random tree model called \textit{non-uniform random recursive trees} \cite{szymanski1987nonuniform}.
Fix $n \ge 1$ and an \textit{attachment function} $f:\{0,1,2\dots\} \to  (0, \infty)$.
\chsb{A sequence of random trees $\set{\cT_k: 1\leq k \leq n}$ is grown as follows ($\cT_k$ has $k$ vertices labelled by the integers $[k] := \{1, \dots, k\}$). For $k=1$, $\cT_1$ has one vertex, which we call the ``root.''}
For fixed $k \ge 2$, $\cT_k$  is constructed conditional on $\cT_{k-1}$ as follows.
\chsb{A new vertex, $k$, is born into the system and attaches to a previously existing vertex $v \in [k-1]$ with probability proportional to $f(\text{deg}(v))$, where $\text{deg}(v)$ denotes the number of children of $v$ (which is one less than its graph degree in $\cT_{k-1}$).}
Thus,
\begin{equation*}
\chsb{\mathbb{P}\left(k \text{ attaches to } v \in [k-1] \ \vert \ \cT_{k-1}\right) := \frac{f(\text{deg}(v))}{\sum_{u=1}^{k-1} f(\text{deg}(u))}.}
\end{equation*}
The vertex that $k$ selects is called the ``parent" and the edge is directed from the parent to the new ``child" vertex.
The case of $f(\cdot)\equiv 1$ corresponds to the famous class of \chnr{\textit{random recursive trees}} \cite{smythe1995survey}.
The specific case of \chsb{``linear preferential attachment" when $f$ is affine} was considered in \cite{barabasi1999emergence} to provide a generative story for heavy tailed degree distributions of real networks. % where they showed, via non-rigorous arguments, that the resulting graph has a heavy tailed degree distribution with exponent $3$ in the large $n\to\infty$ limit; this was rigorously proved in \cite{Bollobas:2001:DSS:379831.379835}.

Next, consider the non-uniform random recursive tree model with a \emph{change point}.
Here, the random tree is grown according to one rule till some (possibly random) time called the change point, after which the dynamics switch to another rule.
In detail, let $1 \le \tau < n$ and  $f_0, f_1:\set{0,1,2,\ldots} \to (0,\infty)$ be two attachment functions.
For $1 \le k \le \tau$ the process evolves according to the \textit{initializer} function $f_0$ i.e. node $k$ attaches to pre-existing vertex $v \in [k-1]$ with probability proportional to \sba{$f_0(\text{deg}(v))$}.
After the change point for $k \in [\tau + 1, n]$ the process evolves according to $f_1$. We denote this sequence of random trees by $\set{\cT_k^{\mvtheta}: 1\leq k \leq n}$, where $\mvtheta = (f_0,f_1,\tau)$.
While the focus of this paper is on one change point, the methodology allows one to derive analogous results for multiple change points.

\subsection{Informal description of our aims and results}

This paper has the following two major aims for the models described above:

\begin{enumeratea}
    \item \chr{Asymptotics in the large network limit as well as corresponding functionals have been derived for a host of random tree models \cite{aldous1991asymptotic,holmgren2017fringe,bhamidi2007universal}. One major driver of research has been proving convergence of the empirical distribution of these functionals to model dependent constants. Establishing (even suboptimal) rates of convergence for these models has been non-trivial other than for models related to urn models e.g. see the seminal work of Janson \cite{janson2004functional}. The aim of this paper is to develop robust methodology for proving such error bounds for general models.} \chsb{These error bounds play a key role in understanding robustness properties of network source detection problems, see e.g. \cite{banerjee2020root}.} % Our results will not be optimal owing to the generality of the model considered in the paper; however using the techniques in this paper coupled with higher moment assumptions can easily lead to more refined results for specific models. To keep the paper to manageable length, we focus on the degree distribution but see Section \ref{sec:disc} for our work in progress of using the methodology in this paper for more general functionals. 
	
\item %\chiain{Suppose there is a change point in the model above.Individuals initially entering the system are governed by an attachment function $f$, then after some time point, new individuals are governed by another attachment function $g$.}
%Consider general models of network evolution as described above but wherein, beyond some point, new individuals entering the system change their behavior. This is reflected via a change in the the attachment function $f$ to a different attachment function $g$. 
\chiain{We aim to understand the effect of change points on structural properties of the network.} % of the network model and the interplay between the time scale of the change point and the nature of the attachment functions before and after the change point. 
	  \chnr{Analogous to classical change point detection, we start by considering models with a change point at time $\chiain{\tau =  \lfloor \gamma n \rfloor}$ for $0<\gamma <1$ (\bbb{referred to as the \emph{standard model} in the sequel }). \chsb{Using \sba{techniques involving embedding into inhomogeneous} continuous time branching processes, we approximate the empirical degree distribution (EDF) as it evolves in time by deterministic curves derived from the attachment functions \chiain{$f_0$ and $f_1$}. Using this, we devise a non-parametric estimator for $\gamma$.}  Counter-intuitively, we find that irrespective of the value of $\gamma$, structural properties of the network such as the tail of the degree distribution are {\bf only} determined by model parameters {\bf before} the change point. Motivated by this we consider other time scales of the change point, \chsb{namely, when the change happens at time $\chiain{\tau =  \lfloor n^\gamma \rfloor}$ for $0<\gamma <1$ (the \emph{quick big bang} model) to understand the extent of this long range dependence phenomenon. In this case, we show that the effect of the pre-change point dynamics asymptotically vanishes in the EDF. However, for specific examples of attachment functions \chiain{$f_0$ and $f_1$}, we conclude that \bbb{this change point has a drastic effect on asymptotics for the maximal degree.}}}

\end{enumeratea}

\section{Preliminaries}
\label{sec:not}

% \subsection{Mathematical notation}
We use $\stod$ for stochastic domination between two real valued probability measures. For $J\geq 1$, let $[J]:= \set{1,2,\ldots, J}$. A random variable $Y$ with rate $\lambda$ exponential distribution is denoted by $Y\sim \exp(\lambda)$. Write $\bZ$ for the set of integers, $\bR$ for the real line, $\bN$ for the set of natural numbers and let $\bZ_+:=\set{0,1,2,\ldots}$, $\bR_+:=(0,\infty)$. Write $\convas,\convp,\convd$ for convergence \bbb{almost surely}, in probability and in distribution respectively. \bbb{For non-negative function $g$ and another function $f$ both defined on $\bN$},
we write $f(n)=O(g(n))$ when $|f(n)|/g(n)$ is uniformly bounded, and $f(n)=o(g(n))$ when $\lim_{n\rightarrow \infty} f(n)/g(n)=0$. Furthermore, write $f(n)=\Theta(g(n))$ if $f(n)=O(g(n))$ and $g(n)=O(f(n))$. A sequence of events $(A_n)_{n\geq 1}$
occurs \emph{with high probability} (whp) when $\pr(A_n)\rightarrow 1$. \chsb{For some $\sigma$-field $\cF$, \bbb{an} integrable random variable $X$ and non-random constant $C$, when we write $\mathbb{E}\left(X \vert \cF\right) \le C$, this will formally mean that the bound holds with probability one.}
For a sequence of increasing rooted trees $\set{\cT_n:n\geq 1}$ (random or deterministic), we assume that edges are directed from parent to child (with the root as the original progenitor).
 \chiain{For exposition purposes we will write \emph{degree} for \emph{out-degree} i.e. the number of children of a vertex. This should not be confused with the \textit{total degree} or \textit{graph degree}, which is the sum of incoming and outgoing edges (and thus the graph degree of a vertex is always one more than the out-degree in our tree networks).}
For $n\geq 1$ and $k\geq 0$, let $D_n(k)$ be the number of vertices in $\cT_n$ with degree $k$; thus $D_n(0)$ counts the number of leaves in $\cT_n$.

\subsection{Organization of the paper}
\bbb{The rest of} Section \ref{sec:not} defines \bbb{key} objects required to state our main results. Section \ref{sec:struct-res} contains the main results. In Section \ref{sec:disc} we discuss the relevance of this work and related literature. The remaining sections are devoted to \bbb{the} proofs of the main results.

\subsection{Branching processes}
\label{sec:br-def}

Fix \bbb{an} attachment function $f$. For $i\geq 0$ let $E_i \sim \exp(f(i)), i \ge 0$ be a sequence of independent exponential random variables. \bbb{Define for $i\geq 1$,} $L_i:= \sum_{j=0}^{i-1} E_j$. Let $\xi_f$ be the point process on $\bR_+$:
\vspace{-.1in}
\begin{equation} % MEOW: can we do -vspace here?
\label{eqn:xi-f-def}
\xi_f:=(L_1, L_2, \ldots).
\end{equation}
Abusing notation, write for $t\geq 0$,
\begin{equation}
\label{eqn:xi-f-t}
    \xi_f[0,t]:= \#\set{i: L_i \leq t}, \qquad \mu_f[0,t]:= \E(\xi_f[0,t]). 
\end{equation}
Here we view $\mu_f$ as a measure on $(\bR_+, \cB(\bR_+))$. We will \bbb{also need variants} of the above objects: for fixed $k\geq 0$, let $\xi_f^{(k)}$ denote the \chsb{$k$-shifted version of the point process $\xi_f$} where the first inter-arrival time is $E_k$ i.e. define the sequence, $L_i^{(k)} = E_k + E_{k+1}+\cdots E_{k+i-1}, i\geq 1$ and then let
\begin{equation}
\label{eqn:xi-k-f-t-def}
    \xi_f^{(k)}:=(L_1^{(k)}, L_2^{(k)}, \ldots),  \qquad \mu_f^{(k)}[0,t]:= \E(\xi_f^{(k)}[0,t]). 
\end{equation}
As above, $\xi_f^{(k)}[0,t]:= \#\{i: L_i^{(k)} \leq t\}$. We abbreviate $\xi_f[0,t]$ as $\xi_f(t)$ and similarly $\mu_f(t)$, $\xi_f^{(k)}(t)$, $\mu_f^{(k)}(t)$. %$\xi_f[0,t], \mu_f[0,t], \xi_f^{(k)}[0,t],\mu_f^{(k)}[0,t]$ by $\xi_f(t), \mu_f(t), \xi_f^{(k)}(t),\mu_f^{(k)}(t)$ respectively. 
\sba{Define the Ulam-Harris set $\mathcal{I} := \cup_{d=0}^{\infty}\mathbb{N}^d$ where $\mathbb{N} = \{1,2,\dots\}$ and \rev{$\mathbb{N}^0 = \set{\emptyset}$ for the root of the tree.}}

\begin{defn}[Continuous time Branching process (CTBP) \cite{jagers-ctbp-book,athreya1972}] \label{def:ctbp} \chiain{Fix an attachment function $f:\bZ_+ \to \bR_+$.} A continuous time branching process driven by $f$, written as $\{\BP_f(t) : t \ge 0\}$, is a \sba{$\mathcal{I}$-valued process, started with one individual $\emptyset$ (the root) at time $t=0$, such that every individual $x \in \mathcal{I}$ born into the system gives birth to offspring $\{(x,i) : i \in \mathbb{N}\}$ with birth times given by an independent copy of the point process $\xi_f$ defined in \eqref{eqn:xi-f-def}. For $t \ge 0$, $\BP_f(t)$ denotes the set of individuals alive at time $t$ and $Z_f(t):= |\BP_f(t)|$ denotes the size of this set. For $x \in \mathcal{I}$, let $\sigma_x$ denote the birth time of $x$.}  
\end{defn}
% We refer the interested reader to \cite{jagers-ctbp-book,athreya1972} for general theory regarding continuous time branching processes.
\chsb{In analogy with the original tree model, we will often refer to individuals in the branching process as vertices and the number of children of an individual in the population at time $t$ as its \emph{degree} at time $t$.} In our construction, by \bbb{the strict positivity} assumption on the attachment function, individuals continue to reproduce forever. Let % Write $m_f(\cdot)$ for the corresponding expectation i.e.,
\begin{equation}
\label{eqn:mft-def}
    m_f(t):= \E(Z_f(t)), \qquad t\geq 0,
\end{equation} 
\chnr{\chsb{If $\limsup_{k \rightarrow \infty}f(k)/k < \infty$}, it can be shown \cite[Chapter 3]{jagers-ctbp-book} that for all $t>0$, $m_f(t) <\infty$, \sba{and that $m_f(\cdot)$} is strictly increasing with $m_f(t)\uparrow\infty$ as $t\uparrow\infty$. In the sequel, to simplify notation we will suppress dependence on $f$ and write $\BP(\cdot), m(\cdot)$ etc. The connection between CTBP and the discrete random tree models is given by the following result and is the starting point of the Athreya-Karlin embedding \cite{athreya1968}. }

\begin{lemma}\label{lem:ctb-embedding-no-cp}
    Fix \bbb{an} attachment function $f$ and consider the sequence of random trees $\set{\cT_m: 2\leq m\leq n }$ constructed using attachment function $f$. Consider the continuous time construction in Definition \ref{def:ctbp} and define for $m\geq 1$ the stopping times $T_m:=\inf\set{t\geq 0: |\BP_f(t)| =m}$. Then viewed as a sequence of growing random labelled rooted trees we have, $\set{\BP_f(T_m): 2\leq m\leq n} \stackrel{d}{=} \set{\cT_m:2\leq m\leq n}.$
\end{lemma}
\subsection{Continuous embedding of model with \bbb{single} change point}\label{cpembedding}
\chsb{The continuous time embedding of the tree model without change point has a natural extension to the model with a \bbb{single} change point $\tau \in \bN$. Individuals in the population reproduce according to independent copies of the point process $\xi_{f_0}$ up till the time $t(\tau)$ when the total population size is $\tau$. After this time, individuals continue to reproduce independently \bbb{as follows}. An individual of degree $\ell$ at time $t(\tau)$ reproduces according to the point process $\xi_{f_1}^{(\ell)}$. New individuals born into the system after time $t(\tau)$ reproduce according to $\xi_{f_1}$. \bbb{Recalling the notation for driving parameters} $\mvtheta = (f_0,f_1,\tau)$ and denoting this \emph{inhomogeneous} branching process by $\BP_{\mvtheta}(\cdot)$, the same approach used to prove Lemma \ref{lem:ctb-embedding-no-cp} shows that $\set{\BP_{\mvtheta}(T^{\mvtheta}_m): 2\leq m\leq n} \stackrel{d}{=} \set{\cT^{\mvtheta}_m:2\leq m\leq n}$, where $T^{\mvtheta}_m := \inf\set{t\geq 0: |\BP_{\mvtheta}(t)| =m}, m \ge 1$. Note that $t(\tau) = T^{\mvtheta}_{\tau}$. We will refer to this random time $t(\tau)$ as the change point for the branching process $\BP_{\mvtheta}(\cdot)$. When $\mvtheta$ is clear from context, we will often drop the superscript (or subscript) $\mvtheta$ from associated quantities for notational convenience.}

\begin{rem}\label{probspccoup}
\sba{There exists a common probability space $(\Omega^*, \cF^*, \mathbb{P}^*)$ on which the process $\set{\BP_{\mvtheta}(T^{\mvtheta}_m): 2\leq m\leq n}$, and hence $\set{\cT^{\mvtheta}_m:2\leq m\leq n}$, can be constructed for all $n$.
$(\Omega^*, \cF^*, \mathbb{P}^*)$ can be taken to be a probability space on which the countable i.i.d. collection $\{\xi_{f_0,x} : x \in \mathcal{I}\}$ and i.i.d. collection $\{\xi_{f_1,x} : x \in \mathcal{I}\}$ of point processes are defined. Individual $x$ with $\sigma_x < t(\tau)$ uses $\xi_{f_0,x}$ to reproduce until the stopping time $t(\tau)$ when the total population size hits $\tau$. If $x$ has $k$ children at time $t(\tau)$, it uses the $k$-shifted version $\xi_{f_1,x}^{(k)}$ of the point process $\xi_{f_1,x}$ for future reproduction. For $x \in \mathcal{I}$ with $\sigma_x \ge t(\tau)$, $x$ has all its reproduction according to $\xi_{f_1,x}$. Future references to convergence in probability and almost sure convergence for the associated branching processes and trees with change point will all be implicitly assumed to take place on $(\Omega^*, \cF^*, \mathbb{P}^*)$.}
\end{rem}

\subsection{Assumptions on attachment functions}

\chnr{Here we set up assumptions as well as constructions needed to state the main results.  We mainly follow \cite{jagers-ctbp-book,jagers1984growth,nerman1981convergence,rudas2007random}. }

\begin{ass}
    \label{ass:attach-func} 
    \begin{enumeratei}
   %     \item {\bf Positivity:} $f$ is assumed strictly positive i.e. $f(k) > 0$ for all $k$. 
        \item Every attachment function $f$ \chiain{is strictly positive and} can grow at most linearly,  \chsb{$$\sup_{k \ge 0} f(k)/(k+1) = C< \infty.$$} % This is equivalent to there existing a constant $C$ such that $f(k)\leq C(k+1)$ for all $k\geq 0$.
        \item Consider the following function $\hat{\rho}:(0,\infty)\to (0,\infty]$ defined via,
\begin{equation}
\label{eqn:rho-hat-def}
    \hat{\rho}(\lambda):= \sum_{k=1}^\infty \prod_{i=0}^{k-1} \frac{f(i)}{\lambda + f(i)}. 
\end{equation}
 Define 
$\underline{\lambda}:= \inf\set{\lambda > 0: \hat{\rho}(\lambda) < \infty}$. 
We assume,
\begin{equation}
\label{eqn:prop-under-lamb}
 \lim_{\lambda\downarrow\underline{\lambda}} \hat{\rho}(\lambda) > 1. 
\end{equation}
\end{enumeratei}
\end{ass}
Using (ii) of the above Assumption, let $\lambda^*:=\lambda^*(f)$ be the unique $\lambda$ such that 
\begin{equation}
\label{eqn:malthus-def}
    \hat{\rho}(\lambda^*) = 1. 
\end{equation}

%This object is often referred to as the \textit{Malthusian rate of growth parameter}. \chr{While not at all obvious and intricately tied to the growth rate of a continuous time branching process, this object plays an essential role in all law of large number results about the {\bf discrete time} tree model with attachment function $f$ (see e.g. Lemma \ref{lem:deg_dist_quad_conv}). } % and will play a key role in our analysis and results.

 \chnr{The \textit{Malthusian rate of growth} parameter, $\lambda^*$, is intricately tied to the growth rate of a continuous time branching process.} \chsb{In fact, recalling that $Z_f(t)$ denotes the size of the branching process at time $t$, under Assumption \ref{ass:attach-func}, $e^{-\lambda^* t} Z_f(t)$ converges in probability as $t \rightarrow \infty$ to a finite random variable \cite[Theorem 3.1]{nerman1981convergence} (see also Lemma \ref{lem:deg_dist_quad_conv} below).}
 \chr{While not obvious, $\lambda^*$ plays an essential role in all law of large number results about the {\bf discrete time} tree model with attachment function $f$ \chsb{(see, for example, \eqref{eqn:pk-zero-def} below).}}

\begin{rem}
\chsb{The requirement \eqref{eqn:prop-under-lamb} is a standard assumption in branching process literature that implies almost sure convergence of a broad collection of branching process statistics of ratio-type \cite[Theorem 6.3]{nerman1981convergence}. It is algebraic in nature and can be checked for a given attachment function using the explicit form of $\hat{\rho}(\cdot)$ given in \eqref{eqn:rho-hat-def}. In particular, \eqref{eqn:prop-under-lamb} is satisfied for $f(\cdot) \equiv 1$ (easy to check that $\hat{\rho}(\lambda) = \lambda^{-1}, \lambda>0$), $f(k) = k+1+\beta, k \ge 0,$ for any $\beta>0$ \cite[Section 4.2]{rudas2007random} and $f(k) = (k+1)^{\alpha}, k \ge 0,$ for any $\alpha \in (0,1)$ \cite[Lemma 10]{jog2016analysis}. See \cite[Lemma 7.8]{banerjee2020persistence} for \bbb{additional} checkable conditions for \eqref{eqn:prop-under-lamb} to hold.}
 \end{rem}

\section{Main Results}
\label{sec:struct-res}
\subsection{Convergence rates for model without change point} \label{subsec:conv_rates_no_cp}

Consider a continuous time branching process with attachment function $f$ and Malthusian rate $\lambda^*$. For $k \ge 0, t \ge 0$, \bbb{let} $D(k,t)$ \bbb{denote} the number of vertices in $\BP_f(t)$ \bbb{with} degree $k$ and abbreviate $Z_f(t)$ to $Z(t)$. Let $\lambda^* = \lambda^*(f)$ be as in \eqref{eqn:malthus-def}. % Recall the initializer function $f_0$ in the definition of the model in Section \ref{sec:model}.
Define the probability mass function $\vp(f):=\set{p_k:k\geq 0}$ via,
 \vspace{-.1in}
\chnr{
\begin{equation}
\label{eqn:pk-zero-def}
    p_k=p_k(f) := \int_0^{\infty} \lambda^* e^{-\lambda^* t} \mathbb{P}\left(\xi_{f}(t)=k\right)dt = \frac{\lambda^*}{\lambda^* + f(k)} \prod_{j=0}^{k-1}\frac{f(j)}{\lambda^* + f(j)}, \qquad k\geq 0. 
\end{equation}}
\chnr{For $k=0$, $\prod_{j=0}^{k-1}$ is taken to be $1$.} \chsb{The last equality above follows from standard calculations involving exponential distributions (see, for example, the proof of Theorem 2 (a) in \cite{rudas2007random}).}
%Verification that $ p_k$ is an honest probability mass function can be found in \cite[Theorem 2]{rudas2007random}. 
Following the seminal work of \cite{jagers-ctbp-book,jagers1984growth,nerman1981convergence,rudas2007random} for each $k \ge 0$, $D(k,t)/Z(t) \to p_k $ in probability  as   $t \to \infty$. % \frac{D(k,t)}{Z(t)} \probc p_k \ \ \  \text{as }  \ t \rightarrow \infty.
However to get consistent change point estimators we need to strengthen this convergence to a sup-norm convergence on a time interval whose size goes to infinity with growing $t$ \chsb{as well as obtain} a quantitative rate for this convergence. Such results have been obtained for very specific attachment functions via functional central limit theorems but do not extend to the \bbb{setting of general attachment functions}; \bbb{see e.g.} \cite{janson2004functional}; % for results on models whose degree evolution can be reduced to the evolution of urn processes satisfying regularity conditions.
specific to linear attachment see  \cite{Bollobas:2001:DSS:379831.379835,van2009random,resnick2015asymptotic}. The following assumptions \bbb{on the attachment function} will play a crucial role in this section.
 \begin{ass}\label{kickass}
 There exists $C^* \ge 0$ such that
 $
 \lim_{k \rightarrow \infty} {f(k)}/{k} = C^*.
 $
 \end{ass}
\begin{ass}\label{varass}
 $\operatorname{Var}\left(\int_0^{\infty}e^{-\lambda^*t}\xi_{f}(dt)\right) < \infty$. 
 \end{ass}
 
 \vspace{-.4in}

 \bbb{
 \begin{rem}\label{assimp}
	 Assumption \ref{varass} might at first sight seem opaque. Here we give three conceptually easier sufficient conditions that cover a wide array of functions. Throughout we assume Assumption \ref{ass:attach-func}. Assumption \ref{varass} holds if any of the following three conditions hold.  
	 \begin{enumeratea}
	 	\item {\bf Diverging attachment functions:}  $f(k) \rightarrow \infty$ as $k \rightarrow \infty$. See \cite[Proof of Lemma 1]{rudas2007random}.
		\item {\bf Finite variance of the degree distribution:} $\sum_{k=0}^{\infty}k^2 p_k(f) < \infty$. To see this note
\begin{multline*}
\E\left[\left(\int_0^{\infty} e^{-\lambda^* t} \xi_{f} (dt)\right)^2\right] = \E\left[\left(\int_0^{\infty} \lambda^* e^{-\lambda^* t} \xi_{f}(t)dt\right)^2\right] \le \E\left(\int_0^{\infty} \lambda^* e^{-\lambda^* t} \xi^2_{f}(t)dt\right)\\
= \int_0^{\infty} \lambda^* e^{-\lambda^* t} \sum_{k=1}^{\infty}k^2\mathbb{P}\left(\xi_{f}(t)=k\right)dt =  \sum_{k=1}^{\infty}k^2\left(\int_0^{\infty} \lambda^* e^{-\lambda^* t} \mathbb{P}\left(\xi_{f}(t)=k\right)dt\right) = \sum_{k=1}^{\infty} k^2 p_k(f) < \infty
\end{multline*}
where the last equality follows from \eqref{eqn:pk-zero-def}. \sba{For a given $f$, the finiteness of the above sum can possibly be checked using the explicit formula for $p_k(f)$ given in \eqref{eqn:pk-zero-def}.}
\item {\bf Lower boundedness and asymptotic linearity:} If $\inf_{k \ge 0}f(k)>0$ and \sba{$\lim_{k \rightarrow \infty} {f(k)}/{k} = C^* \ge 0$} (Assumption \ref{kickass}).  \ \\ This assertion was largely proven in \cite{banerjee2020root} which we now explain. In Lemma \ref{bbt0} we show that under these assumptions, $\lim_{k \rightarrow \infty}f(k)/k < \lambda^*$ and hence $\hat{\rho}\left(\lim_{k \rightarrow \infty}f(k)/k\right)>1$. Then \cite[Proposition 5.7]{banerjee2020root} shows that in this case $\int_0^{\infty}e^{-\lambda^*t}\xi_{f}(dt)$ has finite {\bf exponential} moments and thus, in particular, Assumption \ref{varass} on finiteness of the second moment holds. 
	 \end{enumeratea}
% \chsb{Moreover, if Assumption \ref{ass:attach-func} holds, $\inf_{i \ge 0}f(i)>0$ and $\hat{\rho}\left(\limsup_{k \rightarrow \infty}f(k)/k\right)>1$, then \cite[Proposition 5.7]{banerjee2020root} shows that $\int_0^{\infty}e^{-\lambda^*t}\xi_{f}(dt)$ has finite exponential moments and thus, in particular, Assumption \ref{varass} holds. It follows from Lemma \ref{bbt0} below that Assumptions \ref{ass:attach-func} and \ref{kickass} imply $\lim_{k \rightarrow \infty}f(k)/k < \lambda^*$ and hence $\hat{\rho}\left(\limsup_{k \rightarrow \infty}f(k)/k\right)>1$. Thus, by \cite[Proposition 5.7]{banerjee2020root}, \emph{$\inf_{i \ge 0}f(i)>0$ and Assumptions \ref{ass:attach-func} and \ref{kickass} imply Assumption \ref{varass}}.}
 \end{rem}
 }
 Fix a sequence of growing trees $\set{\cT_m: m\geq 2}$ and recall that for any $N\geq 2$ and $k\geq 0$, $D_N(k)$ denotes the number of vertices in $\cT_N$ with degree $k$. \chsb{The following theorem establishes convergence of the empirical degree distribution to its limit in a certain `uniform' sense and furnishes a rate for this convergence.}
 
 \begin{theorem}\label{ratewocp}
Consider a continuous time branching process with attachment function $f$ that satisfies Assumptions \ref{ass:attach-func}, \ref{kickass} and \ref{varass}. Let \bbb{$\vp(f)$} be the limiting degree distribution \bbb{as in \eqref{eqn:pk-zero-def}}. There exist $\omega^* \in (0,1),\epsilon^{**} \in (0,1)$, such that for any $\epsilon \le \epsilon^{**}$,
$$
n^{\omega^*}\sum_{k=0}^{\infty}2^{-k}\left(\sup_{t \in [0,2\epsilon \log n / \lambda^*]}\left| \frac{D\left(k,\frac{1-\epsilon}{\lambda^*}\log n + t\right)}{Z\left(\frac{1-\epsilon}{\lambda^*}\log n + t\right)} - p_k\right|\right) \overset{P}{\longrightarrow} 0.
$$
Thus for a sequence of non-uniform recursive trees $\set{\cT_m: m\geq 2}$ grown using attachment function $f$, 
%\[n^{\omega^*}\sum_{k=0}^{\infty}2^{-k}\left(\sup_{n^{1-\eps} \leq  N \leq n^{1+\eps}}\left| \frac{D_N(k)}{N} - p_k\right|\right) \overset{P}{\longrightarrow} 0.\]
\[n^{\omega^*}\sum_{k=0}^{\infty}2^{-k}\sup_{n^{1-\eps} \leq  N \leq n^{1+\eps}}\left| D_N(k)/N  - p_k\right| \overset{P}{\longrightarrow} 0.\]
\end{theorem}
% \chnr{\begin{rem}
%     % In the notation of \cite{jagers1984growth,nerman1981convergence}, the result above is stated for the ``characteristic'' corresponding to degree (see discussion below). We believe our proof techniques are robust enough to generalize to more complex functionals such as the fringe distribution \cite{aldous1991asymptotic,holmgren2017fringe}; this will pursued in a separate paper. Below we describe one of the key estimates derived in this paper of more general relevance.
% 	For special cases such as the uniform or linear preferential attachment, stronger results are obtainable via Janson's ``superball'' argument \cite{janson2004functional} as well as application of the Azuma-Hoeffding inequality \cite{Bollobas:2001:DSS:379831.379835,van2009random}. However these do not appear to work for the general model considered in this paper.
% \end{rem}}

% \begin{rem}
%     For special cases such as the uniform or linear preferential attachment, stronger results are obtainable via Janson's ``superball'' argument \cite{janson2004functional} as well as application of the Azuma-Hoeffding inequality \cite{Bollobas:2001:DSS:379831.379835,van2009random}. However these do not appear to work for the general model considered in this paper.
% \end{rem}

\sba{The analysis of branching processes in continuous time starts via scoring individuals existing at any fixed time $t$ via so called characteristics, measuring individuals (and their offspring) in various phases of their life, weighting existing individuals using these scores and then deriving asymptotics as $t\to\infty$. Such characteristics can in principle depend on the entire set of descendants (not just immediate offspring) of an individual, including ones that are born at future times. We refer the interested reader to \cite{jagers1984growth,jagers-nerman-2} for for further discussion on the importance of such characteristics and \cite{aldous1991asymptotic} for describing the importance of such results in the context of local weak convergence of large discrete random structures. An important technical contribution of this paper is the next result, Theorem \ref{l1conv}, regarding rates of convergence for normalized counts associated with general characteristics.}

We introduce some notation related to functionals of branching processes, closely following \cite{nerman1981convergence,jagers1984growth}. \sba{Recall that the individuals in the population are indexed by $\mathcal{I} = \cup_{d=0}^{\infty}\mathbb{N}^d$ and for $x \in \mathcal{I}$, $\sigma_x$ denotes the birth time of $x$.} Let $\{\xi_{f,x}, x \in \mathcal{I}\}$ be i.i.d. copies of the point process $\xi_f$ (see \eqref{eqn:xi-f-def}), where each $\xi_{f,x}$ is defined on some probability space $(\Omega_x, \mathcal{A}_x, \mathbb{P}_x)$. $\xi_{f,x}$ encodes the times of birth of children of $x$. The underlying probability space for the branching process (without a change point) is taken to be $(\Omega, \mathcal{A}, \mathbb{P}) = \Pi_{x \in \mathcal{I}}(\Omega_x, \mathcal{A}_x, \mathbb{P}_x)$. Elements of $\Omega$ are denoted by $\omega = \{\omega_x : x \in \mathcal{I}\}$. For each $x \in \mathcal{I}$, define the shift operator $S_x: \Omega \rightarrow \Omega$ which maps $\{\omega_y : y \in \mathcal{I}\}$ to $\{\omega_{xy} : y \in \mathcal{I}\}$. Thus, the shift operator $S_x$ maps $\emptyset$ and its descendants to $x$ and its descendants. 
% \sba{The analysis of branching processes in continuous time starts via scoring individuals existing at any fixed time $t$ via so called characteristics, measuring individuals (and their offspring) in various phases of their life, weighting existing individuals using these scores and then deriving asymptotics as $t\to\infty$. Such characteristics can in principle depend on the entire set of descendants (not just immediate offspring) of each individual;  we refer the interested reader to \cite{jagers1984growth,jagers-nerman-2} for further discussion on the importance of such characteristics and \cite{aldous1991asymptotic} for describing the importance of such results in the context of local weak convergence of large discrete random structures.}
 \rev{A characteristic $\phi: \mathbb{R} \times \Omega \rightarrow \mathbb{R}_+$ is a $\cB(\bR) \times \mathcal{A}$-measurable, separable, non-negative random process. 
% in the sense that the map $(a, \set{\omega_x:x\in \cI}) \to \phi(a, \set{\omega_x: x\in \cI})$ is $\cB(\bR) \times \mathcal{A}$-measurable. 
We assume $\phi(t, \omega) = 0$ for every $t<0, \omega \in \Omega$. Later in \eqref{charclass}, we will make further assumptions on the stochastic process $\set{\phi(t, \omega), t\in \bR}$. }
%\sba{The analysis of branching processes in continuous time starts via scoring individuals existing at any fixed time $t$ via so called characteristics, measuring individuals (and their offpsring) in various phases of their life, weighting existing individuals using these scores and then deriving asymptotics as $t\to\infty$. Such characteristics can in principle depend on the entire set of descendants (not just immediate offspring) of each individual;  we refer the interested reader to \cite{jagers1984growth,jagers-nerman-2} for the general analysis of such characteristics. In this paper, we will be mainly interested in characteristics that are functionals of the {\bf direct} offspring of each individual so we describe this specific setting next. Recall the offspring point process $\xi$ defined on $(\Omega, \cA)$ and abusing notation, let $\vA:= \set{\cA_t:t\geq 0}$ be the natural filtration of this process.  A characteristic $\phi: \mathbb{R_+} \times \Omega \rightarrow \mathbb{R}$ is a $\vA$-adapted, separable, non-negative random process. For notational convenience later we extend the domain to $\bR$ by setting $\phi(t, \omega) = 0$ for every $t<0, \omega \in \Omega$.}  

\rev{Informally, for each $t \ge 0$, $\phi(t)$ can be thought of as a `score' assigned to the root at time $t$, namely when the root is of {\bf age} $t$.} For each $x \in \mathcal{I}$, the characteristic corresponding to $x$, naturally obtained from $\phi$, is defined by $\phi_x(t,\omega) := \phi(t, S_x(\omega)), t \ge 0$. Thus, $\phi_x(t)$ can be thought of as the score given to $x$ based on $x$ and its descendants when $x$ is of age $t$.
%We write $\phi_x(t,\omega) = \phi(t, \omega_x)$ which can be thought of as a `score' given to $x$ based on its progeny when $x$ is of age t. 
We suppress the dependence of $\phi, \phi_x$ on $\omega$ and write $\phi(t), \phi_x(t)$ for $\phi(t,\omega)$ and $\phi_x(t,\omega)$ respectively.

For any characteristic $\phi$, define \chnr{$Z^{\phi}_f(t) := \sum_{x \in \mathcal{I}}\phi_x(t - \sigma_x) = \sum_{x \in \BP_f(t)}\phi_x(t - \sigma_x)$}. This can be thought of as the sum of $\phi$-scores, \chnr{or aggregate $\phi$-score}, of all individuals in $\BP_f(t)$. \rev{In particular, the age of individual $x$ in $\BP_f(t)$ is $t-\sigma_x$, and hence its contribution to the aggregate $\phi$-score is $\phi_x(t - \sigma_x)$.} Write $m^{\phi}_f(t) = \E(Z^{\phi}_f(t))$ and $M^{\phi}_f(t) = \E(e^{-\lambda^*t}Z^{\phi}_f(t))$. Note the characteristics $\phi(t) = \ind\set{t \ge 0}$ and $\phi(t) = \ind\set{\xi(t) =k}, k\ge0,$ count the total number of vertices and number of vertices of degree $k$ at time $t$ respectively. For these two specific characteristics we write \chnr{the associated scores as} $Z_{f}(t)$ and $Z_{f}^{\sss(k)}(t)$ respectively; analogously we write $m_f(t), m^{\sss(k)}_f(t)$ and $M_f(t), M^{\sss(k)}_f(t)$. 
% $\phi(t) = \ind\set{t \ge 0}$, we write $Z_{f}(t)$ which counts the total number of vertices at time $t$ and similarly, $M_f(t) =  \E(e^{-\lambda^*t}Z_f(t))$. For fixed $k\geq 0$ and for the characteristic $\phi(t) = \ind\set{\xi(t) =k}$, write $m^{\sss(k)}_f(\cdot):=m^{\phi}_f(\cdot)$.
It is easy to check that for a general (integrable) characteristic $\phi$, $M^{\phi}_f(t)$ satisfies the renewal equation
\begin{equation}\label{renmean}
M^{\phi}_f(t) = e^{-\lambda^* t}\E(\phi(t)) + \int_0^t M^{\phi}_f(t-s) e^{-\lambda^* s}\mu_f(ds).
\end{equation}
Write $M^{\phi}_f(\infty) = \lim_{t \rightarrow \infty} M^{\phi}_f(t)$ when the limit exists.
Following \cite{nerman1981convergence}, for $t\ge 0$, let $\mathcal{I}(t) = \{x= (x', i): \sigma_{x'} \le t \text{ and } t < \sigma_x < \infty\}$ \chsb{denote the set of individuals born after time $t$ to parents who were born at or before time $t$}.
%\begin{align*}
%\mathcal{I}(t) = \{x= (x', i): \sigma_{x'} \le t \text{ and } t < \sigma_x < \infty\}.
%%\mathcal{I}(t,c) = \{x= (x', i): \sigma_{x'} \le t \text{ and } t  + c < \sigma_x < \infty\}.
%\end{align*}
Write $W_t := \sum_{x \in \mathcal{I}(t)}e^{-\lambda^* \sigma_x}$. By Corollary 2.5 of \cite{nerman1981convergence}, $W_t$ converges almost surely to a finite random variable $W_{\infty}$ as $t \rightarrow \infty$. By Theorem 3.1 of \cite{nerman1981convergence}, $e^{-\lambda^* t} Z^{\phi}_f(t) \convp W_{\infty}M^{\phi}_f(\infty)$ for any $\phi \in \mathcal{C}$.

For this article, we are interested in the following class of characteristics \rev{(where once again recall $\emptyset$ denotes the root of the tree)}:
\begin{equation}\label{charclass}
\mathcal{C} := \{\phi \text{ with c\`adl\`ag paths }: \exists \chr{\text{ a non-random}} \ b_{\phi}>0 \text{ such that } \phi(t) \le b_{\phi}(\xi_{f, \rev{\emptyset}}(t) + 1) \text{ for all } t \ge 0\}.
\end{equation}

\begin{theorem}\label{l1conv}
Consider a continuous time branching process with attachment function \chsb{$f$ that satisfies Assumptions \ref{ass:attach-func} and \ref{varass}}. There exist positive constants $C_1, C_2$ such that for any $b_{\phi}> 1$ and any characteristic $\phi \in \mathcal{C}$ satisfying $\phi(t) \le b_{\phi}(\xi_{\rev{f,\emptyset}}(t) + 1)$ for all $t \ge 0$,
$$
\E\left|e^{-\lambda^* t} Z^{\phi}_f(t) - W_{\infty}M^{\phi}_f(\infty)\right| \le C_1b_{\phi} e^{-C_2 t}, \ t \ge 0.
$$
\end{theorem}

\begin{rem}
The constants $\omega^*$ in Theorem \ref{ratewocp} and $C_1, C_2$ in Theorem \ref{l1conv} are explicitly computable from our proof techniques. However, they depend on the Malthusian rate \chsb{and $\underline{\lambda}$ (see \eqref{eqn:prop-under-lamb})} and thus we have not tried to derive an explicit form of these objects. 
\end{rem}
% As the results on convergence rates described above are directly used to prove consistency of the change point estimator and the proofs require results from Section \ref{sec:snc} below (which are proved independently of the above results), we will postpone the proofs of the above theorems to Section \ref{prwocp}.
\subsection{Sup-norm convergence of degree distribution for the standard model}
\label{sec:snc}
 We start by studying the model under the following assumption which we refer to as the ``standard'' model owing to the analogous assumptions for change point methodology in time series:
\vspace{-.1in}
\begin{ass}
    \label{ass:stand}
\chnr{There exist $0 < \gamma < 1$ such that the change point is $\tau = \lfloor n\gamma \rfloor$. }  
\end{ass}
\vspace{-.1in}
To simplify notation we will drop $\lfloor~ \rfloor$. Recall the sequence of random trees $\set{\cT_m^{\mvtheta}: 2\leq m\leq n}$. \chsb{For any $0 < t\leq 1$ and $k\geq 0$, write $D_n(k,\cT_{nt}^{\mvtheta})$ for the number of vertices with degree $k$ when the tree is of size $nt$.} 
%We sometimes abuse notation and write $D_n(k,\cT_{nt}):=D_n(k,nt)$ to explicitly specify the dependence of this object on the underlying tree.} % In this section we consider the case where there is exactly one change point at time $n\gamma_1$ for fixed $0< \gamma_1 <1$. In Section \ref{sec:mcp-res} we describe the general result for multiple change points.  The notation is cumbersome so this general case  can be skipped over on an initial reading. We also give the proof for the single change point case; the general case follows via straight-forward extensions.
\chsb{Fix initializer attachment function $f_0$ and let $\lambda_0^* = \lambda^*(f_0)$ be as in \eqref{eqn:malthus-def}.  Define the probability mass function $\set{p_k^{\sss 0 }:k\geq 0}$ via \eqref{eqn:pk-zero-def} with $(\lambda_0^*, f_0)$ in place of $(\lambda^*, f)$. As before \bbb{write $f_1$ for} the attachment function after change point.}

%\chsb{The continuous time embedding of the tree model without change point has a natural extension to the model with change point. Individuals in the population reproduce according to independent copies of point process $\xi_{f_0}$ up till the change point. After the change point, individuals continue to reproduce independently in the following way. An individual of degree $\ell$ at the change point reproduces according to the point process $\xi_{f_1}^{(\ell)}$. New individuals born into the system after the change point reproduce according to $\xi_{f_1}$.}
\chsb{Recall the continuous time embedding of $\set{\cT_m^{\mvtheta}: 2\leq m\leq n}$ into an inhomogeneous branching process $\BP_{\mvtheta}(\cdot)$ as described in Section \ref{cpembedding}. At the change point of $\BP_{\mvtheta}(\cdot)$, different individuals have different degrees, and their} \chiain{offspring process} \chsb{after the change point need to be quantified in terms of their degree at the change point. We now introduce some key quantities required in this quantification. Recall $m_{f_1}(\cdot)$ from \eqref{eqn:mft-def}. For fixed $k\geq 0$, recall the functions  $\mu_{f_1}^{(k)}[0,\cdot]$ from \eqref{eqn:xi-k-f-t-def} and define, for $t \ge 0$, %progeny distributions  ->  \chiain{offspring process} 
 $
 m_ {f_1}^{(k)}(t) := \E\left(\sum_{x \in \BP_{f_1}(t)}\ind\set{\xi_{f_1,x}(t-\sigma_x) = k}\right),
 $
 which denotes the expected number of individuals with $k$ children in $\BP_{f_1}(t)$.}
 It can be checked \chnr{(using the continuity estimates obtained in Lemmas \ref{lem:1} and \ref{lipcont})} that for any $k \ge 0$, $t \ge 0$,
 $
 m_ {f_1}^{(k)}(t) = \int_0^t\pr\left(\xi_{f_1}(u) = k\right)m_{f_1}(t-du).
 $
 
%\chsb{At the change point, different individuals have different degrees, and their progeny distributions after the change point needs to be quantified in terms of their degree at the change point. Towards this end, the following objects will play a crucial role.}
 For $\ell, k\ge0$, define
 \begin{align}\label{lambdadef}
\lambda_{\ell}(t) = 1 + \int_0^{t}m_{f_1}(t-s) \mu^{(\ell)}_{f_1}(ds), \ \ \ 
\lambda^{(k)}_{\ell}(t) = \pr\left(\xi^{(\ell)}_{f_1}(t)=k-\ell\right) + \int_0^{t}m^{(k)}_{f_1}(t-s) \mu^{(\ell)}_{f_1}(ds).
\end{align}
\chsb{Given that an individual is of degree $\ell$ at the change point, $\lambda_{\ell}(t)$ (respectively, $\lambda^{(k)}_{\ell}(t)$) denotes the expected number of descendants (respectively, the expected number of descendants having degree $k$), including possibly itself, $t$ time units after the change point.}
Let $\mathcal{P}$ denote the collection of all probability measures on $\bZ_+$. For each $a>0$, consider the functional
$
\Phi_a : \mathcal{P} \rightarrow \mathcal{P}
$
given by
\vspace{-.1in}
\begin{equation}\label{phidef}
\Phi_a(\mathbf{p}) = \left(\sum_{\ell = 0}^{\infty} p_{\ell}\lambda_{\ell}^{(k)}(a)\big/\sum_{\ell = 0}^{\infty} p_{\ell}\lambda_{\ell}(a)\right)_{k \ge 0}
\end{equation}
where $\mathbf{p}=(p_0, p_1, \dots) \in \mathcal{P}$. Write $(\Phi_a(\mathbf{p}))_k$ for the $k$-th co-ordinate of the above map. 
Let $\mathbf{p}^{i} = \vp(f_i):= (p^{i}_0, p^{i}_1, \dots)$ for $i=0,1$ denote the limiting degree \chnr{distribution for} a \chsb{non-uniform} random recursive tree grown with attachment function $f_i$ (i.e. without any change point). 
\chiain{Informally, $\Phi_a(\mathbf{p}^{0})$ shows how the degree distribution in the continuous time embedding evolves in $a$ units of time after the change point.}
Corollary \ref{cor:timing} shows that for each $t > \gamma$, there is a unique $0< a_t <\infty$ such that
\vspace{-.1in}
\begin{equation}
\label{eqn:alpha-def}
    \sum_{k=0}^\infty p_k^{\sss 0}\left[ \int_0^{a_t} m_{f_1}(a_t -s) \mu_{f_1}^{(k)}(ds)\right] =(t-\gamma)/\gamma. 
\end{equation}
\bbb{Recall the continuous time embedding of $\set{\cT^{\mvtheta}_m:2\leq m\leq n}$ in $\BP_{\mvtheta}$ described in Section \ref{cpembedding}. {Conceptually here}, for $t> \gamma$, $a_t$ denotes (in the large $n$ limit) the time required \emph{in the continuous time embedding} for the process starting at $\cT_{n\gamma}$ (i.e. at the change point) to reach size $nt$.  Set $a_{t}=0$ for $t \le \gamma$.} 

Suppose $f_0, f_1$ satisfy Assumption \ref{ass:attach-func}. 
\chsb{The following theorem shows that the empirical degree distribution of the (discrete) standard model can be approximated uniformly on compact time intervals after the change point by a deterministic curve, obtained using the continuous time embedding.}
 \begin{theorem}\label{supthm}
 For \bbb{each fixed} $k \ge 0$ and $s \in [\gamma, 1]$
 $,
%\sup_{t \in [\gamma, s]} \left| \frac{D_n(k, nt)}{nt} -  (\Phi_{a_t}(\mathbf{p^0}))_k\right| \probc 0.
\sup_{t \in [\gamma, s]} \left| D_n(k,\cT_{nt}^{\mvtheta})/nt -  (\Phi_{a_t}(\mathbf{p^0}))_k\right| \probc 0.
 $
 %and
% $$\pr\left( \sup_{t \in [0, T_{sn}]} \left|Z_n(t) - n \sum_{\ell=0}^{\infty}\gamma p_{\ell}^{0}\lambda_{\ell}(t) \right| > \epsilon n  \right) \to 0.$$
 \end{theorem}
There is a probabilistic way to view the limit.  Write $\alpha$ for $a_1$. 
% Construct an integer valued random variable $D_{\mvtheta}$ using the following auxiliary random variables:
\begin{constr}[$X_{\bc}$]
    \label{constr:bc}
   \chnr{ Generate $D\sim \set{p_k^{\sss 0}:k\geq 0}$.} Conditional on $D =k$, generate point process $\xi_{f_1}^{(k)}$ and let $\sC = \xi_{f_1}^{(k)}[0,\alpha]$. Now set $X_{\bc} = D + \sC$.
\end{constr}
\begin{constr}[$X_{\ac},\; \age$]
    \label{constr:ac}
    \begin{enumeratea}
        \item  Generate $D\sim \set{p_k^{\sss 0}:k\geq 0}$. 
        Conditional on $D=k$, generate $\age$ supported on the interval $[0,\alpha]$ with distribution 
$$
   \pr(\age > u):= \int_0^{\alpha -u} m_{f_1}(\alpha -u -s) d\mu_{f_1}^{(k)}(ds) \Big/ \int_0^\alpha m_{f_1}(\alpha -s) \mu_{f_1}^{(k)}(ds), \qquad 0\leq u\leq \alpha.
$$
        \item Conditional on $D$ and $\age$, let $X_{\ac} = \xi_{f_1}[0,\age]$, with $\xi_{f_1}$ as in \eqref{eqn:xi-f-t}. % $\xi_{f_1}$ is the point process constructed using attachment function $f_1$.
    \end{enumeratea}
\end{constr}
\chr{Conceptually in the above notation, `BC' stands for `before change' and `AC' stands for `after change'.} \chnr{Thus \bbb{(in the large $n$ limit)}, $X_{BC}$ denotes the final degree (when the tree is of size $n$) of an individual which had degree $D$ at the change point. $X_{AC}$ denotes the final degree of an individual born $\alpha - \age$ time units after the change point.} Now, \chsb{slightly abusing notation}, let $\mvtheta = (f_0,f_1, \gamma)$. Let $D_{\mvtheta}$ be the integer valued random variable defined as follows: with probability $\gamma$,  $D_{\mvtheta} = X_{\bc}$ and with probability $1-\gamma$,  $D_{\mvtheta} = X_{\ac}$. The following is a restatement of the convergence result implied by Theorem \ref{supthm} for time $t=1$.
\begin{theorem}[\chnr{Standard model}]
    \label{thm:standard}
\chnr{Fix $k\geq 0$} and let $D_n(k)$ denote the number of vertices with degree $k$ in the tree $\cT_n^{\mvtheta}$.  Under Assumption \ref{ass:attach-func} on the attachment functions $f_0,f_1$ and Assumption \ref{ass:stand} on the change point $\gamma$, we have that 
$
% \frac{D_n(k)}{n} \convp \pr(D_{\mvtheta} = k).
D_n(k)/n \convp \pr(D_{\mvtheta} = k).
$
\end{theorem}
Write $\vp(\mvtheta)$ for the pmf of $D_{\mvtheta}$. The next result, albeit intuitively reasonable, is non-trivial to prove in the generality of the models considered in the paper. 
\begin{corollary}\label{cor:cp-diff}
    Assume that $\vp^0\neq \vp^1$. Then for any $0< \gamma <1$ one has $\vp^0\neq \vp(\mvtheta)$. Thus the change point always changes the degree \chsb{distribution}. 
\end{corollary}
 \vspace{-.1in}
% The next result describes the tail behavior of the ensuing random variable.
 % \vspace{-.1in}
\chsb{For the following corollary, we say that a random variable $X$ has an exponential tail if there exist positive constants $C_1, C_2$ such that $\mathbb{P}(X > x) \le C_1 \exp\{-C_2x\}$ for all $x \ge 0$. We say $X$ has a power law tail with exponent $\kappa>0$ if there exist positive constants $C_1, C_2$ such that  $C_1x^{-\kappa} \le \mathbb{P}(X > x) \le C_2x^{-\kappa}$ for all $x \ge 1$.}
\begin{corollary}[Initializer wins \chnr{under the standard model}]\label{cor:int-wins}
    The initializer function $f_0$ determines the tail behavior of $D_{\mvtheta}$ in the sense that 
    \begin{enumeratei}
        \item If in the model without change point using $f_0$, the degree distribution has an exponential tail then so does the model with change point irrespective of $\gamma > 0$ and \chnr{$f_1$}. 
        \item If in the model without change point using $f_0$, the degree distribution has a power law tail with exponent $\kappa >0$ then so does model with change point irrespective of $\gamma > 0$ and \chnr{$f_1$}. 
    \end{enumeratei}
\end{corollary}
\begin{corollary}[Maximum degree \chnr{under the standard model}] \label{cor:max-degree} Suppose the initializer is linear with  $f_0(i) = i+1+\alpha$ for $i\geq 0$. \chsb{For $k \ge 1$}, let $M_n(k)$ be the size of the $k$-th maximal degree. If \chnr{$f_1$} satisfies Assumption \ref{ass:attach-func} then $M_n(k)/n^{1/(\alpha+2)}$ is a tight collection of random variables bounded away from zero as $n\to\infty$. 
\end{corollary}
\begin{rem}
\chnr{
Corollary \ref{cor:max-degree} shows the initializer determines the behavior of the maximal degree in the case of a linear initializer under the standard model.
}
\chsb{In the absence of a change point, for each fixed $k \ge 1$, $M_n(k)/n^{1/(\alpha+2)}\convd X_k(\alpha)$ for a non-degenerate random variable $X_k(\alpha)$} \bbb{with $\pr(X_k(\alpha) > 0)=1$} \cite{mori2007degree}. Thus the above result shows that irrespective of the second attachment function $f_1$, the maximal degree asymptotics for linear preferential attachment remain unaffected. Proof of the above result follows via analogous arguments as \cite[Proof of Theorem 2.2]{bhamidi2015change} and thus is not provided in this paper. 
%Without change point it is known \cite{mori2007degree} that for each fixed $k$, $M_n(k)/n^{1/(\alpha+2)}\convd X_k(\alpha)$ for a non-degenerate distribution. Thus the above result shows that irrespective of the second attachment function $f_0$, the maximal degree asymptotics for linear preferential attachment remain unaffected. Proof of the above result follows via analogous arguments as \cite[Theorem 2.2]{bhamidi2015change} and thus is not provided in this paper. 
\end{rem}

\subsection{The quick big bang model}
\label{sec:big-bang-res}
%\chnr{Now consider the case where the change point happens ``early'' in the evolution of the process where it scales like $o(n)$.} 
%Now  consider the case where the change point happens ``early'' in the evolution of the process where the change point scales like $o(n)$.
\bbb{Now consider the case where the change point scales like $o(n)$ i.e. happens ``early'' in the evolution of the process.  We call this version of the process ``quick big bang''  to fix the idea that the change happens way back in the origin of the process (akin to the ``big bang''), but despite this change close to the origin of the process (relative to the entire time scale), the effect of this can be felt and observed all the way till the present via carefully chosen functionals. }
Let $\set{p_k^1:k\geq 0}$ be the probability mass function as in \eqref{eqn:pk-zero-def}, but \chsb{using the function $f_1$ in place of $f$ to obtain $\lambda^*$ in \eqref{eqn:malthus-def} and in \eqref{eqn:pk-zero-def}.} 
For $\alpha>0$ and any non-negative measure $\mu$, let $\hat{\mu}(\alpha):=\int_0^{\infty} \alpha e^{-\alpha t}\mu(t)dt.$ 
We work under the following assumption.
\begin{ass}\label{xlogx}
$\E\left(\hat{\xi}_{f}(\lambda^*) \left|\log\left(\hat{\xi}_{f}(\lambda^*)\right) \right|\right) < \infty$.
\end{ass} 
\begin{rem}
Assumption \ref{xlogx} is, in some sense, the `minimal assumption' required to ensure non-degeneracy of the random variable $W_{\infty} := \limsup_{t \rightarrow \infty}e^{-\lambda^*t}Z_f(t)$ \cite[Proposition 1.1]{nerman1981convergence}. \sba{In particular, $W_{\infty} >0$ almost surely if Assumption \ref{xlogx} is satisfied and $W_{\infty} = 0$ almost surely if Assumption \ref{xlogx} fails.}
\end{rem}

Recall that in the previous section, one of the messages was that the initializer function $f_0$ determined various macroscopic properties of the degree distribution for the standard model. 

\begin{theorem}[\chnr{Initializer loses under the quick big bang}]\label{thm:qbbm}
    Suppose $\tau_1 = n^\gamma$ for fixed $0 < \gamma <1$. 
     \chsb{If $f_0$ satisfies Assumption \ref{ass:attach-func} and $f_1$ satisfies Assumptions \ref{ass:attach-func}, \ref{kickass} and \ref{xlogx}}, the \chsb{limiting} degree distribution {\bf does not} feel the effect of the change point or the initializer attachment function $f_0$ in the sense that for any fixed $k\geq 0$, 
$
D_n(k) / n \convp p_k^1 \ \ 
$
as $n\to\infty$.
\end{theorem}
\begin{rem}
    The form $\tau_1:= n^\gamma$ was assumed for simplicity. We believe the proof techniques are robust enough to handle any $\tau_1 = \omega_n$, where $\omega_n = o(n)$ and $\omega_n\uparrow \infty$. We defer this to future work.  
\end{rem}

 The next result implies the maximal degree \textit{does} feel the effect of the change point. Instead of proving a general result we consider the following special cases. Let $M_n(1)$ denote the maximal degree in $\cT_n^{\mvtheta}$. \sba{Fix two deterministic positive sequences $\set{a_n}_{n\geq 1},\set{b_n}_{n\geq 1} $ with $a_n,b_n \uparrow \infty$ and $a_n/b_n\to 0$ as $n\to\infty$. For a sequence of non-negative random variables $\set{M_n}_{n\geq 1}$, say that $a_n \ll M_n \ll b_n$ with high probability as $n\to\infty$ if $M_n/a_n\probc \infty$ and $M_n/b_n\probc 0$ as $n\to\infty$.  }
 
 \begin{theorem}[Maximal degree under quick big bang] \label{thm:max-deg-quick-big-bang} Assume $\tau_1 = n^\gamma$ and consider:
    \begin{enumeratea}
        \item {\bf Uniform $\leadsto$ Linear:} Suppose $f_0(\cdot) \equiv 1$ whilst $f_1(k) = k+1+\alpha$ for fixed $\alpha > 0$. Then \chsb{for any sequence $\omega_n\uparrow \infty$, } with high probability as $n\to\infty$,
        $n^{\frac{1-\gamma}{2+\alpha}} \log{n} \big/ \omega_n \ll M_n(1)\ll n^{\frac{1-\gamma}{2+\alpha}} (\log{n})^2.$ 
        \item {\bf Linear $\leadsto$ Uniform:} Suppose $f_0(k) = k+1+\alpha$ for fixed $\alpha >0$ whilst $f_1(\cdot) \equiv 1$. Then \chsb{for any sequence $\omega_n\uparrow \infty$,} with high probability as $n\to\infty$, 
        $n^{\frac{\gamma}{2+\alpha}} \log{n} \big/ \omega_n \ll M_n(1)\ll n^{\frac{\gamma}{2+\alpha}} (\log{n})^2. $
        \item {\bf Linear $\leadsto$ Linear:} Suppose $f_0(k) = k+1+\alpha $ whilst $f_1(k) = k+1+\beta$ where $\alpha \neq \beta$. Then $M_n(1)/n^{\eta(\alpha,\beta)}$ is tight and bounded away from zero where
$$
\eta(\alpha,\beta):= \left( \gamma(2+\beta) + (1-\gamma)(2+\alpha) \right) \big / (2+\alpha)(2+\beta). 
$$
\end{enumeratea}
 \end{theorem}
 \begin{rem}
 
\chnr{ \chsb{Writing $\tilde{M_n}:=M_n(1)/n^{\eta(\alpha,\beta)}$,  in (c) by bounded away from zero we mean $\set{1/\tilde{M_n}:n\geq 1}$ is tight. 
This result shows that while the} initializer does not affect the limiting degree distribution in the quick big bang model (Theorem \ref{thm:qbbm}), it can influence the maximal degree. %While limiting degree distribution in the big bang case does not feel the initializer function (Theorem \ref{thm:qbbm}), the maximal degree degree can be influenced by the initializer.
It is instructive to compare the above results to the setting without change point.
For the uniform $f\equiv 1$ model, it is known \cite{devroye:1995,szymanski1990maximum} that the maximal degree scales like $\log n$ whilst for the linear preferential attachment, the maximal degree scales like $n^{1/(\alpha+2)}$ \cite{mori2007degree}.
\chsb{Thus, for example, in the `Linear $\leadsto$ Uniform' case, Theorem \ref{thm:qbbm} implies that the limiting degree distribution in this case is the same as that of the \textit{uniform random recursive tree} (URRT) namely Geometric with parameter $1/2$; however Theorem \ref{thm:max-deg-quick-big-bang} (b) implies that the maximal degree scales polynomially in $n$ and {\bf not} like $\log{n}$ as in the URRT.} 
}
%It is instructive to compare the above results to the setting without change point. For the uniform $f\equiv 1$ model, it is known \cite{devroye:1995,szymanski1990maximum} that the maximal degree scales like $\log_2(n)$ whilst for the linear preferential attachment, the maximal degree scales like $n^{1/(\alpha+2)}$ \cite{mori2007degree}. Thus for example, (b) of the above result coupled with Theorem \ref{thm:qbbm} implies that the limiting degree distribution in this case is the same as that of the \chnr{\textit{uniform random recursive tree}} (URRT) namely Geometric with \chnr{parameter} $1/2$; however the maximal degree scales polynomially in $n$ and {\bf not} like $\log{n}$ as in the URRT. 
 \end{rem}
% \begin{rem}
%     \label{rem:bub}
%     For any $\tau_1 \to\infty$, the initial segment should always leave its signature in some functional of the process. See for example \cite{bubeck2015influence,bubeck2017finding,curien2014scaling} where the evolution of the system (using typically linear preferential attachment albeit \cite{bubeck2017finding} also considered the uniform attachment case) starting from a fixed ``seed'' tree was considered and the aim was to detect (upto some level of accuracy) this seed tree after observing the tree $\cT_n$. Similar heuristics suggest that in the context of our model, the initial segment of the process however small should show its signature at some level. We discuss this aspect further in Section \ref{sec:disc}.
% \end{rem}

% Proofs of results for the quick big bang model are given in Section \ref{sec:proofs-qbb}.

\subsection{Change point detection}
\label{sec:res-cp}
\bbb{In the context of the standard model, now consider} the \chnr{issue of change point detection} from an observation of the network.   % Our proof techniques do not directly provide estimation techniques for multiple change points and we defer this to future work.
Consider any two sequences $h_n \rightarrow \infty, b_n \rightarrow \infty$ satisfying \chnr{${\log h_n}/{\log n} \rightarrow 0$, ${\log b_n}/{\log n} \rightarrow 0$} as $n \rightarrow \infty$. Define:
$$
\hat{T}_n = \inf \left\lbrace t \ge \frac{1}{h_n}: \sum_{k = 0}^{\infty}2^{-k}\left|\frac{D_n(k, \cT_{\lfloor nt \rfloor}^{\mvtheta})}{nt} - \frac{D_n(k, \cT_{\lfloor n/ h_n \rfloor}^{\mvtheta})}{n/ h_n}\right| > \frac{1}{b_n}\right\rbrace.
$$
The following theorem establishes the consistency of the above estimator.
\begin{theorem}\label{cpconv}
Assume that $\vp^0\neq \vp^1$. Suppose \chsb{$f_0$ satisfies Assumptions \ref{ass:attach-func}, \ref{kickass} and \ref{varass}, and $f_1$ satisfies Assumptions \ref{ass:attach-func} and \ref{xlogx}}. Then $\hat{T}_n \probc \gamma$.
\end{theorem}

\section{Discussion}
\label{sec:disc}
(i) {\bf Random recursive trees:} \chnr{Random recursive trees have now been studied for decades \cite{mahmoud2008polya,drmota2009random,smythe1995survey,devroye1998branching,goldschmidt2005random}. }% , motivated by a wide array of fields including convex hull algorithms, linguistics, epidemiology and first passage percolation 
For specific examples such as the uniform attachment or the linear attachment model with $f(i):= i+1$, one can use the seminal work of Janson \cite{janson2004functional} via a so-called ``super ball'' argument to obtain functional central limit theorems for the degree distribution. Obtaining quantitative error bounds let alone weak convergence results in the general setting considered in this paper is much more \chsb{involved}. Regarding proof techniques, we proceed via embedding the discrete time models into continuous time branching processes and then using martingale/renewal theory arguments for the corresponding continuous time objects; this approach goes back all the way to \cite{athreya1968}. Limit results for the corresponding CTBPs in the setting of interest for this paper were developed in  \cite{jagers-ctbp-book,jagers1984growth,nerman1981convergence}. One contribution of this work is to derive quantitative versions for this convergence, a topic less explored, but required to answer questions regarding statistical estimation of the change point. \chsb{In the context of growing random trees (without change point) with either uniform or linear attachment functions, understanding the effect of the initial seed graph and in particular constructing algorithms to estimate the root (the so-called ``Adam problem'') has inspired intense activity over the last decade.  See for example \cite{bubeck2015influence,bubeck2017finding,curien2014scaling} for more details.}

(ii) {\bf General change point:} Change point detection, especially in the context of univariate time series,  has matured into a vast field, see \cite{csorgo1997limit,brodsky2013nonparametric}. Even in this context, consistent estimation especially in the setting of multiple change points is non-trivial and requires specific assumptions see e.g. \cite{yao1988estimating};  in the context of econometric time series see \cite{bai1997estimating,bai1998estimating,bai2003computation}; for applications in the biological sciences see \cite{olshen2004circular,zhang2007modified}. \chnr{The only pre-existing work on change point in the context of \chnr{growing} networks formulated in this paper that we are aware of was carried out in \cite{bhamidi2015change} where one assumed linear attachment functions.} \bbb{ Regarding the estimator proposed in this paper in Theorem \ref{cpconv}, we do not believe the estimator is ``optimal'' in terms of rates of convergence. \chsb{ These issues are deferred to future work}. }

\section{\bbb{Overview of the proofs}}
%The rest of this section is devoted to some preliminary estimates and constructions that will then be repeatedly used in the proofs. 
%Although the results about convergence rates for the model without change point are stated before the change point results, the proof of Theorem \ref{ratewocp} is quite technical and an essential ingredient is a ``sup-norm estimate" given in Lemma \ref{lem:Nt-sumNlam-close} which is proven more generally in the context of a change point. Thus, we defer the proof of Theorem \ref{ratewocp} to Section \ref{prwocp}.
%Some of the proofs are technical so in order to motivate the reader, we will start with the analysis of the change point model; this will motivate some the subsetting of intervals constructions used in the proof of Theorem \ref{ratewocp} but which might be hard to understand directly.  
\chsb{The rest of the paper proves the main results. Section \ref{sec:proofs-embed} lays out some preliminary constructions and estimates used subsequently in the paper.}
Section \ref{sec:proofs-sp} deals with the continuous time version of the change point model analyzed for a fixed time $a$ after the change point. Theorem \ref{thm:1} proved \bbb{in this section} estimates, for a general characteristic $\phi \in \cC$, the $L^1$-error in approximating the aggregate $\phi$-score at time $a$ of all individuals born after the change point with a weighted linear combination of the degree counts at the change point. This estimate, apart from directly yielding a law of large numbers (see second part of Theorem \ref{thm:1}), turns out to be crucial in most subsequent proofs. 

The estimates derived in Section \ref{sec:proofs-sp} are then used in Section \ref{supdoc} to analyze the standard model and prove the main theorems in this setting (Theorems \ref{supthm} and \ref{thm:standard}) as well as Corollary \ref{cor:int-wins}. Corollary \ref{cor:cp-diff} follows directly from Lemma \ref{diffop} and requires an in-depth analysis of the fluid limits derived in Theorem \ref{supthm} and is postponed to Section \ref{sec:proofs-cpd}. 

Section \ref{sec:proofs-qbb} contains proofs of the quick big bang model. We note here that all the estimates obtained in Sections \ref{sec:proofs-sp} and \ref{supdoc} to analyze the model for a fixed time $a$ after the change point explicitly exhibit the dependence on $a$. This turns out to be crucial in Section \ref{sec:proofs-qbb} where we take $a = \eta_0 \log n$ and the estimates above still hold if $\eta_0$ is sufficiently small. Roughly speaking, we partition the interval $[T_{n^{\gamma}}, T_{n}]$ into finitely many subintervals of size at most  $\eta_0 \log n$ and `bootstrap' the estimates \chsb{obtained in Sections \ref{sec:proofs-sp} and \ref{supdoc}} to prove Theorem \ref{thm:qbbm}.
%The estimates developed till this point then result in the proof of Theorem \ref{ratewocp} in Section \ref{prwocp}. 

\chnr{In Section \ref{prwocp}, we prove Theorems \ref{ratewocp} and \ref{l1conv}.} %using a ``sup-norm estimate" given in Lemma \ref{lem:Nt-sumNlam-close} which is proven more generally in the context of a change point.}
We conclude in Section \ref{sec:proofs-cpd} with the proof of Theorem \ref{cpconv} on the change point detection estimator.

\section{Initial constructions}
\label{sec:proofs-embed}

% In this section we will start by setting some initial constructions that in particular highlight the role of continuous time branching processes.
%Throughout the rest of the sequel, we will focus on the setting one change point, results for multiple change points follow via identical arguments (and extra notation). Section \ref{sec:proofs-sp} contains proofs of the structural properties of the standard model; Section \ref{sec:proofs-qbb} proves all results regarding quick big bang. We conclude in Section \ref{sec:proofs-cpd} with proofs for estimating change point(s). 
\chnr{This section is devoted to some preliminary \chsb{constructions and estimates} that will then be repeatedly used in the proofs.} \chsb{The first set of lemmas deal with properties of \emph{linear} preferential attachment and an important class of offspring processes associated to it.}

\begin{defn}[Rate $\nu$ Affine $\kappa$ PA model]
    \label{def:rate-c-pa}
Fix $\nu>0, \kappa\geq 0$. A branching process \chsb{with} attachment function $f(i) = \nu(i+1)+\kappa, i \ge 0,$ will be called a linear PA branching process with rate $\nu$ and affine parameter $\kappa$. Denote this as $\set{\PA_{\nu,\kappa}(t):t\geq 0}$. 
\end{defn}

\begin{defn}[Rate $\nu$ Yule process]
\label{def:yule-process}
\chsb{The offspring process $\xi_{\nu, 0}(t)$ associated with a $\PA_{\nu,0}(\cdot)$ process is called a rate $\nu$ Yule process. Thus, the rate of birth of new individuals in a Yule process is proportional to the size of the current population. We \bbb{write $\{Y_{\nu}(t) : t \ge 0\}$ for this process}.}
 %   Fix $\nu > 0$. A rate $\nu$ Yule process is a pure birth process $\set{Y_\nu(t):t\geq 0}$ with $Y_\nu(0)=1$ and where the rate of birth of new individuals is proportional to size of the current population. More precisely, $\pr(Y_\nu(t+dt) - Y_\nu(t) = 1|\cF(t)):= \nu Y_\nu(t) dt + o(dt),$ and $\pr(Y_\nu(t+dt) - Y_\nu(t) \ge 2|\cF(t)) = o(dt),$ where $\set{\cF(t):t\geq 0}$ is the natural filtration of the process. 
\end{defn}

% The following is a standard property of the Yule process, see e.g. .
\begin{lem}[{\cite[Section 2.5]{norris-mc-book}}]
\label{lem:yule-prop}
Fix $t >0$ and rate $\nu > 0$. Then $Y_\nu(t)$ has a Geometric distribution with parameter $p=e^{-\nu t}$. Precisely, $\text{ }\pr(Y_\nu(t) = k) = e^{-\nu t}(1-e^{-\nu t})^{k-1}, k\geq 1.$ The process $\set{Y_\nu(t)\exp(-\nu t):t\geq 0}$ is an $\bL^2$ bounded martingale and thus $\exists~ W>0 $ such that  $Y_\nu(t)\exp(-\nu t)\convas W$. Further $W \sim \exp(1)$. 
\end{lem} 
Next we derive moment bounds for the attachment point processes for linear preferential attachment.
\begin{lemma} \label{lem:affine-pp-bound}
Fix $\nu > 0$, $\kappa \ge 0$. \chsb{Let $\xi_{\nu, \kappa}(t)$ be the offspring process associated with a $\PA_{\nu,\kappa}(\cdot)$ process}. Then with respect to the natural filtration the following two processes are martingales which give the proceeding two expectations:
\begin{equation*}
M_1(t) := e^{- \nu t} \xi_{\nu, \kappa}(t)  - \frac{\nu + \chr{\kappa}}{\nu} \left( e^{\nu t} - 1\right), \text{ } t \ge 0
\end{equation*}
\begin{equation*}
M_2(t) := e^{- 2 \nu t} \xi_{\nu, \kappa}(t)^2  - \int_0^t  (2 \kappa + 3\nu) \xi_{\nu, \kappa}(s) e^{- 2 \nu s} ds - \frac{\nu + \kappa}{2 \nu} \left( 1 - e^{- 2 \nu t} \right), \text{ } t \ge0.
\end{equation*}
\begin{equation*}
\E \xi_{\nu, \kappa}(t) = \frac{\nu + \kappa}{\nu} \left( e^{\nu t} - 1\right), \text{ and } \E \left( \xi_{\nu, \kappa}(t)\right)^2 = \frac{(2 \kappa + 3 \nu)(\nu + \kappa)}{ 2\nu^2}\left(e^{\nu t} - 1\right)^2 + \frac{\nu + \kappa}{2\nu} \left(e^{2 \nu t} - 1\right).
\end{equation*}
\end{lemma}

\begin{proof}
We sketch the proof. Let $\cF(t)$ be the natural filtration corresponding to the continuous time branching process with attachment function $f$. Note that $\xi_{\nu, \kappa}(t)  \leadsto \xi_{\nu, \kappa}(t) + 1$ at rate $\nu(\xi_{\nu, \kappa}(t) + 1) + \kappa$. This can be used to check $\E\left[ d M_1(t) | \cF(t) \right] = 0$ showing $M_1(t)$ is a martingale. Similarly, $\xi_{\nu, \kappa}(t)^2\leadsto \xi_{\nu, \kappa}(t)^2 + 2 \xi_{\nu, \kappa}(t) + 1$ at rate  $\nu(\xi_{\nu, \kappa}(t) + 1) + \kappa$. This expression can similarly be used to check $M_2(t)$ is a martingale.  The first expectation claimed in the lemma follows immediately by setting the expectation of $M_1(t)$ equal to zero. The second expectation follows by computing the expectation of $M_2(t)$ and then using the expectation of $\xi_{\nu, \kappa}(t)$. 
\end{proof}

% The next result derives moment bounds for a particular class of CTBP.

We now derive expressions for moments of the process $\PA_{\nu,\kappa}$. To simplify notation, when possible we will suppress dependence on $\nu,\kappa$ and write the above as $\PA(\cdot)$. Note the proof of Proposition \ref{prop:mean-PA-model} is similar to the proof of Lemma \ref{lem:affine-pp-bound} and is therefore omitted. 
\begin{prop} \label{prop:mean-PA-model}
Fix $\nu>0,\kappa\geq 0$. With respect to the natural filtration the following processes are martingales:
\[M_1(t):= e^{-(2\nu+\kappa) t}(|\PA_{\nu, \kappa}(t)| - 1) -\frac{\nu + \kappa}{2\nu+\kappa}(1-e^{-(2\nu+\kappa) t}), \qquad t\geq 0\]
 \[M_2(t):= (|\PA_{\nu,\kappa}(t)|-1)^2 - \int_0^{t}((4\nu+2\kappa)(|\PA_{\nu,\kappa}(s)|-1)^2+ (4\nu +3\kappa)(|\PA_{\nu,\kappa}(s)|-1) + (\nu + \kappa))ds, \quad t\geq 0.\]
In particular, for any fixed $a>0$, $\exists$ $C$ (dependent on $\nu$ and $\kappa$) such that for $0\leq t\leq a$
\begin{equation}
\label{eqn:mom-bound-PA}
    \E(|\PA_{\nu,\kappa}(t)|)-1\leq Ce^{(2\nu + \kappa)a}t; \qquad \E((|\PA_{\nu,\kappa}(t)|-1)^2)\leq Ce^{(4\nu + 2\kappa)a}t.
\end{equation}
\end{prop}
%\begin{proof}
%% A new individual is born into $\PA(t)$ at rate $(2 \nu + \kappa) |\PA(t)| - \nu$. 
%Write $\{\cF(t) | t \ge 0\}$ for the natural filtration of the process. Note that  $|\PA(t)| \leadsto |\PA(t)|+1$ at rate $\sum_{x \in \PA(t)} (\nu (d_x(t) + 1) + \kappa)= (2 \nu + \kappa) |\PA(t)| - \nu$ where $d_x(t)$ is the number of children of $x$ at time $t$. This can be used to check $\E(d M_1(t) | \cF(t)) = 0$.  Computing expectations gives $\E \PA(t) - 1 = \frac{\nu + \kappa}{2\nu+\kappa} \left(e ^{(2\nu+\kappa) t} - 1 \right)$ from which the first moment bound follows for $t \le a$. 
%Similarly, $\PA(t) - 1$ undergoes the change $(\PA(t+) - 1)^2 - (\PA(t) - 1)^2= 2(\PA(t) - 1)+ 1$ at rate  $(2\nu + \kappa) (\PA(t) - 1) + \nu + \kappa$. This can be used to check $M_2(\cdot)$ is a martingale. Computing the expectation of this martingale gives the second moment bound. %The second assertion follows from the first assertion by explicitly computing the first two moments of $\PA_\nu(t)$. We only prove the first assertion on $M_1(\cdot)$. The proof for $M_2(\cdot)$ is similar. 
%\end{proof}

% The next lemma describes limit results for a number of important characteristics of relevance in this paper.
Recall the class of characteristics $\mathcal{C}$ defined in \eqref{charclass}, \chnr{the Malthusian rate of growth $\lambda^*$ and the mean measure of the offspring process $\mu_{f}$}.
Let $m^{\star}:=\int_{\bR_+}ue^{-\lambda^* u} \mu_f(du)$. For any fixed characteristic $\chi \in \cC$ and any $\alpha>0$, define,
$$
\hat{\chi}(\alpha):=\int_0^{\infty} \alpha e^{-\alpha t}\chi(t) dt, \qquad \hat{\mu}_f(\alpha):=\int_0^{\infty} \alpha e^{-\alpha t}\mu_f(t)dt.
$$
It can be checked that for any $\alpha>0$, recalling $\hat{\rho}$ from Assumption \ref{ass:attach-func} (ii), $\hat{\rho}(\alpha) = \hat{\mu}_f(\alpha) = \int_0^{\infty}e^{-\alpha t} \mu_f(dt).$ Recall the definitions of $Z^{\chi}_f(t)$, $M^{\chi}_f(t)$,$Z_f(t)$ and  $M_f(t) = \E\left(e^{-\lambda^*t}Z_f(t)\right)$ from Section \ref{subsec:conv_rates_no_cp}. \chsb{The first part of the following lemma is a consequence of \cite[Theorem 6.3]{nerman1981convergence} and the second part follows from \cite[Theorem 5.4 and Corollary 3.3]{nerman1981convergence}.}
% Recall $Z^{\chi}_f(t) = \sum_{x \in \BP_f(t)}\chi(t - \sigma_x)$ and $M^{\chi}_f(t) = \E\left(e^{-\lambda^*t}Z^{\chi}_f(t)\right)$.
%Recall the definitions of $Z_f(t)$ and $M_f(t)$ from Section \ref{subsec:conv_rates_no_cp}. The following Lemma is a consequence of \cite[Theorem 6.3]{nerman1981convergence}. % Recall $Z_f(t)$ is the total number of vertices at time $t$ and $M_f(t) = \E\left(e^{-\lambda^*t}Z_f(t)\right)$. 
\begin{lemma}\label{lem:deg_dist_quad_conv}
\begin{itemize}
\item[(i)] Under Assumption \ref{ass:attach-func} (ii), for any characteristic $\chi \in \cC$,
$
Z^{\chi}_f(t) \big /Z_f(t) \stackrel{a.s.}{\longrightarrow} \E(\hat{\chi}(\lambda^*)).
$
\item[(ii)] Under Assumptions \ref{ass:attach-func} and \ref{xlogx}, there exists a strictly positive random variable $W_\infty$ with $\E(W_\infty)=1$ such that for characteristics $\chi \in \cC$, 
$
e^{-\lambda^* t} Z_f^{\chi}(t) \stackrel{a.s., \ \bL^1}{\longrightarrow} W_\infty \E(\hat{\chi}(\lambda^*)) \big /\lambda^* m^{\star} .
$
 \end{itemize}
\end{lemma}

\section{Change point model run for fixed time $a$ after change point: point-wise convergence for general characteristics}
\label{sec:proofs-sp}

%Recall that in the continuous time construction of the standard model in Definition \ref{def:ctbp-embed}, 
%\begin{enumeratea}
%   \item we grow the branching process with attachment function $f_0$ till stopping time $T_{\gamma n}$;
%   \item then all vertices change their dynamics using attachment function $f_1$ and now we grow the branching process for an additional $T_n - T_{\gamma n}$ units of time. 
%\end{enumeratea}
% The aim of this Section is to analyze a related model where we (b) above is replaced by 
 
% (b)$^\prime$ grow the branching process for a fixed $a$ units of time. 
\chsb{In this section we consider growing the tree \bbb{(in continuous time)} using attachment function $f_0$ till it reaches size $\gamma n$, and using the second attachment function $f_1$ for a constant time $a$ after the change point i.e. $f_1$ is used for $t \in [T_{\gamma n}, T_{\gamma n} + a]$. We will also assume throughout this section that $f_0,f_1$ satisfy Assumption \ref{ass:attach-func}}. We count vertices \textit{born after the change point} according to a general characteristic $\phi \in \cC$ (defined in \eqref{charclass}) and prove a law of large numbers for this aggregate $\phi$-score at time $a$ as $n \rightarrow \infty$ (see Theorem \ref{thm:1}). This will be a key tool in the rest of the paper. For notational convenience we \chnr{consider the change point as time $t=0$} (i.e. $t=s$ corresponds to actual time $T_{\gamma n} + s$ for any $s \in [0,a]$). \chsb{Recall the continuous time embedding of the model in an inhomogeneous branching process $\BP_{\mvtheta}(\cdot)$ described in Section \ref{cpembedding}. For $t \ge 0$, we will write $\BP_n(t) := \BP_{\mvtheta}(t + T_{\gamma n})$ to denote the branching process at time $t$ (i.e. time $t$ after the change point).}

 \subsection{Notation}
% Note in this section we only count vertices which are born after the change point e.g. $Z_{f_1}^{\phi}(0) = 0$ (although there are $n\gamma$ verticies in the tree at time $t=0$). 
Let $\lambda^*_{i}$ be the Malthusian parameter for the branching process with attachment function \chiain{$f_{i}, i = 0,1$}. \chsb{For $k \ge 0$, $i=0,1$, recall $\xi^{(k)}_{f_i}(\cdot)$ and $\mu^{(k)}_{f_i}(\cdot)$ from \eqref{eqn:xi-k-f-t-def} with $f_i$ in place of $f$. For $0 \le s \le t$, let $\xi^{(k)}_{f_i}[s,t] := \xi^{(k)}_{f_i}(t) - \xi^{(k)}_{f_i}(s)$ and $\mu^{(k)}_{f_i}[s,t] := \mu^{(k)}_{f_i}(t) - \mu^{(k)}_{f_i}(s)$.}
For the branching process (without change point) with attachment function $f_1$, and for any characteristic $\phi$, recall $Z_{f_1}^{\phi}(t), m_{f_1}^{\phi}(t), Z_{f_1}$ from Section \ref{subsec:conv_rates_no_cp} \chsb{(just after \eqref{charclass})}. Let $v_{f_1}^{\phi}(t) = \var \left( Z_{f_1}^{\phi}(t) \right)$. 
%recall $Z_{f_1}^{\phi}(t) = \sum_{x \in \BP_{f_1}(t)} \phi_x(t - \sigma_x)$. When $\phi(t) = \ind\set{t \ge 0}$, we will write $Z_{f_1}$ for $Z_{f_1}^{\phi}$. Let $m_{f_1}^{\phi}(t) := \E Z_{f_1}^{\phi}(t)$ and let $v_{f_1}^{\phi}(t) = \var \left( Z_{f_1}^{\phi}(t) \right)$. 
%For post change point analysis, we will consider the class of characteristics $\cC$ defined in \eqref{charclass}. 
%For post change point analysis, we will consider characteristics which are either bounded by $1$ or are of the form $\phi(t) = \sum_{\ell=0}^{\infty} c_{\ell}\ind\set{\xi_{f_1} = \ell}$ where $c_{\ell}$ are deterministic positive constants satisfying $c_{\ell} \le c(\ell + 1)$ for some constant $c>0$. We will call this class of characteristics $\cC$. 
\chsb{Recall the class of characteristics defined in \eqref{charclass}.} For $\phi \in \cC$, an easy computation implies there exists \chnr{non-random $b_{\phi}>0$} such that $Z_{f_1}^{\phi}(t) \le 2\chr{b_{\phi}} Z_{f_1}(t)$ for every $t \ge 0$. \chsb{Moreover, by Assumption \ref{ass:attach-func}{(i)} on $f_1$, $Z_{f_1}(\cdot)$ is stochastically dominated by $|\PA_{2C,0}(\cdot)|$ (see Definition \ref{def:rate-c-pa}), in the sense that there exists a coupling $(Z_{f_1}(\cdot), |\PA_{2C,0}(\cdot)|)$ satisfying $Z_{f_1}(t) \le |\PA_{2C,0}(t)|$ for all $t \ge 0$. Combining these observations and using \eqref{eqn:mom-bound-PA}, we obtain for any $\phi \in \mathcal{C}$,}
\begin{equation}
 \label{eqn:vf1-bound}
    \sup_{t \in [0, a]}m_{f_1}^{\phi}(t)  \leq 2\chr{b_{\phi}}\E(Z_{f_1}(a)) \le \chr{C_1e^{C_2a}}, \ \ \sup_{t \in [0, a]}v_{f_1}^{\phi}(t)  \leq 4\chr{b_{\phi}}^2\E(Z_{f_1}^2(a)) \le \chr{C_1e^{C_2a}}
 \end{equation} 
 where \chsb{$C_1, C_2$ depend on $\phi$ but not on $a$. For $\phi \in \mathcal{C}$, let $Z^{\phi}_n(a)$ denote the aggregate $\phi$-score at time $a$ (see Section \ref{sec:struct-res}, just after \eqref{charclass}) of all individuals born after the change point, namely} 
 $$
 \chsb{Z^{\phi}_n(a) := \sum_{x \in \BP_n(a) \setminus \BP_n(0)}\phi_x(a - \sigma_x).}
 $$
 \chsb{For $k \ge 0$, $0 \le t \le a$, let $\lambda_{k}^\phi(t) = \int_0^{t} m_{f_1}^\phi(t - s) \mu^{(k)}_{f_1}(ds)$ denote the expected aggregate $\phi$-score at time $t$ of all descendants that are born in $(0,t]$ to a vertex with degree $k$ at time $0$. For $k \ge 0$, let $D_n(k,0)$ denote the number of vertices with degree $k$ at the change point time $0$.}
 \sba{Let $\cF_n(0)$ denote the $\sigma-$field containing the information on the entire branching process till time $T_{n\gamma}$, the change point. }
 %This follows by Assumption \ref{ass:attach-func}{(i)} on $f_1$ that implies $\BP_{f_1}(\cdot)$ is stochastically dominated by a rate $C$ $\PA$ branching processes (see Definition \ref{def:rate-c-pa}) and then by appealing to \eqref{eqn:mom-bound-PA}.
%Recall that $f_i$ satisfy $f_i(k) \le Ck$ for all $k$, $i = 0, 1$.
% \item $e^{- \alpha_1 t}\E Z_{f_1}^{\phi}(t) $ is uniformly continuous. NB: This should follow from point-wise continuity and $m_{f_1}^{\phi}(t) \to m_{f_1}^{\phi}(\infty)$ [TODO] \label{asmpt:3}
% \item There exists $\beta_0 < \alpha_0$ such that $\E \int_0^{\infty} e^{-\beta_0 s} \xi_{f_0}(ds) < \infty$ \label{asmpt:4}

The following is the main result \chnr{proved} in this section. 
\chnr{ \begin{theorem} \label{thm:1}
\chsb{Suppose $f_0,f_1$ satisfy Assumption \ref{ass:attach-func}}. Fix any $\phi \in \cC$. There exist deterministic positive constants $C, C'<\infty$ such that for every $a >0$ and $n \ge 2$,
$$ \small \E\left[\left| Z^{\phi}_n(a) -  \sum_{k=0}^{\infty} D_n(k,0)  \lambda_k^\phi(a)\right| \ \Big| \cF_n(0)\right]\le Ce^{C'a}\sqrt{n}.$$
\chsb{Furthermore}, as $n \to \infty$, $ Z^{\phi}_n(a) / n \probc  \gamma  \sum_{k=0}^{\infty} p_k^0 \lambda_{k}^\phi(a).$% \int_0^a m_{f_1}^{\phi} (a - s) \mu_{f_1}^{(k)}(ds)$$
 \end{theorem}}
 
 \subsection{Definitions}\label{defsec}
 In this section we define constructs \chsb{for the branching process $\BP_n(\cdot)$} which will be used in the proof of Theorem \ref{thm:1} \chnr{(and are motivated by the proof outline in Section \ref{ss:pf_thm_6.1}).} %Next we define constructs which will be used in the proof of Theorem \ref{thm:1}. 
 \chsb{For notational simplicity, since $a$ is fixed in this section, we will write $Z_n^{\phi}$ for $Z_n^{\phi}(a)$ and $\lambda_{k}^\phi$ for $\lambda_{k}^\phi(a)$.}
 \chnr{\textbf{For the rest of this section, unless specified otherwise, we always work conditional on $\cF_n(0)$} so that expectation operations such as $\E(\cdot)$ and $\var(\cdot)$ \chsb{for the associated branching process statistics post change point} in the ensuing results mean $\E(\cdot|\cF_n(0))$ and $\var(\cdot|\cF_n(0))$. }
Divide the interval $[0, a]:= \cup_{i=0}^{n^\delta-1}[ia n^{-\delta}, ((i+1)a)n^{-\delta}]$ into subintervals of size $a n^{-\delta}$. We will eventually take limits as $\delta \rightarrow \infty$.
  
(i) {\bf System at change point:} Define the filtration $\set{\cF_n(t):t\geq 0}:= \set{\sigma(\BP_n(t)):t\geq 0}$ \chsb{(information till $t$ time units after change point)}. For fixed $k\geq 0$, we write $\cD_n(k,t)$ \chnr{for} the set of vertices with degree $k$ at time $t$ and let $D_n(k,t):=|\cD_n(k,t)|$. The initial set $\cD_n(k,0)$ which arose from the pre-change point dynamics will play a special role. Label the vertices in $\cD_n(k,0)$ in the order they were born into $\BP_n(0)$ as $\cD_n(k,0):=\set{v_1^{\sss(k)}, v_2^{\sss(k)}, \ldots, v_{D_n(k,0)}^{\sss(k)}}$. \chsb{Let $\cD_n(0) := \cup_{k \ge 0} \cD_n(k,0)$.}
 
% \textcolor{blue}{SBan: The $D_n(k,t)$ notation is also used for the tree. Replace with better notation.}
 
(ii) {\bf Descendants in small intervals:} For $0\leq i\leq n^{\delta}-1, k \ge 0$ and \bbb{vertex} $v_j^{\sss(k)}\in \cD_n(k,0)$, let $\cV_n^{\sss(k)}(i, j)$ denote the set of children born in the interval $\left[ian^{-\delta}, (i + 1)a n^{-\delta} \right]$ to $v_j^{\sss(k)}$. Let $N_n^{\sss(k)}(i, j) := |\cV_n^{(k)}(i, j)|$ be the number of such vertices.  % As we condition on $\cF_0$, assume $N_k$ parents of degree $k$.
   Write $N_n^{\sss (k)}(i) := \sum_{j=1}^{D_n(k,0)} N_{n}^{\sss(k)} (i, j)$ \chsb{for the total number of children of vertices that were of degree $k$ at the change point, born in the time interval $\left[ian^{-\delta}, (i + 1)a n^{-\delta} \right]$}.
   % Note $\E N_{n}^{(k)} (i, j) = \int_{\frac{ia}{n^{\delta}}}^{\frac{(i+1)a}{n^{\delta}}} \mu_{f_1}^{(k)}(ds)$. \todo{Remove this line and move to more appropriate location later.}

(iii) {\bf Good and bad vertices: } Call a vertex in $\cV_n^{(k)}(i, j)$ a \emph{good} vertex if it does {\bf not} give birth to any children  by time $(i + 1)an^{-\delta}$. Let $\widetilde{\cV}_n^{(k)}(i, j)\subseteq \cV_n^{(k)}(i, j) $ denote the set of  \chsb{good children of $v_j^{\sss(k)}$ born} in the interval $\left[ian^{- \delta}, (i + 1)a n^{-\delta} \right]$.  % parent vertex who had degree $k$ at the change point AND who have \textbf{not} reproduced
 Let $\widetilde{N}_n^{\sss(k)}(i, j) := |\widetilde{\cV}_n^{\sss(k)}(i, j)|$ be the number of such vertices. Let $\widetilde{N}_n^{\sss(k)}(i) := \sum_{j=1}^{D_n(k,0)} \widetilde{N}_{n}^{\sss(k)} (i, j)$ be the total number of \chsb{\emph{good} children of vertices} which originally had degree $k$ at the change point \chsb{born in the interval $\left[ian^{- \delta}, (i + 1)a n^{-\delta} \right]$}. Let $\cB_n^{\sss(k)}(i, j):= \cV_n^{\sss(k)}(i, j)\setminus \widetilde{\cV}_n^{\sss (k)}(i, j)$ be the collection of bad children namely those in $\cV_n^{\sss (k)}(i, j)$ who \textbf{have} reproduced by time $(i + 1)a n^{- \delta}$.  Let $B_n^{(k)}(i, j)= |\cB_n^{\sss(k)}(i, j)|$. \chsb{Let $\cR_n^{\sss(k)} (i, j)$ be the set of descendants of vertices in $\cB_n^{(k)}(i, j)$ (excluding the parent vertices in $\cB_n^{(k)}(i, j)$) born in the time interval $\left[ ia n^{-\delta}, (i + 1)an^{-\delta} \right]$ and let $R_n^{\sss(k)}(i,j):=|\cR_n^{\sss(k)} (i, j)|$.} 
 
(iv) {\bf Vertices counted by a characteristic:} \chsb{For $0\leq i\leq n^{\delta}-1$, $k \ge 0$, $1 \le j \le D_n(k,0)$ and $x \in  \cV_n^{\sss(k)}(i, j)$, let $Z_n^{\sss(k), \phi}(i, j, x)$ be the aggregate $\phi$-score (defined in Section \ref{sec:struct-res}, just after \eqref{charclass}) of $x$ and its descendants at time $a$. More precisely, denoting the set of $x$ and its descendants at time $a$ by $\mathcal{U}_n^{\sss(k)}(i, j,x)$,
$
Z_n^{\sss(k), \phi}(i, j, x) := \sum_{y \in \mathcal{U}_n^{\sss(k)}(i, j,x)}\phi_y(a - \sigma_y).
$
Write $Z_n^{(k), \phi} = \sum_{j=1}^{D_n(k,0)} \sum_{i=0}^{n^{\delta}-1} \sum_{x \in \cV_n^{\sss(k)}(i, j)} Z_n^{\sss(k), \phi}(i, j, x)$ for the aggregate $\phi$-score at time $a$ of all vertices of degree $k$ at the change point along with their descendants at time $a$. Thus, $Z^{\phi}_n = \sum_{k=0}^{\infty} Z^{\sss(k), \phi}_n.$
For $k \ge 0$, let $\widetilde{Z}_n^{(k), \phi} = \sum_{j=1}^{D_n(k,0)} \sum_{i=0}^{n^{\delta}-1} \sum_{x \in \widetilde{\cV}_n^{\sss(k)}(i, j)} Z_n^{\sss(k), \phi}(i, j, x)$ denote the aggregate $\phi$-score at time $a$ of all \emph{good} vertices born in $\left[ 0,a \right]$ which are descendants of vertices of degree $k$ at the change point.}
%\item Let $m_{f_1}^{\phi}(t) = \E \left( Z_{f_1}^{\phi}(t) \right)$ and $v_{f_1}^{\phi}(t) = \var \left( Z_{f_1}^{\phi}(t) \right)$ [TODO: Need to define this. ]
%For $k \ge 0$, $i=0,1$, recall $\xi^{(k)}_{f_i}(\cdot)$ from \eqref{eqn:xi-k-f-t-def} with $f_i$ in place of $f$. For $0 \le s \le t$, let $\xi^{(k)}_{f_i}[s,t] := \xi^{(k)}_{f_i}(t) - \xi^{(k)}_{f_i}(s)$.}
% denote the random variable \chnr{distributed according to} the number of children born in the time interval $[s,t]$ to a vertex  who had degree $k$ at time 0 with attachment function $f_{\ell}, \ell = 0, 1$. Write $\xi^{(k)}_{f_{\ell}}(t)$ for $\xi^{(k)}_{f_{\ell}}[0,t]$. 
 %\todo[inline]{This makes no sense. I think you mean the distribution of such a vertex. There is no dependence on the vertex $v$ in the above definition. } 

(v) {\bf Technical conditioning tool:} \chsb{For $0 \le s < t \le \infty$, let $\mathcal{G}[s,t]$ be the $\sigma$-field generated by the biographies of all individuals born in $[s,t]$ over the same time interval. Formally,
$$
\sba{\mathcal{G}[s,t] := \sigma\left(\{s \le \sigma_x \le t - u\} \cap \{\xi_{x,f_1}(u) = k\}, \ x \in \mathcal{I} \setminus \cD_n(0), u \in [0,t-s], k \in \mathbb{Z}_+\right).}
$$
Moreover, let $\mathcal{G}_0$ denote the $\sigma$-field generated by the entire biographies of the individuals at time $0$, namely,
$$
\mathcal{G}_0 := \sigma\left(\{\xi_{x,f_1}(u) = k\}, \ x \in \cD_n(0), u \in [0,\infty), k \in \mathbb{Z}_+\right).
$$
The following $\sigma$-field will play a crucial role in controlling fluctuations of aggregate $\phi$-scores of good vertices}
$$
\cG_n := \sigma\Big( \mathcal{G}_0
  \bigcup_{0 \le i \le  n^{\delta}-1} \mathcal{G}[ia n^{-  \delta}, (i + 1)a n^{-\delta}]\Big).
$$

%(vi) {\bf Mean of characteristics emanating from degree $k$ parent:} \chnr{For $0 \le t \le a$, let $\lambda_{k}^\phi(t) = \int_0^{t} m_{f_1}^\phi(t - s) \mu^{(k)}_{f_1}(ds)$ denote the expected aggregate $\phi$-score of descendants born in $[0,t]$ to pre-change point vertices of degree $k$.} For notational simplicity since $a$ is fixed in this Section, we will write $\lambda_{k}^\phi := \lambda_{k}^\phi(a)$. 

 % \todo[inline]{Should this not be $\E_{k}(\phi(a)) + \lambda_{\ell}^\phi(t) $ where $\E_{k}$ is the expectation under $\xi_{f_1}^{\sss(k)}$? We need to count the characteristic associated with the parent vertex?}

 % for $t \ge \tcp$ and 0 otherwise.

 % \item Let $\lambda_{\ell}^{(k)}(t) = \lambda_{\ell}(t) + \pr \left( \xi_{f_1}^{(\ell)}(t ) = k - \ell \right)$. Note the second term is 0 when $\ell > k$.  % for $t \ge \tcp$ and 0 otherwise. Note the second term is 0 when $\ell > k$. 

 %The following is the main result \chnr{proved} in this section. 

% \begin{theorem} \label{thm:1}
% Fix any $\phi \in \cC$. There exist deterministic positive constants $C, C'<\infty$ (not dependent on $a$) such that for every $a >0$ and $n \ge 2$,
% $ \E\left[\left| Z^{\phi}_n -  \sum_{k=0}^{\infty} D_n(k,0)  \lambda_k^\phi\right| \ \Big| \cF_n(0)\right]\le Ce^{C'a}\sqrt{n}.$
%In particular, as $n \to \infty$, $ Z^{\phi}_n / n \probc  \gamma  \sum_{k=0}^{\infty} p_k^0 \lambda_{k}^\phi(a).$% \int_0^a m_{f_1}^{\phi} (a - s) \mu_{f_1}^{(k)}(ds)$$
 %\end{theorem}

 \subsection{Proof of Theorem \ref{thm:1}} \label{ss:pf_thm_6.1}
 \chnr{We first give an outline of the proof. We discretize the time interval $[0,a]$ into small subintervals of the form $\{[ia n^{-\delta}, ((i+1)a)n^{-\delta}]\}_{0 \le i \le n^{\delta}-1}$, for $\delta>0$. % MEOW: this line goes over -- should we fix it?
 For an individual born in the interval $[ia n^{-\delta}, ((i+1)a)n^{-\delta}]$, the final aggregate $\phi$-score of its descendants at time $a$ (counting the parent) is estimated by the expected aggregate $\phi$-score of a degree zero parent and its descendants tracked till time $a - ((i+1)a)n^{-\delta}$, which equals $m^{\phi}_{f_1} \left(a - \frac{(i +1)a}{n^{\delta}} \right)$. For this approximation to be valid, we need to show that the total number of bad vertices defined above is small in an appropriate sense. In fact, a number of lemmas below are `continuity estimates' that validate this discrete approximation. These lemmas are very general and are also used in subsequent sections. As the expected number of children born in the time interval $[ia n^{-\delta}, ((i+1)a)n^{-\delta}]$ to a vertex having degree $k$ at time zero equals $\mu_{f_1}^{(k)} \left[\frac{ia}{n^{\delta}}, \frac{(i + 1)a}{n^{\delta}} \right]$, and there are $D_n(k,0)$ degree $k$ vertices at time zero, $Z^{\phi}_n$ is thus estimated by the Riemann sum $\sum_{k=0}^{\infty} D_n(k,0) \sum_{i=0}^{n^{\delta} - 1} m^{\phi}_{f_1} \left(a - \frac{(i +1)a}{n^{\delta}} \right) \mu_{f_1}^{(k)} \left[\frac{ia}{n^{\delta}}, \frac{(i + 1)a}{n^{\delta}} \right]$. This Riemann sum can then be shown to be close to $\sum_{k=0}^{\infty} D_n(k,0)  \lambda_k^\phi$. }

We fix a characteristic $\phi \in \cC$ throughout the proof. The main tools used are Lemmas \ref{lem:8}, \ref{lem:9} below. We will need a number of supporting results which we now embark upon. 
 % [TODO: we could probably collect: Lemma \ref{lem:1}, the Yule fact from Lemma \ref{lem:2}, Lemma \ref{lem:deg_dist_quad_conv}, Lemma \ref{lem:6}, Lemma \ref{lem:9} into an separate tools section]
\chnr{For the rest of this section we write $C_1, C_2, C_3, C_4, C, C', c, a_0$ for generic \chsb{non-random} constants which are independent of $a, n, \delta, k$, whose values might change between lines and inequalities. We start with a technical lemma controlling the number of children a vertex with degree $k$ at \chsb{time $0$} can produce within a fixed interval.} 
%For the rest of this section we write $C_1, C_2, C_3, C_4, C, C', a_0$ for constants which are independent of $a, n, \delta, k$.
 \begin{lemma} \label{lem:1}
For \chsb{any $k \ge 0$ and} any interval $[b, b + \eta] \subseteq [0, a]$,
 \[\E\left[ \xi^{(k)}_{f_1} [b, b + \eta] \right] \le C_1 e^{C_2a}(k+1) \eta, \qquad \E\left[ \xi^{(k)}_{f_1} [b, b + \eta]^2 \right]  \le C_3e^{C_4a} \left\lbrace(k+1)^2 \eta^2 + (k+1)\eta\right\rbrace.\]
 \end{lemma}

 \begin{proof}
 % TODO: maybe reprhase this proof in terms of Yule process, not PA 
 % $$\E \left[ \xi^{(k)}_{f_1} [b, b + \eta ]  | \sF_b \right] = E \left[ \sum_{d=0}^{\infty} 1\left( \xi^{(k)}_{f_1} (b) = d \right)  \xi^{(k + d)}_{f_1} (\eta)    | \sF_b \right] $$
 %To simplify notation write $\xi_{f_1}^{(k)}(b):= \xi^{(k)}_{f_1} [0, b] $. 
 By Assumption \ref{ass:attach-func}(i), the process $\set{U(t):= \xi_{f_1}^{(k)}(t/C): t\geq 0}$ is stochastically dominated by the offspring process $\set{P_k(t):t\geq 0}$ of a rate $1$ affine $k$ PA model (see Definition \ref{def:rate-c-pa}), namely a point process constructed using attachment function $f^{(k)}(i)  = k + 1 + i$ for $i \ge 0 $ with initial condition $P_k(0) =0$. From the first moment computed in Lemma \ref{lem:affine-pp-bound} (with $\nu=1$ and $\kappa=k$) we obtain
%  By \cite[Lemma 4.5]{bhamidi2015change}, 
% \begin{equation}
% \label{eqn:m1-mart}
%   M_1(t) = e^{-t } P_k(t)  - (1 + k)(1 - e^{-t}), \qquad t\geq 0,
% \end{equation}  
% is a martingale resulting in $\E P_k(t) = (1 + k)(e^t - 1)$. By arguments similar to \cite[Eqns 4.22-4.24]{bhamidi2015change}, one can also check that 
% \begin{equation}
% \label{eqn:m2-mart}
%   M_2(t) := e^{-2t} P_{k}^2(t) - (3 + 2k) \int_0^t e^{-2s} P_{k}(s) ds + \frac{1+k}{2}(1-e^{-2t})
% \end{equation}
% is also a Martingale with $M_2(0):=0$. Using these two processes, we get that 
$
\E (P_k(t)) = (1 + k)(e^t - 1).
$
 We show how to use the first moment of $P_k(\cdot)$ to obtain the first assertion in the lemma. The second assertion follows from the same argument using the second moment of $P_k(\cdot)$ which is also obtained from Lemma \ref{lem:affine-pp-bound}.
%We will use these two processes to obtain our results. We show how to use the first moment of $P_k(\cdot)$ in \eqref{eqn:m1-m2-expec} to obtain the first assertion in the Lemma. The second assertion follows in an identical fashion using the second moment of $P_k(\cdot)$. 
 Conditioning on $\xi_{f_1}^{(k)}(b)$ and using the  Markov property we get, 
 \begin{equation}
    \label{eqn:xi-markov-decomp}
 \E \xi^{(k)}_{f_1} [b, b + \eta]  = \sum_{d=0}^{\infty} \pr \left( \xi^{(k)}_{f_1} (b) = d \right)  \E \xi^{(k + d)}_{f_1} (\eta)
 \end{equation}
 Now for any fixed $k\geq 0$ and $t\leq a$, using domination by the corresponding linear PA process, we get  
 \begin{equation}
 \label{eqn:xi-fm-bound}
    \E[\xi_{f_1}^{(k)}(t)] \le  \E (P_k(tC)) =  e^{tC}(1 + k)(1 - e^{-tC})\leq e^{Ca}C(1+k)t.
 \end{equation}
 Using this bound twice in \eqref{eqn:xi-markov-decomp} gives, % $\E \xi^{(k)}_{f_1} [b, b + \eta]\leq Ce^{Ca}\eta\sum_{d=0}^{\infty} \pr\left( \xi^{(k)}_{f_1} (b) = d \right)(1+k+d)=Ce^{Ca}\eta(1+k + \E(\xi^{(k)}_{f_1}(b))) \leq Ce^{Ca}\eta(1+k + Cbe^{Ca}C(1+k)) \le C' e^{C''a}(k+1)\eta $
 $$ 
 \E \xi^{(k)}_{f_1} [b, b + \eta]\leq Ce^{Ca}\eta\sum_{d=0}^{\infty} \pr\left( \xi^{(k)}_{f_1} (b) = d \right)(1+k+d)=Ce^{Ca}\eta(1+k + \E(\xi^{(k)}_{f_1}(b))) \le C' e^{C''a}(k+1)\eta 
 $$
% \begin{multline}
% \label{eqn:xi-k-final-bound}
% \E \xi^{(k)}_{f_1} [b, b + \eta]\leq Ce^{Ca}\eta\sum_{d=0}^{\infty} \pr\left( \xi^{(k)}_{f_1} (b) = d \right)(1+k+d)=Ce^{Ca}\eta(1+k + \E(\xi^{(k)}_{f_1}(b)))\\
%    \leq Ce^{Ca}\eta(1+k + Cbe^{Ca}C(1+k)) \le C' e^{C''a}(k+1)\eta
% \end{multline}
where $C', C''$ are constants that do not depend on $k,a$. This completes the proof. \end{proof}

 Recall from Section \ref{defsec} (ii) the random variable $N_n^{\sss(k)}(i,j)\stackrel{d}{=} \xi_{f_1}^{\sss(k)}\left[ ia n^{- \delta},(i+1)a n^{- \delta} \right]$. Using Lemma \ref{lem:1} now gives the following result.  
 \begin{corollary} \label{cor:1.5}
 For all $1\leq j\leq D_n(k,0)$, $\E (N^{(k)}_n(i, j)) \le C_1e^{C_2a} (k+1) n^{- \delta}$ and \\ $\E \left[ N^{(k)}_n(i, j)^2 \right] \le C_3e^{C_4a} \left\lbrace(k+1)^2 n^{- 2\delta} + (k+1)n^{-\delta}\right\rbrace$. 
 \end{corollary}

 The next Lemma bounds the number of ``bad'' vertices and their descendants born within small intervals.
 
\begin{lemma} \label{lem:2} % MEOW: there is a big vertical space here
For any $k, i, j$,
$\E (R_{n}^{(k)}(i, j) ) \le C_1e^{C_2a} (k+1) n^{- 2 \delta}.$
\end{lemma}
 \begin{proof}
 For every child $u\in \cV_n^{\sss(k)}(i,j)$, write $\BP(\cdot;u)$ for the branching process lineage emanating from $u$. Conditional on  $\cV_n^{\sss(k)}(i,j)$, using Assumption \ref{ass:attach-func}{(i)} on $f_1$, generate a collection of independent rate $C$ \chsb{affine $0$ linear $\PA$ branching processes} (see Definition \ref{def:rate-c-pa}) $\set{Y_{\ell}:1\leq \ell\leq |\cV_n^{\sss(k)}(i,j)|}$ such that $|\BP(\cdot;u)|\leq |Y_{\ell}(\cdot)|$.  Now note that \chsb{$X_{\ell}(t):=|Y_{\ell}(t)| -1$} is the number of descendants of the root for this branching process by time $t$. Using this construction we have the trivial inequality % the number of descendants born in $\left[\frac{ia}{n^{\delta}}, \frac{(i + 1)a}{n^{\delta}} \right]$  to another vertex born in $\left[\frac{ia}{n^{\delta}}, \frac{(i + 1)a}{n^{\delta}} \right]$ can be stochastically bounded by running rate $C$ linear preferential attachment for time $n^{- \delta}$ (where $C$ is from assumption (2)). In particular, let $\{X_{\ell}\}_{\ell=1}^{N_{n}^{(k)} (i, j)}$ be an iid collection of random variables (also independent of $N_{n}^{(k)} (i, j)$), where $X_1$ is distributed as the number of descendants born to the root vertex of linear preferential attachment run for $n^{-\delta}$ time.
 $R^{\sss(k)}_n (i, j)  \leq \sum_{\ell=1}^{N_{n}^{\sss(k)} (i, j) } X_{\ell}\left[0, a n^{-\delta}\right].$
 This implies 
 \[\E(R^{\sss(k)}_n (i, j)) \leq \E(N_{n}^{\sss(k)} (i, j))\E\left(X_1\left[0, a n^{- \delta}\right]\right).\]
% \[\E\left(\left[R^{\sss(k)}_n (i, j)\right]^2\right) \leq \E(N_{n}^{\sss(k)} (i, j))\E\left(\left[X_1\left[0, a  n^{- \delta} \right]\right]^2\right) +\E\left(\left[N_{n}^{\sss(k)} (i, j)\right]^2\right)\left(\E\left(X_1\left[0,a n^{- \delta}\right]\right)\right)^2. \]
The lemma follows from this bound upon using Corollary \ref{cor:1.5} for moments of $N_{n}^{\sss(k)} (i, j)$ and \eqref{eqn:mom-bound-PA} for moments of $X_1\left[0, a n^{-\delta}\right]$.
\end{proof}

 The next lemma bounds \chsb{fluctuations of aggregate $\phi$-scores of \emph{good} descendants of ancestors who were of degree $k$ at the change point}. 

 \begin{lemma} \label{lem:3}
For any $k \ge 0$,
 $\var \left( \widetilde{Z}_n^{(k), \phi}\right) \le Ce^{C'a} \left((k+1)^2n^{-\delta} + (k+1)\right) D_n(k,0)$.
 %where $C$ does not depend on $\cF_0$ or $k$ (but does $a$).
 \end{lemma}

 \begin{proof}
    By construction we have 
 \begin{align}
 \var\left( \widetilde{Z}_n^{(k), \phi} \right) & =\var\left(  \sum_{j=1}^{D_n(k,0)} \sum_{i=0}^{n^{\delta} - 1} \sum_{x \in \widetilde{\cV}_n^{(k)}(i, j)} Z_n^{(k), \phi}(i, j, x)  \right)\\
 & = D_n(k,0) \var\left( \sum_{i=0}^{n^{\delta} - 1} \sum_{x \in \widetilde{\cV}_n^{\sss(k)}(i, 1)} Z_n^{(k), \phi}(i, 1, x) \right). \label{eqn:tildz-var}
 \end{align}
 % since the $\left\{Z_n^{(k), \phi}(i, j, x) \Big| j = 1, \dots, D_n(k,0) \right\}$ are iid. We can bound the above variance using the conditional variance formula
 % \begin{align*}
 %  \var\left( \sum_{i=0}^{n^{\delta} - 1} \sum_{x \in \widetilde{V}_n^{(k)}(i, 1)} Z_n^{(k), \phi}(i, 1, x)  \right) = & \E \left [\var\left( \sum_{i=0}^{n^{\delta} - 1} \sum_{x \in \widetilde{V}_n^{(k)}(i, 1)} Z_n^{(k), \phi}(i, 1, x)  \Big | \cG^{(k)} \right)  \right] \\
 %  & + \var\left(  \E \left[ \sum_{i=0}^{n^{\delta} - 1} \sum_{x \in \widetilde{V}_n^{(k)}(i, 1)} Z_n^{(k), \phi}(i, 1, x)  \Big | \cG^{(k)} \right] \right)
 % \end{align*}
 We analyze the variance term on the right by first conditioning on $\cG_n$. Note that,
 \begin{align}
 \E \left [\var\left( \sum_{i=0}^{n^{\delta}-1} \sum_{x \in \widetilde{\cV}_n^{(k)}(i, 1)} Z_n^{(k), \phi}(i, 1, x)  \Big | \cG_n \right)  \right] & = \E \left[\sum_{i=0}^{n^{\delta}-1} \widetilde{N}_n^{(k)}(i, 1) v_{f_1}^{\phi}\left( a - (i + 1)a n^{- \delta} \right)\right] \notag \\
 &\le C_1e^{C_2a}(k+1)n^{-\delta}  n^{\delta} \E(Z_{f_1}^2(a)) \le Ce^{C'a} (k+1) \label{eqn:var-cond-gk}.
 % &\le Ckn^{-\delta}  \sum_{i=0}^{n^{\delta} - 1}  \sup_{t \in [0, a]}\left(v_{f_1}^{\phi}(t)\right) \\
 % & = C'k
 \end{align}
The first equality comes from noting that $\widetilde{\cV}_n^{(k)}(i, 1)$ is $\cG_n$ measurable, the collection $ \{Z_n^{(k), \phi}(i, 1, x) \ \big | \ x \in \widetilde{\cV}_n^{(k)}(i, 1), 1\leq i \leq  n^{\delta}-1\}$ are \chsb{conditionally independent given $\cG_n$} and further, \chsb{conditionally on $\cG_n$}, for each $0\leq i\leq n^{\delta}-1$ and $x\in \widetilde{\cV}_n^{(k)}(i, 1)$, $Z_n^{(k), \phi}(i, 1, x)$ is distributed as $Z_{f_1}^{\phi}\left(a - (i + 1)a n^{-\delta} \right)$, since $x$ has no children by time $(i + 1)a n^{- \delta}$. The second inequality follows by using Corollary \ref{cor:1.5} for $N_n^{\sss(k)}(i,1)$ and \eqref{eqn:vf1-bound}. Similarly 
% \begin{align}
% \var\left(  \E \left[ \sum_{i=0}^{n^{\delta} - 1} \sum_{x \in \widetilde{\cV}_n^{(k)}(i, 1)} Z_n^{(k), \phi}(i, j, x)  \Big | \cG_n \right] \right) & = \var\left( \sum_{i=0}^{n^{\delta} - 1} \widetilde{N}_n^{(k)}(i, 1) m_{f_1}^{\phi}\left( a - \frac{(i + 1)a}{n^{\delta}} \right)  \right) \notag\\
% & \le 4c^2\left( \E(Z_{f_1}(a))\right)^2 \sum_{i=0}^{n^{\delta} - 1}  \E \left[  \left( \widetilde{N}_n^{(k)}(i,1) \right)^2\right]\notag\\
%  & \leq Ce^{C'a} \left((k+1)^2n^{-\delta} + (k+1)\right)  \label{eqn:e-cond-gk}
% \end{align}
 \begin{multline}  \label{eqn:e-cond-gk}
 \var\left(  \E \left[ \sum_{i=0}^{n^{\delta} - 1} \sum_{x \in \widetilde{\cV}_n^{(k)}(i, 1)} Z_n^{(k), \phi}(i, 1, x)  \Big | \cG_n \right] \right) = \var\left( \sum_{i=0}^{n^{\delta} - 1} \widetilde{N}_n^{(k)}(i, 1) m_{f_1}^{\phi}\left( a - (i + 1)a n^{- \delta} \right)  \right)\\
 \le 4c^2\left( \E(Z_{f_1}(a))\right)^2 \sum_{i=0}^{n^{\delta} - 1}  \E \left[  \left( \widetilde{N}_n^{(k)}(i,1) \right)^2\right] \leq Ce^{C'a} \left((k+1)^2n^{-\delta} + (k+1)\right).
 \end{multline}
Here we use Corollary \ref{cor:1.5} in the second inequality. Using \eqref{eqn:var-cond-gk} and \eqref{eqn:e-cond-gk} to bound the variance term in the right of \eqref{eqn:tildz-var} completes the proof. 
 \end{proof}

 The next lemma provides tight bounds on expectations of \chnr{aggregate $\phi$-scores of descendants of good vertices}. Recall $\mu_{f_1}^{\sss(k)}$ denotes the mean measure \chsb{for the offspring process of a vertex} which had degree $k$ at the change point. 

 \begin{lemma} \label{lem:4}
For any $k \ge 0$,
\begin{align*}
\eps_n&:=\left |\E \left [\widetilde{Z}_n^{(k), \phi} \right] - D_n(k,0)\sum_{i=0}^{n^{\delta} - 1}  m_{f_1}^{\phi} \left(a - (i + 1)a n^{- \delta} \right) \mu_{f_1}^{(k)} \left[ia n^{- \delta}, (i + 1)a n^{- \delta} \right] \right |\\
& \le Ce^{C'a} (k+1) D_n(k,0) n^{- \delta}.
\end{align*}
 \end{lemma}

 \begin{proof}
 First note,
% \begin{align*}
% \E \left [\widetilde{Z}_n^{(k), \phi} \right]  & = \sum_{i=0}^{n^{\delta} - 1} \sum_{j=1}^{D_n(k,0)}  \E\left [\sum_{x \in \widetilde{\cV}_n^{(k)}(i, j)}  Z_n^{(k), \phi}(i, j, x)  \right] \\
% & =   \sum_{i=0}^{n^{\delta} - 1} D_n(k,0) \E \left[ \E \left[ \sum_{x \in \widetilde{\cV}_n^{(k)}(i, 1)} Z_n^{(k), \phi}(i, 1, x) \Big | \cG_n  \right]\right]\\
% %& = D_n(k,0) \sum_{i=0}^{n^{\delta} - 1}  \E \left[Z_n^{(k), \phi}(i, 1, \tilde{x}) \right]  \E \left[  \widetilde{N}_n^{(k)}(i, 1) \right] \left(\text{where } \tilde{x} \in \widetilde{V}_n^{(k)}(i, 1) \right )\\
% & = D_n(k,0)  \sum_{i=0}^{n^{\delta} - 1}m_{f_1}^{\phi} \left(a - \frac{(i + 1)a}{n^{\delta}} \right) \E \left[  \widetilde{N}_n^{(k)}(i, 1) \right]. 
% %& = N_k \sum_{i=1}^{n^{\delta} - 1}  m_{f_1}^{\phi} \left(a - \frac{i + 1}{n^{\delta}} \right) \mu_{f_1}^{(k)} \left[\frac{i}{n^{\delta}}, \frac{i + 1}{n^{\delta}} \right]
% \end{align*}
  \begin{align*}
\E \left [\widetilde{Z}_n^{(k), \phi} \right]  &= \sum_{i=0}^{n^{\delta} - 1} \sum_{j=1}^{D_n(k,0)}  \E\left [\sum_{x \in \widetilde{\cV}_n^{(k)}(i, j)}  Z_n^{(k), \phi}(i, j, x)  \right] \\
& =   \sum_{i=0}^{n^{\delta} - 1} D_n(k,0) \E \left[ \E \left[ \sum_{x \in \widetilde{\cV}_n^{(k)}(i, 1)} Z_n^{(k), \phi}(i, 1, x) \Big | \cG_n  \right]\right]\\
& = D_n(k,0)  \sum_{i=0}^{n^{\delta} - 1}m_{f_1}^{\phi} \left(a - (i + 1)a  n^{- \delta} \right) \E \left[  \widetilde{N}_n^{(k)}(i, 1) \right]. 
 \end{align*}

 Here the third equality follows from \chnr{the fact that} $ \widetilde{\cV}_n^{(k)}(i, 1)$ is $ \cG_n$ measurable and for fixed $i$ and, conditional on $ \cG_n$, for each $x \in \widetilde{\cV}_n^{(k)}(i, 1)$, $Z_n^{(k), \phi}(i, j, x)\stackrel{d}{=} Z_{f_1}^\phi \left (a - (i + 1)a / n^{\delta} \right)$. Applying equation \eqref{eqn:vf1-bound}, the error term $\eps_n$ in the statement of the lemma can be bounded as,
 \begin{equation}
 \label{eqn:1257}
    \eps_n \leq 2c D_n(k,0) m_{f_1}(a)\sum_{i=0}^{n^\delta -1}\E \left [N_n^{(k)}(i, 1)  - \widetilde{N}_n^{(k)}(i, 1) \right].
 \end{equation}
 Next using that the total number of descendants of bad vertices in an interval bounds the number of bad vertices in this interval since each bad vertex has at least one child, we get using Lemma \ref{lem:2},
 \begin{align*}
 0  \le \E \left [N_n^{(k)}(i, 1)  - \widetilde{N}_n^{(k)}(i, 1) \right]  = \E[B_n^{(k)}(i, 1)]
  \le \E[R_n^{(k)}(i, j)] \le C_1e^{C_2a} (k+1)  n^{- 2 \delta}.
 \end{align*}
 Using this and \eqref{eqn:vf1-bound} in \eqref{eqn:1257} completes the proof.  \end{proof}
 % where the last inequality comes from Lemma \ref{lem:2}.  Also recall that $\E N_n^{(k)}(i, 1) = \mu_{f_1}^{(k)} \left[\frac{ia}{n^{\delta}}, \frac{(i + 1)a}{n^{\delta}} \right]$. Thus $\left|\E \widetilde{N}_n^{(k)}(i, 1)  - \mu_{f_1}^{(k)} \left[\frac{ia}{n^{\delta}}, \frac{(i + 1)a}{n^{\delta}} \right] \right| \le C k n^{-2 \delta}$.
 %
 % This shows the quantity in LHS of the Lemma can be controlled as follows
 % \begin{align*}
 % LHS  & = D_n(k,0)  \left| \sum_{i=0}^{n^{\delta} - 1}  m_{f_1}^{\phi} \left(a - \frac{(i + 1)a}{n^{\delta}} \right) \left(\mu_{f_1}^{(k)} \left[\frac{ia}{n^{\delta}}, \frac{(i + 1)a}{n^{\delta}} \right]  - \E\widetilde{N}_n^{(k)}(i, 1) \right) \right| \\
 % & \le  D_n(k,0)  \sum_{i=0}^{n^{\delta} - 1}  m_{f_1}^{\phi} \left(a - \frac{(i + 1)a}{n^{\delta}} \right) \left| \E \widetilde{N}_n^{(k)}(i, 1)  - \mu_{f_1}^{(k)} \left[\frac{ia}{n^{\delta}}, \frac{(i + 1)a}{n^{\delta}} \right] \right| \\
 % & \le C  D_n(k,0)  \sum_{i=0}^{n^{\delta} - 1}  m_{f_1}^{\phi} \left(a - \frac{(i + 1)a}{n^{\delta}} \right) k n^{-2 \delta} \\
 % & \le C' k  D_n(k,0)  \sum_{i=0}^{n^{\delta} - 1} n^{-2 \delta} \\
 % & = C'' k D_n(k,0)  n^{- \delta}
 % \end{align*}
 %
 % Where the third inequality comes from replacing each $ m_{f_1}^{\phi} \left(a - \frac{(i + 1)a}{n^{\delta}} \right) $ with $\sup_{t \in [0, a]} m_{f_1}^{\phi}(t) < \infty$.

 \begin{lemma} \label{lem:5}
 There exists a positive constant $a_0<\infty$ independent of $n,\delta$ such that for $k \ge 0$ and $a \le \frac{\delta}{a_0} \log n$,
 $$\E \left [Z_n^{(k), \phi} - \widetilde{Z}_n^{(k), \phi} \right ] \le Ce^{C' a} n^{-\delta} (k+1) D_n(k,0). $$
 \end{lemma}

 \begin{proof}
 \begin{multline}\label{eqn:1834}
 \E \left [Z_n^{(k), \phi} - \widetilde{Z}_n^{(k), \phi} \right ]  \le \E \left [ \sum_{j=1}^{D_n(k,0)}\sum_{i=0}^{n^{\delta} - 1} \sum_{x \in \cV_n^{(k)}(i, j)} Z^{(k), \phi}(i, j, x)\ind\set{B_x} \right ]\\
 = D_n(k,0)  \sum_{i=0}^{n^{\delta} - 1} \E \left [ \sum_{x \in \cV_n^{(k)}(i, 1)} Z^{(k), \phi}(i, 1, x)\ind\set{B_x} \right ],
 \end{multline}
 where $B_x$ is the event that the vertex $x$ is bad namely has one or more descendants in the interval that it was born. Now, \chsb{recalling that $\phi \in \cC$,} note that for a fixed $i$, conditional on the number of births $N_n^{\sss(k)}(i,1)$, we have 
 \begin{equation}
 \label{eqn:stod-bp-yule}
    \sum_{x \in \cV_n^{(k)}(i, 1)} Z^{(k), \phi}(i, 1, x)\ind\set{B_x} \stod \sum_{l=1}^{N_n^{\sss(k)}(i,1)} 2b_{\phi}|\PA^{\sss(l)}[0,a]|\ind\set{\tilde{B}_l},
 \end{equation}
 where $\set{\PA^{\sss(l)}:l\geq 1}$ is a collection of linear $\PA$ branching processes with parameters $\nu=C$ and $\kappa=0$ (independent of $N_n^{\sss(k)}(i,1)$) and 
$\tilde{B}_l:=\set{\left|\PA^{\sss(l)}\left[0,a/n^{\delta}\right]\right|\geq 2},$
 namely the root of $\PA^{\sss(l)}$ has at least one child by time $a/n^{\delta}$ \chsb{(here $C$ can be taken to be the same constant appearing in Assumption \ref{ass:attach-func}(i))}. Using this in \eqref{eqn:1834} implies,
 \begin{equation}
 \label{eqn:1843}
    \E \left [Z_n^{(k), \phi} - \widetilde{Z}_n^{(k), \phi} \right ]\leq 2cD_n(k,0)\sum_{i=1}^{n^{\delta}-1}\E(N_n^{\sss(k)}(i,1))\E(|\PA^{\sss(1)}[0,a]|\ind\set{\tilde{B}_1}). 
 \end{equation}
Conditioning on the number of births $Y(a/n^{\delta})$ of the root of $\PA^{\sss(1)}$ in $[0,a/n^{\delta}]$ and by the Markov property, 
 \[\E(|\PA^{\sss(1)}[0,a]|\ind\set{\tilde{B}_1})\leq \sum_{j=1}^{\infty} \pr\left(Y\left(a n^{ -\delta}\right) = j\right) \E(\PA^{\sss(1),j}[0,a]),\]
 where $\PA^{\sss(1),j}$ is a modified linear PA process with $\nu=C,\kappa=0$ with the modification that the offspring process of the root of $\PA^{\sss(1),j}$ is constructed using attachment function $f(i):=C(j+i+1)$ for $i\geq 0$. Comparing rates, it is easy to see that for each $j\geq 1$, $\PA^{\sss(1),j}[0,a]\stod U_j(a)$, where $U_j(a)$ is constructed by first running a linear PA process $\PA_{\nu,\kappa}$ with $\nu=C$ and $\kappa = Cj$ and then setting $U_j(a) = |\PA_{\nu,\kappa}[0,a]|$. By Lemma \ref{lem:yule-prop} for $Y(a/n^{\delta})$ and Proposition \ref{prop:mean-PA-model} for $\E(U_j(a))$, we get $a_0>0$ such that whenever $a \le \frac{\delta}{a_0} \log n$,
 \begin{equation}
 \label{eqn:2143}
    \E(|\PA^{\sss(1)}[0,a]|\ind\set{\tilde{B}_1})\leq \sum_{j=1}^\infty \left(Ca n^{-\delta}\right)^{j}e^{a(2C+Cj)}\leq Ce^{C'a}n^{-\delta}.
 \end{equation}
In \eqref{eqn:1843}, using this bound and using Corollary \ref{cor:1.5} for $\E(N_n^{\sss(k)}(i,1))$ completes the proof.  \end{proof}

 \begin{lemma} \label{lem:7}
For any $k \ge 0$, whenever $a \le \frac{\delta}{a_0} \log n$,
\begin{multline*}
\chnr{\varpi_n:=\E \left\lvert Z^{\phi}_n -  \sum_{k=0}^{\infty} D_n(k,0) \sum_{i=0}^{n^{\delta} - 1} m^{\phi}_{f_1} \left(a - (i +1)an^{-\delta} \right) \mu_{f_1}^{(k)} \left[ia n^{-\delta}, (i + 1)a n^{-\delta} \right] \right\rvert}\\
    \le Ce^{C'a}\left(n^{1-\delta} + \sqrt{n} + n^{-\delta/2}\left(\sum_{k=1}^{\infty} (k+1)^2 D_n(k,0)\right)^{1/2}\right).
\end{multline*}
 %Let $\zeta > 0$ be such that \ref{assmpt:pk_grows_slowly} holds for $f_0$.
 % Let $Z_n^{\phi} := \sum_{k=1} Z_n^{(k), \phi}$.

 % $$P\left( \left | \frac{Z^{\phi}_n}{n} -  \sum_{k=1}^{\infty} \frac{D_n(k,0)}{n} \sum_{i=0}^{n^{\delta} - 1} m^{\phi}_{f_1} \left(a - \frac{(i +1)a}{n} \right) \mu_{f_1}^{(k)} \left[\frac{ia}{n^{\delta}}, \frac{(i + 1)a}{n^{\delta}} \right] \right | > \epsilon \right )  < \frac{C}{n^{\zeta \wedge \delta \wedge 1}}$$
 \end{lemma}

 \begin{proof}
We can write $\varpi_n:= \varpi_n^{\sss(1)}+\varpi_n^{\sss(2)}+\varpi_n^{\sss(3)}$ where $\varpi_n^{\sss(1)}:={Z^{\phi}_n-\widetilde{Z}^{\phi}_n}$, $\varpi_n^{\sss(2)}:= \widetilde{Z}^{\phi}_n -\E(\widetilde{Z}^{\phi}_n)$ and 
\chnr{ \[ \varpi_n^{\sss(3)}:=\E(\widetilde{Z}^{\phi}_n) - \sum_{k=0}^{\infty} D_n(k,0) \sum_{i=0}^{n^{\delta} - 1} m^{\phi}_{f_1} \left(a - (i +1)an^{-\delta} \right) \mu_{f_1}^{(k)} \left[ia n^{- \delta}, (i + 1)a n^{-\delta} \right].  \]}
 Now fix $\eps>0$. Using Lemma \ref{lem:5} we get,
 \begin{equation}
 \label{eqn:wn-1}
    \E(|\varpi_n^{\sss(1)}|)\leq Ce^{C'a} n^{-\delta} \sum_{k=0}^{\infty} (k+1) D_n(k,0)  \le 2\gamma Ce^{C'a}n^{1-\delta},
 \end{equation}
 since $\sum_{k=1}^{\infty} (k+1) D_n(k,0) = 2\gamma n - 1$ for tree $\cT_{n\gamma}$. Next using Lemma \ref{lem:3} and Jensen's inequality,
 \begin{align} \label{eqn:wn-2}
 \E \left( |\varpi_n^{\sss(2)}| \right)
  &\le Ce^{C'a}\left(\sum_{k=1}^{\infty} \left((k+1)^2n^{-\delta} + (k+1)\right) D_n(k,0)\right)^{1/2}\nonumber\\
  &\le Ce^{C'a}\left(n^{-\delta/2}\left(\sum_{k=1}^{\infty} (k+1)^2 D_n(k,0)\right)^{1/2} + \sqrt{n}\right). 
 \end{align}
 Finally using Lemma \ref{lem:4} gives,
 \begin{equation}
 \label{eqn:wn-3}
    |\varpi_n^{\sss(3)}|\leq Ce^{C'a}\sum_{k=0}^{\infty} (k+1) D_n(k,0) n^{- \delta} \le Ce^{C'a}n^{1-\delta}.
 \end{equation}
 Combining \eqref{eqn:wn-1}, \eqref{eqn:wn-2} and \eqref{eqn:wn-3} completes the proof.  \end{proof}

The next lemma establishes Lipschitz continuity of $m^{\phi}_{f_1}(t)$ in $t$ for any $\phi \in \cC$.
\begin{lemma}\label{lipcont}
For any $k \ge 0$ and any $\eta \in [0,1]$,
$
\sup_{t \in [0,a]}|m^{\phi}_{f_1}(t+\eta) - m^{\phi}_{f_1}(t)| \le Ce^{C'a}\eta.
$
\end{lemma}
\begin{proof}
Let $\bar\tau_1$ be the time of the first birth for the branching process with attachment function $f_1$. For any $t \in [0,a]$ and $\eta \in [0,1]$, using the Markov property at time $\eta$, we obtain
\begin{multline}\label{meancont}
m_{f_1}^{\phi}(t + \eta)  = \E\left[ Z_{f_1}^{\phi}(t + \eta)\right] 
 = \E\left[ Z_{f_1}^{\phi}(t + \eta) \mathds{1}\left(\bar\tau_1 > \eta\right)\right] +\E\left[ Z_{f_1}^{\phi}(t+ \eta) \mathds{1}\left(\bar\tau_1 \le \eta\right)\right] \\
  = \E\left[ Z_{f_1}^{\phi}(t)\right] \E \left[ \mathds{1}\left(\bar\tau_1 > \eta\right)\right] + \E\left[ Z_{f_1}^{\phi}(t+ \eta) \mathds{1}\left(\bar\tau_1 \le \eta\right)\right]
 %\E \left[ \mathds{1}\left(\bar\tau_1 \le \delta\right) \E\left[ Z_f^{\phi}(t+ \delta - \bar\tau_1) \Big | \bar\tau_1\right] \right] \\
  = m_{f_1}^{\phi}(t)(1 - \pr\left(\bar\tau_1 \le \eta\right)) + \E\left[ Z_{f_1}^{\phi}(t+ \eta) \mathds{1}\left(\bar\tau_1 \le \eta\right)\right].
\end{multline}

Using the strong Markov property at $\bar\tau_1$, we can write the second term above as $\E\left[ Z_{f_1}^{\phi}(t+ \eta) \mathds{1}\left(\bar\tau_1 \le \eta\right)\right] = \E\left[ \E\left(Z_{f_1}^{\phi}(t+ \eta) \mid \cF_{\bar\tau_1}\right) \mathds{1}\left(\bar\tau_1 \le \eta\right)\right]$, where $\cF_{\bar\tau_1}$ denotes the associated stopped sigma field. Note that at time $\bar\tau_1$, there are two vertices, one with degree one and the other with degree zero. Thus, conditional on $\cF_{\bar\tau_1}$, for $i=1,2$, if $U_i(t)$ is distributed as the size of the linear PA process $\PA_{\nu,\kappa_i}$ with $\nu = C$ and $\kappa_i=C(i-1)$ at time $t$ (where $C$ is the same constant appearing in Assumption \ref{ass:attach-func}(i)), we have 
$$
\E\left(Z_{f_1}^{\phi}(t+ \eta) \mid \cF_{\bar\tau_1}\right) \le  2c\E\left(Z_{f_1}(t+ \eta) \mid \cF_{\bar\tau_1}\right) \le 2c\E(U_1(a+1) + U_2(a+1)) \le Ce^{C'a}
$$
for constants $C,C'$ not depending on $\eta, a,t$, where we used Proposition \ref{prop:mean-PA-model} to get the last inequality. Using this bound and \eqref{eqn:vf1-bound} in \eqref{meancont}, we obtain
\begin{equation*}
| m_{f_1}^{\phi}(t + \eta) - m_{f_1}^{\phi}(t) |  =  \left| - m_{f_1}^{\phi}(t)\pr\left(\bar\tau_1 \le \eta\right) +Ce^{C'a}\pr\left(\bar\tau_1 \le \eta\right) \right| \le 2Ce^{C'a}\pr\left(\bar\tau_1 \le \eta\right)\le C''e^{C'a}\eta
% \le 2Ce^{C'a} (1 - e^{- f_1(0) \eta}) \le C''e^{C'a}\eta
\end{equation*}
for a constant $C''$ not depending on $\eta,a,t$, where the last equality comes from the fact $\bar\tau_1 \sim \text{Exp}(f_1(0))$. 
\end{proof}
% Now recall $\lambda_k^{\phi}$ defined at the beginning of this Section. 
\begin{lemma} \label{lem:8}
Recall $\lambda_{k}^\phi = \int_0^{a} m_{f_1}^\phi(a- s) \mu^{(k)}_{f_1}(ds)$. For any $k \ge 0$, whenever $a \le \frac{\delta}{a_0} \log n$,
 $$ \E\left\lvert Z^{\phi}_n -  \sum_{k=0}^{\infty} D_n(k,0)  \lambda_k^\phi\right\rvert \le Ce^{C'a}\left(n^{1-\delta} + \sqrt{n} + n^{-\delta/2}\left(\sum_{k=1}^{\infty} (k+1)^2 D_n(k,0)\right)^{1/2}\right).$$ % \int_0^a m_{f_1}^{\phi} (a - s) \mu_{f_1}^{(k)}(ds)
 \end{lemma}

 \begin{proof}
 By Lemma \ref{lem:7} it is enough to show, for positive constants $C, C'$ not depending on $a,n,\delta$,
\chnr{ \begin{equation}
 \label{eqn:sum-int-zero}
    \varpi_n^*:= \left\lvert\sum_{k=0}^{\infty} D_n(k,0)\lambda_k^\phi - \sum_{k=0}^{\infty} D_n(k,0) \sum_{i=0}^{n^{\delta} - 1} m^{\phi}_{f_1} \left(a - \frac{(i +1)a}{n^{\delta}} \right) \mu_{f_1}^{(k)} \left[\frac{ia}{n^{\delta}}, \frac{(i + 1)a}{n^{\delta}} \right] \right\rvert \le Ce^{C'a} n^{1-\delta}.
 \end{equation}}

% Since $m_{f_1}^{\phi} (\cdot)$ is continuous (and thus uniformly continuous on $[0,a]$), $\eps_n:=\sup_{|s-t|\leq a/n^\delta, s,t\in [a]}|m_{f_1}^{\phi}(s)- m_{f_1}^{\phi}(s)| \to 0$ as $n\to\infty$. Thus for any $k\geq 0$,  % we can find a sequence $\epsilon_n \downarrow 0$ such that if $|s - t| < \frac{1}{n}$ then $|m_{f_1}^{\phi} (s) - m_{f_1}^{\phi} (t)| < \epsilon_n$. The following bound connects terms in Lemma \ref{lem:7} to terms in \ref{lem:8}
Using Lemma \ref{lipcont},
 \begin{align*}
\varpi_n^* & \le \chnr{\sum_{k=0}^{\infty}D_n(k,0)\int_0^a  \sum_{i=0}^{n^{\delta} - 1} \left| m_{f_1}^{\phi} (a - s) - m^{\phi}_{f_1} \left(a - \frac{(i +1)a}{n^{\delta}} \right) \right | \mathds{1} \left(s \in\left[\frac{ia}{n^{\delta}}, \frac{(i + 1)a}{n^{\delta}} \right] \right) \mu_{f_1}^{(k)}(ds)} \\
 & \le Ce^{C'a}n^{-\delta}\sum_{k=0}^{\infty}D_n(k,0)\int_0^a  \sum_{i=0}^{n^{\delta} - 1} \mathds{1} \left(s \in\left[\frac{ia}{n^{\delta}}, \frac{(i + 1)a}{n^{\delta}} \right] \right) \mu_{f_1}^{(k)}(ds) = Ce^{C'a}n^{-\delta} \sum_{k=0}^{\infty}D_n(k,0)\mu_{f_1}^{(k)}[0, a]\\
 & \le (Ce^{C'a})^2an^{-\delta} \sum_{k=0}^{\infty}(k+1)D_n(k,0) = (Ce^{C'a})^2an^{-\delta}(2\gamma n-1),
 \end{align*}
 where the last inequality uses Lemma \ref{lem:1} and the last equality uses $\sum_{k=0}^{\infty}(k+1)D_n(k,0) = 2\gamma n-1$.  \end{proof}

 \begin{lemma} \label{lem:9}
 Let $\phi \in \cC$. As $n \rightarrow \infty$, $ n^{-1}\sum_{k=1}^{\infty} D_n(k,0) \lambda_k^\phi \overset{a.s.}{\longrightarrow} \gamma  \sum_{k=1}^{\infty}  p_k^0 \lambda^{\phi}_k.$
% $$  \sum_{k=1}^{\infty} \frac{D_n(k,0)}{n} \lambda_k^\phi(a) \overset{a.s.}{\longrightarrow} \gamma  \sum_{k=1}^{\infty}  p_k^0 \lambda^{\phi}_k(a).$$ % m_{f_1}^{\phi} (a - s) \mu_{f_1}^{(k)}(ds) 
 \end{lemma}

 \begin{proof}
    % First note that by

 Let $\chi$ be the characteristic $ \chi(t) = \sum_{k=0}^{\infty}\lambda_k^\phi \ind\set{\xi_{f_0}(t) = k}$.  Note by equation \eqref{eqn:vf1-bound} and Lemma \ref{lem:1} that $\lambda_k^\phi \le Ce^{C'a} (k+1)$ and thus $\chi \in \cC$. Now apply Lemma \ref{lem:deg_dist_quad_conv} (i).  \end{proof}

 \noindent{\bf Completing the proof of Theorem \ref{thm:1}:}  By letting $\delta \rightarrow \infty$ \chnr{and} keeping $n \ge 2$ fixed in Lemma \ref{lem:8}, the first claim follows. Lemma \ref{lem:9} then gives the second claim.

 \section{Proofs: Sup-norm convergence of degree distribution for the standard model}\label{supdoc}

 %\subsection{Temporary technical estimates}

 %\todo[inline]{Temporary tech estimates for sup norm convergence. Will be moved to correct place later. }
\chsb{We will assume throughout this section that $f_0,f_1$ satisfy Assumption \ref{ass:attach-func}}.

 \subsection{Proof of Theorems \ref{supthm} and \ref{thm:standard}}
\chnr{
Here we prove convergence results for the empirical degree distribution post change-point. As before, time starts at the change point, i.e. $t=0$ represents the time $T_{\gamma n}$. We focus on the characteristic $\phi(t) = \ind\set{\xi_{f_1}(t) = k}$ for \bbb{fixed} $k \ge 0$ and denote the corresponding $Z^{\phi}_{f_1}$ and $m^{\phi}_{f_1}$ by $Z^{(k)}_{f_1}$ and $m^{(k)}_{f_1}$. $\BP_n(t)$ denotes the branching process at time $t$ (i.e. $t$ time units after the change point).
 }
 
 \subsubsection{Notation}
 We will use the following notation for fixed $t\geq 0$ in this section.
 \begin{enumeratei}
 \item Recall $n\gamma$ \chnr{is} the number of vertices born \textbf{before} the change point. Let $Z_{AC,n}(t) := $ number of vertices at time $t$ who were born \textbf{after}  the change point. \chnr{$Z_n(t) := n\gamma + Z_{AC,n}(t)$ denotes the total number of vertices \bbb{in the system} at time $t$}. 

 \item Let $\cD^{BC}_n(k, t)$ be the set of vertices with degree $k$ at time $t$ who were born \textbf{before} the change point $T_{\gamma n}$. Let  $D^{BC}_n(k, t) = |\cD^{BC}_n(k, t)|$. Similarly, let $\cD^{AC}_n(k, t)$ be the set of vertices with degree $k$ at time $t$ who were born \textbf{after} the change point. Let $D^{AC}_n(k, t) = |\cD^{AC}_n(k, t)|$.  Let $D_n(k,t) = D^{BC}_n(k, t) + D^{AC}_n(k, t)$ be the total number of vertices with degree $k$ \chsb{at time $t$}.

 % \item Let $N_{k, BC}(t) = $ number of vertices with degree $k$ at time $t$ who were born \textbf{before} the change point. Let $N_{k, AC}(t) = $ number of vertices with degree $k$ at time $t$ who were born \textbf{after} the change point. Let $N_k(t) = N_{k, BC}(t) + N_{k, AC}(t)$ be the total number of vertices at time $t$. 

 % \item Let $N_k^{cp}$ be the number of vertices who have degree $k$ at the change point (i.e. $N_k^{cp} = N_{k, BC}(\tcp) )$

 \item Let $\lambda^{AC}_{\ell}(t) = \int_0^{t} m_{f_1}(t  - s) \mu^{(\ell)}_{f_1}(ds)$ and $\lambda^{AC, (k)}_{\ell}(t) = \int_0^{t} m^{(k)}_{f_1}(t  - s) \mu^{(\ell)}_{f_1}(ds)$. Let $\lambda_{\ell}(t) = 1+\lambda^{AC}_{\ell}(t)$ and $\lambda^{(k)}_{\ell}(t) =  \pr \left( \xi_{f_1}^{(\ell)}(t) = k - \ell \right) + \lambda^{AC, (k)}_{\ell}(t)$.

%  \item Let $q_k(t) = \pr\left( \text{Child is born in [0, t] if root has degree $k$ and attachment function is } f_1 \right)$.
 \item Let $q_k(t): = \pr\left( \xi_{f_1}^{(k)}(t) > 1  \right)$.
 \end{enumeratei}
  The following is the main theorem proved in this section. \bbb{As will be seen below, Theorems \ref{supthm} and \ref{thm:standard} are consequences of this theorem}. 
 \begin{theorem} \label{thm:2}
 For any $k \ge 0$, $a > 0$, $\epsilon > 0$, \chnr{as $n \rightarrow \infty$},
 $$\pr\left( \sup_{t \in [0, a]} \left| D_n(k, t) - n \sum_{\ell=0}^{\infty}\gamma p_{\ell}^{0}\lambda_{\ell}^{(k)}(t) \right| > \epsilon n \right) \to 0, \  \  \
\pr\left( \sup_{t \in [0, a]} \left|Z_n(t) - n \sum_{\ell=0}^{\infty}\gamma p_{\ell}^{0}\lambda_{\ell}(t) \right| > \epsilon n  \right) \to 0.$$
 \end{theorem}
 Assuming the above result for the time being, we now describe how \bbb{Theorem \ref{thm:2}} (coupled with a technical continuity result, Lemma \ref{lem:lam-t-ts-diff}) \chnr{is enough} to prove Theorems \ref{supthm} and \ref{thm:standard}. Recall for $m \ge 1$, $T_m =\inf\set{t\geq 0: |\BP_n(t)| =m}$.

  \begin{corollary} \label{cor:timing}
 \chnr{Let $G(t) :=  \sum_{\ell=0}^{\infty} p_{\ell}^{0}\lambda^{AC}_{\ell}(t)$, $t \ge 0$. For any $s \in [\gamma, 1]$, let $a_s$ be the unique solution to $G(a_s) = (s - \gamma)/\gamma$. Then for any $s \in [\gamma, 1]$, $\sup_{t \in [\gamma, s]}\left|T_{\lfloor t n \rfloor} - a_t\right| \probc 0$ as $n\to \infty$}.
   \end{corollary}

\begin{proof}
As $f_1$ is a strictly positive function, it is easy to see that $G(t)$ is strictly increasing in $t$ and $G(\gamma) = 0$. By Lemma \ref{lem:lam-t-ts-diff} proved below, $G$ (hence $G^{-1}$) is continuous. Moreover since \chnr{$m_{f_1}(t) \ge 1$ and $\lambda^{AC}_{\ell}(t) \ge \mu_{f_1}^{(\ell)}(t) \uparrow \infty$} we see $G(t) \to \infty$ as $t \to \infty$. Therefore $G(a_s) = \frac{s - \gamma}{\gamma}$ has a unique solution for $s \in [\gamma, 1]$.

Next fix $s \in [\gamma, 1]$ and let $a_s$ be as above. For any $\eta>0$, choosing $\epsilon = \frac{G(a_s+\eta) - G(a_s)}{2\gamma}$, the second assertion in Theorem \ref{thm:2} readily implies
$
\pr(Z_n(a_s + \eta) > sn + 1) \rightarrow 1.
$
Similarly, it follows that
$
\pr(Z_n(a_s - \eta) < sn - 1) \rightarrow 1.
$ 
Therefore, $T_{\lfloor sn \rfloor} \probc a_s$. From this, Theorem \ref{thm:2}, \chnr{and the definition of $G$, $\lambda^{(\ell)}$}, we conclude that\\
$
\frac{1}{n}\sup_{t \in [0, T_{\lfloor sn \rfloor}]}\left|Z_n(t) - \gamma n \left(1 + G(t)\right)\right| \probc 0
$ which implies $ \sup_{t \in [\gamma, s]}\left|\frac{t-\gamma}{\gamma} - G(T_{\lfloor tn \rfloor})\right| \probc 0.$ By continuity of $G^{-1}$, this implies  $\sup_{t \in [\gamma, s]}\left|G^{-1}\left(\frac{t-\gamma}{\gamma}\right) - T_{\lfloor tn \rfloor}\right| \probc 0$ which proves the corollary.
\end{proof}

 \begin{proof}[Proof of Theorem \ref{supthm}]
Fix $s \in [\gamma, 1]$. It follows from Lemma \ref{lem:lam-t-ts-diff} and Corollary \ref{cor:lamLK-t-ts-diff} proved below that $t \mapsto \Phi_t\left(\mathbf{p}^0\right)$ is continuous and hence, from Corollary \ref{cor:timing} for each fixed $k\geq 0$,
\begin{equation}\label{st1}
\sup_{t \in [\gamma, s]} \left|\left(\Phi_{T_{\lfloor tn \rfloor}}\left(\mathbf{p}^0\right)\right)_k - \left(\Phi_{a_t}(\mathbf{p^0})\right)_k\right| \probc 0.
\end{equation}
It is easy to see that
\begin{multline}\label{st2}
 \sup_{t \in [\gamma, s]} \left| \frac{D_n(k, T_{\lfloor tn \rfloor})}{tn} -  \left(\Phi_{T_{\lfloor tn \rfloor}}(\mathbf{p^0})\right)_k\right|\\
 \le \frac{1}{\gamma n}\left(\sup_{t \in [0, T_{sn}]} \left| D_n(k, t) - n \sum_{\ell=0}^{\infty}\gamma p_{\ell}^{0}\lambda_{\ell}^{(k)}(t) \right| + \sup_{t \in [0, T_{sn}]} \left| Z_n(t) - n \sum_{\ell=0}^{\infty}\gamma p_{\ell}^{0}\lambda_{\ell}(t) \right|\right) \probc 0.
\end{multline}
The theorem follows from \eqref{st1} and \eqref{st2}. \end{proof}
 \begin{proof}[Proof of Theorem \ref{thm:standard}]
 Follows immediately from Theorem \ref{supthm}.
 \end{proof}
 %\begin{proof}
 %Theorem \ref{thm:2} follows from Lemmas \ref{lem:Nt-sumNlam-close} and \ref{lem:sumNlam-sumPllam-close} given below.
 %\end{proof}

 %\subsubsection{Some technical lemmas}
\bbb{{\bf Proof of Theorem \ref{thm:2}:} The rest of this Section is devoted to the proof of this result. We start with a brief outline of the proof. }   \bbb{We start by partitioning the interval $[0,a]$ into subintervals $[t_j, t_{j+1}]_{1 \le j \le an^{\widetilde{\theta}} - 1}$ and showing by means of some continuity estimates that $D_n(k,t)$ and $Z_{n}(t)$ do not vary too much as $t$ varies within each such subinterval (see Lemmas \ref{lem:sup-sup-Nt-Ntj} and \ref{lem:Z_t_js_close}). We then use Theorem \ref{thm:1} (for vertices born post-change point) and a variance computation \eqref {sub2} (for vertices born pre-change point) to show that $\left| D_n(k, t) - \sum_{\ell = 0}^{\infty} D_n(\ell, 0) \lambda_{\ell}^{(k)}(t)  \right |$ and $\left |Z_n(t) - \sum_{\ell = 0}^{\infty} D_n(\ell, 0) \lambda_{\ell}(t)  \right |$ are small for each $t=t_j$. This, combined with the continuity estimates, implies that the above quantities are small uniformly for all $t \in [0,a]$ for appropriately chosen partitions. Finally, a law of large numbers type argument along with continuity estimates is used to show $\left| \frac{1}{n} \sum_{\ell=0}^{\infty} D_n(\ell, 0) \lambda_{\ell}^{(k)}(t) - \gamma \sum_{\ell=0}^{\infty} p_{\ell}^0 \lambda_{\ell}^{(k)}(t)  \right|$ and
$\left| \frac{1}{n} \sum_{\ell=0}^{\infty} D_n(\ell, 0) \lambda_{\ell}(t) - \gamma \sum_{\ell=0}^{\infty} p_{\ell}^0 \lambda_{\ell}(t)  \right|$ are uniformly small in $t$, which proves Theorem \ref{thm:2}. 
}

For the remaining portion of this  section $C,C',C'', n_0$ will denote generic positive constants not depending on $n,a,k,\ell,t$ whose values might change from line to line and between inequalities. 

 %We next give some minor technical Lemmas related to the above definitions.

 \begin{lemma}\label{qdec} $q_k(t) \le C (k+1) t$  where $C$ is the constant appearing in Assumption \ref{ass:attach-func}{(i)} on $f_1$.
 \end{lemma}
 \begin{proof}
 Let $\bar\tau_1^k$ be the time of the first \chnr{birth} to a vertex started with degree $k$. Note $\bar\tau_1^k \sim \text{Exp}(f_1(k))$. Thus $ \pr(\bar\tau_1^k < t) = 1 - e^{- f_1(k) t} \le f_1(k) t \le C (k+1) t.$ The final inequality comes from Assumption \ref{ass:attach-func}{(i)} on $f_1$. \end{proof}

 \begin{lemma} \label{lem:lam-t-ts-diff}
 For any $\ell, k \ge 0$ and $t, t + s \le a$,
 $$|\lambda_{\ell}(t + s) - \lambda_{\ell}(t)| \le Ce^{C'a} (\ell + 1) s, \  \  \ |\lambda_{\ell}^{AC,(k)}(t + s) - \lambda_{\ell}^{AC,(k)}(t)| \le Ce^{C'a} (\ell + 1) s.$$
 \end{lemma}

 \begin{proof}
 We will only prove the first inequality. The second one follows similarly.
 \begin{multline*}
 |\lambda_{\ell}(t + s) - \lambda_{\ell}(t)|  \le \int_{0}^{t} \left |m_{f_1}(t  + s - x) - m_{f_1}(t  - x)\right| \mu^{(\ell)}_{f_1}(dx) + \int_t^{t+s}m_{f_1}(t  + s - x)\mu^{(\ell)}_{f_1}(dx)\\
  \le Ce^{C'a}s \E\left[\xi_{f_1}^{(\ell)}[0, t]\right] + Ce^{C'a} m_{f_1}(t+s)\E\left[\xi_{f_1}^{(\ell)}[t, t+s]\right] \le Ce^{2C'a}a(\ell+1)s +  Ce^{2C'a}(\ell+1)s
 \end{multline*}
 where the second inequality uses Lemma \ref{lipcont} and the third inequality uses Lemma \ref{lem:1} and \eqref{eqn:vf1-bound}. \end{proof}

 \begin{lemma}\label{lem:Pt-ts-diff}
 For \chnr{$k \ge \ell$} and $t, t+s \le a$,
 $\left| \pr \left( \xi_{f_1}^{(\ell)}(t+s) = k - \ell \right) -  \pr \left( \xi_{f_1}^{(\ell)}(t) = k - \ell \right) \right | \le Ce^{C'a} (k+1) s.$
 \end{lemma}

 \begin{proof}
 %If $k < \ell$ both probabilities are equal to zero an the result is trivially true so WLOG assume $k \ge \ell$. 
 We prove this inequality in two steps. By repeated applications of the Markov property, Markov's inequality and Lemma \ref{lem:1},
  \begin{align*}
 &\pr \left( \xi_{f_1}^{(\ell)}(t + s) = k - \ell \right) \\
 & = \sum_{d=0}^{k - \ell - 1} \pr \left( \xi_{f_1}^{(\ell)}(t) = d \right)\pr \left( \xi_{f_1}^{(d + \ell)}( s) = k - \ell  - d\right)  +\pr \left( \xi_{f_1}^{(\ell)}(t) = k-\ell \right)\pr \left( \xi_{f_1}^{(k)}( s) =0\right) \\
  &\le \sum_{d=0}^{k - \ell - 1} \pr \left( \xi_{f_1}^{(\ell)}(t) = d \right) \E \xi_{f_1}^{(d + \ell)}(s)   +\pr \left( \xi_{f_1}^{(\ell)}(t) = k-\ell \right)\\
  &\le  Ce^{C'a}s\left(\E\left(\xi_{f_1}^{(\ell)}(t)\right) + \ell + 1\right) + \pr \left( \xi_{f_1}^{(\ell)}(t) = k-\ell \right)\\
  &\le  C''e^{2C'a} (\ell + 1) s+ \pr \left( \xi_{f_1}^{(\ell)}(t) = k-\ell \right).
 \end{align*}
We now show the opposite inequality. 
$$
 \pr \left( \xi_{f_1}^{(\ell)}(t + s)   = k - \ell  \right)   \ge \pr \left( \xi_{f_1}^{(\ell)}(t) = k - \ell   \right) \pr \left( \xi_{f_1}^{(k)}(s)  = 0\right) = \pr \left( \xi_{f_1}^{(\ell)}(t) = k - \ell   \right) \left(1- \pr \left( \xi_{f_1}^{(k)}(s)  \ge 1 \right)\right)
$$
 Thus
 \begin{align*}
 \pr \left( \xi_{f_1}^{(\ell)}(t + s)   = k - \ell  \right)  - \pr \left( \xi_{f_1}^{(\ell)}(t) = k - \ell \right) &\ge -\pr \left( \xi_{f_1}^{(\ell)}(t) = k - \ell   \right) \pr \left( \xi_{f_1}^{(k)}(s)  \ge 1 \right)\\
 & \ge -\E  \xi_{f_1}^{(k)}(s) \ge - Ce^{C'a} (k+1) s.
 \end{align*}
The second inequalityuses Markov's inequality and the last inequality comes from Lemma \ref{lem:1}.
\end{proof}

 An immediate consequence of Lemmas \ref{lem:lam-t-ts-diff} and \ref{lem:Pt-ts-diff} is
 \begin{corollary}\label{cor:lamLK-t-ts-diff}
 For any $k, \ell>0$ and $t, t + s < a$,
 $| \lambda_{\ell}^{(k)}(t+s)  - \lambda_{\ell}^{(k)}(t) | \le Ce^{C'a}(k + \ell + 2) s.$
 \end{corollary}

 \begin{corollary} \label{cor:sumNl-lam-t-ts-diff}
 For any $k$ and $t, t + s < a$,
 $\sum_{\ell=0}^{\infty} D_n(\ell, 0) |\lambda^{(k)}_{\ell}(t+s)  -\lambda^{(k)}_{\ell}(t) | \le Ce^{C'a}(k+ 3) s n.$
 \end{corollary}

 \begin{proof}
 By the above Corollary \ref{cor:lamLK-t-ts-diff} (with k fixed)
 \begin{align*}
 \sum_{\ell=0}^{\infty} D_n(\ell, 0) | \lambda^{(k)}_{\ell}(t)  -\lambda^{(k)}_{\ell}(t + s) | \le Ce^{C'a} s \sum_{\ell=0}^{\infty}(k+ \ell +2)D_n(\ell, 0) \le  Ce^{C'a}(k+ 3) s \gamma n
 \end{align*}
 since $\sum_{\ell=0}^{\infty} D_n(\ell, 0) = \gamma n$ and $\sum_{\ell=0}^{\infty} \ell D_n(\ell, 0) = \gamma n - 1$. \end{proof} 

 %\subsubsection{Proof of Theorem \ref{thm:2}}
\chnr{\textbf{ For the rest of this section, unless specified otherwise, we always work conditional on $\cF_n(0)$}} so that expectation operations such as $\pr(\cdot)$, $\E(\cdot)$ and $\var(\cdot)$ in the ensuing results mean $\pr(\cdot | \cF_n(0))$, $\E(\cdot|\cF_n(0))$ and $\var(\cdot|\cF_n(0))$ respectively.
 %First we make some constructions which will be used in the next couple Lemmas. \\

 We will use Theorem \ref{thm:1} crucially in what follows for two significant characteristics. Taking $\phi(t) = \ind\set{t \ge 0}$ in Theorem \ref{thm:1}, there exist deterministic positive constants $C, C'<\infty$ independent of $a,n$ such that for every $n \ge 2$,
 \begin{equation} \label{eq:ZAC-Nlam-close}
 \sup_{t \in [0, a]} \E  \Big| Z_{AC,n}(t) - \sum_{k=0}^{\infty} D_n(k, 0)  \lambda_k^{AC}(t) \Big| < C e^{C'a}\sqrt{n}.
 \end{equation}
 Taking any $k \ge 0$ and setting $\phi(t) = \ind\set{\xi_{f_1}(t) = k}$ in Theorem \ref{thm:1}, there exist deterministic positive constants $C, C'<\infty$ independent of $a,n, k$ such that for every $n \ge 2$,
 \begin{equation} \label{eq:ZAC-Nlam-close2}
 \sup_{t \in [0,a]} \E \Big| D_n^{AC}(k,t) - \sum_{\ell=0}^{\infty} D_n(\ell, 0) \lambda^{AC,(k)}_\ell(t) \Big| < Ce^{C'a} \sqrt{n}.
 \end{equation}
 Take any $\widetilde{\theta} \in (0, 1/2)$.
 Take $\omega \in (0,1)$ such that $\omega > \max\left(1 - \widetilde{\theta}, \frac{1}{2} + \widetilde{\theta}\right)$. Now let $\{t_i\}_{i=0}^{n^{\widetilde{\theta}} - 1}$ be an equispaced partition of $[0, a]$ of mesh $a n^{-\widetilde{\theta}}$.

 \begin{lemma} \label{lem:sup-sup-Nt-Ntj}
 Let $\{t_j\}, \widetilde{\theta}$ and $\omega$ be as above. Fix $\epsilon \in (0,1)$ and $k$. Then we have
 $$\sum_{j = 0}^{n^{\widetilde{\theta}}  - 1} \pr \left( \sup_{t \in [t_j, t_{j+1}]} | D_n(k, t) - D_n(k, t_j)| > \epsilon n^{\omega} \right) \le Ce^{C'a}\epsilon^{-2} n^{-(\omega - \widetilde{\theta} - \frac{1}{2} )}.$$
 %where $C$ is a constant.
 \end{lemma}

 \begin{proof}
 Condition on $\cF_n(t_j)$. Fix $j$ and consider $t \in [t_j, t_{j + 1}]$. We clearly have the following lower bound on $D_n(k, t)$:
 $$
 D_n(k, t) \ge D_n(k, t_j) - Y_1
 $$
 where $Y_1$ is the number of degree $k$ vertices at time $t_j$ which have given birth by time $t_{j+1}$. Note that
 $
 Y_1 \equald \text{Bin}\left(D_n(k, t_j), q_{k}(a n^{-\widetilde{\theta}})\right).
 $
 %Equation \ref{eq:bin_lb_N0} holds by the following argument. For any $s>0$, define the random variable $G_k(s) \sim \text{Bin}\left(D_n(k, t_j), 1 - q_{k}(s)  \right) $. From the definitions of $D_n(k, t_j)$ and $q_k$ and the Markov property we see that $G_k(s)$ is distributed as the number of vertices who had degree $k$ at time $t_j$ and have not given birth in the time interval $[t_j, t_j + s]$. Thus $G_k(t - t_j) \lestoch D_n(k, t)$ for all $t \in [t_j, t_{j + 1}]$.  Furthermore, from the definitions of $q_k(s)$ and $G_k(s)$, we see that $s_1 \le s_2 \implies q_k(s_1) \le q_k(s_2) \implies G_k(s_1) \gestoch G_k(s_2)$.
 We also have the following upper bound on $D_n(k, t)$:
 \begin{equation} \label{eq:bin_lb_N}
 D_n(k, t) \le \left(Z_{AC,n}(t_{j+1}) - Z_{AC,n}(t_j) \right) + Y_2 + D_n(k, t_j)
 \end{equation}
 where $Y_2$ denotes the number of vertices existing at time $t_j$ of degree \chnr{strictly} less than $k$ which have given birth by time $t_{j+1}$. Note that
 $
 Y_2 \equald \sum_{\ell=0}^{k - 1}\text{Bin}\left(D_n(\ell, t_j), q_{\ell} \left(a n^{- \widetilde{\theta}} \right)\right).
 $
 To see this upper bound, note that the degree $k$ vertices at time $t$ originate from vertices either existing at time $t_j$ or new vertices born in the time interval $[t_j, t]$. The latter is bounded by $Z_{AC,n}(t_{j+1}) - Z_{AC,n}(t_j)$, namely, the total number of new births in the time interval $[t_j, t_{j+1}]$. The former is bounded by the sum of the number of vertices which are of degree $k$ at time $t_j$ and have not given birth by time $t$ (which is bounded by $D_n(k, t_j)$) and the number of vertices of lower degree at time $t_j$ which have grown to degree $k$ at time $t$ (which is bounded by $Y_2$).
 %Note $D_n(k, t) - D_n(k, t_j) = D^{bad}_n(k, t) + \sum_{\ell=0}^{k - 1}D_n({\ell \to k}, t)$ where $D^{bad}_n(k, t)$ is the number of (bad) vertices who were degree $k$ at time $t$ and are no longer degree $k$ (i.e. gave birth) and $D_n({\ell \to k}, t)$ is equal to the number of vertices who were degree $\ell$ at time $t_j$ and are now degree $k$ at time $t$.  Note $D^{bad}_n(k, t)$ is bounded by the total number of births that happened in $[t_j, t_{j+1}]$ so $D^{bad}_n(k, t) \lestoch Z_n(t_{j+1}) - Z_n(t_j)$. Note $D_n({\ell \to k}, t)$ is bounded by the number of vertices who were degree $\ell$ at time $t_j$ and gave birth at least once by time $t$. This is again bounded by the same vertices who gave birth at least once by time $t_{j+1}$ (i.e. after $a n^{-\widetilde{\theta}}$ time). Finally, this quantity is distributed as $\text{Bin}\left(D_n(\ell, t_j), q_{\ell}(a n^{- \widetilde{\theta}})\right)$.% by the markov property and definitions of $N_{\ell}$ and $q_{\ell}$.
These two bounds give the following
 $$
 |D_n(k, t) - D_n(k, t_j)| \le \left(Z_{AC,n}(t_{j+1}) - Z_{AC,n}(t_j) \right) + Y_1 + Y_2.
 $$
 Note the right hand side does not depend on $t$. We now have for all $0 \le j \le  n^{\widetilde{\theta}} -1 $ and $t \in [t_j, t_{j+1}]$. 
 \begin{multline*}
 \sup_{j \le n^{\widetilde{\theta}} - 1}\pr \left( \sup_{t \in [t_j, t_{j+1}]} | D_n(k, t) - D_n(k, t_j)| > \epsilon n^{\omega} \right)\\
   \le \sup_{j \le n^{\widetilde{\theta}} - 1}\left[\pr\left( \sum_{\ell = 0}^{k} \text{Bin}\left(D_n\left(\ell, t_j \right), q_{\ell} \left(a n^{-\widetilde{\theta}} \right) \right) > \epsilon n^{\omega} / 2\right) + \pr \left( | Z_{AC,n}(t_{j+1}) - Z_{AC,n}(t_j)| > \epsilon n^{\omega} / 2\right)\right] \\
  \le Ce^{C'a} \epsilon^{-2} n^{\frac{1}{2}- \widetilde{\theta} - \omega} + Ce^{C'a} \epsilon^{-1} n^{\frac{1}{2} - \omega}
 \end{multline*}
 where the second inequality comes from Lemmas \ref{lem:bin_right_tail} and \ref{lem:Z_t_js_close} which are proved below. The result now follows after taking the sum of these terms. \end{proof}

% We can now prove the result
% \begin{align*}
% \sum_{j = 0}^{n^{\widetilde{\theta}}  - 1} \pr \left( \sup_{t \in [t_j, t_{j+1}]} | D_n(k, t) - D_n(k, t_j)| > \epsilon n^{\omega} \right)
%  \le  \sum_{j = 0}^{n^{\widetilde{\theta}} - 1}  \frac{2Ce^{C'a}}{\epsilon^2} \frac{1}{n^{\omega - \frac{1}{2}}}
%  =  \frac{2Ce^{C'a}}{\epsilon^2} \frac{1}{n^{\omega - \widetilde{\theta} - \frac{1}{2}}} .
% \end{align*}
 %\end{proof}

 \begin{lemma}\label{lem:bin_right_tail}
 Let $\{t_j\}, \widetilde{\theta}$ and $\omega$ be as above and let $\epsilon \in (0,1)$. Then there exist constants $C'', n_0$ such that for all $n \ge n_0$ and all $a \le C'' \log n$,
 $$\sup_{j \le n^{\widetilde{\theta}}}\pr\left( \sum_{\ell = 0}^{k} \text{Bin}\left(D_n(\ell, t_j), q_{\ell}\left(a n^{-\widetilde{\theta}} \right) \right) > \epsilon n^{\omega} / 2 \right) \le Ce^{C'a} \epsilon^{-2} n^{\frac{1}{2} - \widetilde{\theta} - \omega}.$$
 \end{lemma}

 \begin{proof}
% Define the event $A_j = \left \{ Z_n(t_j) < \left(\gamma + \epsilon/ 8\right) n^{\widetilde{\theta} + \omega} \right \}$. Note that as $\sum_{\ell=0}^{\infty} (\ell+1) D_n(\ell, t_j) = 2Z_n(t_j) -1$, therefore on the event $A_j$,
Let $A_j = \left \{ Z_n(t_j) < \left(\gamma + \epsilon/ 8\right) n^{\widetilde{\theta} + \omega} \right \}$. Note $\sum_{\ell=0}^{\infty} (\ell+1) D_n(\ell, t_j) = 2Z_n(t_j) -1$, so on the event $A_j$,
 \begin{equation}
 \sum_{\ell=0}^{\infty} (\ell + 1) D_n(\ell, t_j) < 2\left(\gamma + \epsilon / 8\right) n^{\widetilde{\theta} + \omega}.
 \end{equation}
 Applying Chebyshev's inequality, on the event $A_j$, we have
 \begin{align}\label{onetot}
& \pr\left( \sum_{\ell = 0}^{k} \text{Bin}\left(D_n(\ell, t_j), q_{\ell} \left(a n^{-\widetilde{\theta}} \right) \right) >\frac{\epsilon}{2} n^{\omega} \Big | \cF_n(t_j) \right)\nonumber\\
 &  \le \frac{4}{\epsilon^2 n^{2\omega}} \sum_{\ell=0}^{k} \var \left( \text{Bin}\left(D_n(\ell, t_j), q_{\ell} \left(a n^{-\widetilde{\theta}} \right) \right) \Big | \cF_n(t_j) \right) \nonumber\\
& \le \frac{4}{\epsilon^2 n^{2\omega}} \sum_{\ell=0}^{k}  D_n(\ell, t_j) q_{\ell} \left(a n^{-\widetilde{\theta}} \right) \left(1 - q_{\ell} \left(a n^{-\widetilde{\theta}} \right) \right)
 \le \frac{4}{\epsilon^2 n^{2\omega}} \frac{Ca}{n^{\widetilde{\theta}}} \sum_{\ell=0}^{k}  D_n(\ell, t_j) (\ell + 1) \nonumber\\
&  \le \frac{4}{\epsilon^2 n^{2\omega}} \frac{Ca}{n^{\widetilde{\theta}}} \left[2\left(\gamma + \frac{\epsilon}{8}\right)n^{\widetilde{\theta} + \omega}\right]
 \le \frac{C'a}{\epsilon^2 n^{\omega}}
 \end{align}
 for $C'$ not depending on $j$, where the first inequality is from Chebyshev's inequality the third inequality is a consequence of Lemma \ref{qdec}, and the fourth inequality follows from the definition of $A_j$. We now have
 \begin{equation}
 \pr\left( \sum_{\ell = 0}^{k} \text{Bin}\left(D_n(\ell, t_j), q_{\ell} \left(a n^{-\widetilde{\theta}} \right) \right) > \epsilon n^{\omega} /2 \right)  \le \frac{C'a}{\epsilon^2 n^{\omega}} + \pr\left(  Z_n(t_j) \ge \left(\gamma + \epsilon /8\right) n^{\widetilde{\theta} + \omega}\right).
 \end{equation}
 Now, we control the second term above. By Lemma \ref{lem:1} (and the fact the integral is over a bounded interval) $\lambda_{\ell}(a) \le Ce^{C'a} (\ell+1)$. As $\widetilde{\theta} + \omega >1$, we can clearly choose $C'', n_0$ such that for all $n \ge n_0$ and all $a \le C'' \log n$, $\frac{\epsilon}{16} n^{\widetilde{\theta} + \omega}> (1+\gamma)Ce^{C'a} n$. For such $n,a$,
 $$
 \sum_{\ell=0}^{\infty} D_n(\ell, 0) \lambda_{\ell}(t_j) \le Ce^{C'a}\sum_{\ell=0}^{\infty} (\ell + 1)D_n(\ell, 0) = Ce^{C'a}(2\gamma n -1) <  \epsilon n^{\widetilde{\theta} + \omega}/16.
 $$
 Consequently, 
 \begin{multline}\label{twotot}
 \pr\left(  Z_n(t_j) \ge \left(\gamma + \frac{\epsilon}{8}\right) n^{\widetilde{\theta} + \omega}\right) \le \pr\left(  Z_n(t_j) - \gamma n \ge \left(\gamma + \frac{\epsilon}{8}\right) n^{\widetilde{\theta} + \omega} - \gamma n \right)
  \le \pr\left(  Z_n(t_j) - \gamma n > \frac{\epsilon}{8} n^{\widetilde{\theta} + \omega} \right)\\
  = \pr\left(  Z_{AC,n}(t_j) > \frac{\epsilon}{8} n^{\widetilde{\theta} + \omega} \right)
  \le \pr\left( \left | Z_{AC,n}(t_j) - \sum_{\ell=0}^{\infty} D_n(\ell, 0) \lambda_{\ell}(t_j) \right | > \frac{\epsilon}{16} n^{\widetilde{\theta} + \omega} \right)\\
  \le \frac{16}{\epsilon} \frac{1}{n^{\widetilde{\theta} + \omega}}  \E \left | Z_{AC,n}(t_j) - \sum_{\ell=0}^{\infty} D_n(\ell, 0)\lambda_{\ell}(t_j) \right |
  \le \frac{16}{\epsilon}Ce^{C'a} \frac{1}{n^{\widetilde{\theta} + \omega - \frac{1}{2}}}
 \end{multline}
 for $C,C'$ not depending on $j$, where the last inequality comes from \eqref{eq:ZAC-Nlam-close}. \eqref{onetot} and \eqref{twotot} and the fact that $\widetilde{\theta} < 1/2$. The result now follows. \end{proof}

 \begin{lemma}\label{lem:Z_t_js_close}
 Let $\{t_j\}, \widetilde{\theta}$ and $\omega$ be as above and let $\epsilon > 0$. Then
 $$ \sup_{j \le n^{\widetilde{\theta}} - 1}\pr \left( \left | Z_{AC,n}(t_{j+1}) - Z_{AC,n}(t_j) \right | > \frac{\epsilon}{2} n^{\omega} \right) \le  Ce^{C'a}\epsilon^{-1} n^{\frac{1}{2} - \omega}.$$
 \end{lemma}

 \begin{proof}
 Applying the triangle inequality,
 \begin{multline*}
 \left| Z_{AC,n}(t_{j+1}) - Z_{AC,n}(t_j)  \right| 
  \le  \left | Z_{AC,n}(t_{j+1}) -\sum_{\ell = 0}^{\infty} D_n(\ell, 0) \lambda_{\ell}(t_{j+1})\right| +  \left | Z_{AC,n}(t_{j}) -\sum_{\ell = 0}^{\infty} D_n(\ell, 0) \lambda_{\ell}(t_{j})\right|\\
   + \sum_{\ell = 0}^{\infty} D_n(\ell, 0) \left | \lambda_{\ell}(t_{j+1}) -\lambda_{\ell}(t_{j}) \right|.
 \end{multline*}
 Note by Lemma \ref{lem:lam-t-ts-diff} and the fact that $t_{j+1} - t_j = a n^{-\widetilde{\theta}}$
 \begin{multline}\label{Deq}
% \sum_{\ell = 0}^{\infty} D_n(\ell, 0)  \left | \lambda_{\ell}(t_{j+1}) -\lambda_{\ell}(t_{j}) \right|  \le Ce^{C'a} a n^{-\widetilde{\theta}} \sum_{\ell = 0}^{\infty} D_n(\ell, 0) (\ell + 1)\\
% = Ce^{C'a} a n^{-\widetilde{\theta}} (2\gamma n -1)  \le C'' a e^{C'a} n^{1 -\widetilde{\theta}}.
  \sum_{\ell = 0}^{\infty} D_n(\ell, 0)  \left | \lambda_{\ell}(t_{j+1}) -\lambda_{\ell}(t_{j}) \right|  \le Ce^{C'a} \frac{a} {n^{\widetilde{\theta}}}\sum_{\ell = 0}^{\infty} D_n(\ell, 0) (\ell + 1) = Ce^{C'a} \frac{a}{n^{\widetilde{\theta}} }(2\gamma n -1)  \le C'' a e^{C'a} n^{1 -\widetilde{\theta}}.
 \end{multline}

From equation \eqref{eq:ZAC-Nlam-close} we get
$
 \sup_{j \le n^{\widetilde{\theta}} - 1}\E \left | Z_n(t_{j}) -\sum_{\ell = 0}^{\infty} D_n(\ell, 0) \lambda_{\ell}(t_{j})\right| \le Ce^{C'a} \sqrt{n}.
$
Putting this all together, using \eqref{Deq}, the fact that $\omega > (1- \widetilde{\theta})$ and Markov's inequality we get for large enough $n$
 \begin{multline*}
 \pr \left( \left | Z_{AC,n}(t_{j+1} - Z_{AC,n}(t_j) \right| > \frac{\epsilon}{2} n^{\omega} \right) = \pr \left( \left | Z_n(t_{j+1} - Z_n(t_j) \right| > \frac{\epsilon}{2} n^{\omega} \right)\\
 \le \pr \left(  \left | Z_n(t_{j}) -\sum_{\ell = 0}^{\infty} D_n(\ell, 0) \lambda_{\ell}(t_{j})\right| +  \left | Z_n(t_{j+1}) -\sum_{\ell = 0}^{\infty} D_n(\ell, 0) \lambda_{\ell}(t_{j+1})\right| > \frac{\epsilon}{4} n^{\omega} \right)\\
 \le \frac{2}{\epsilon} n^{-\omega} \left(\E \left |Z_n(t_{j}) -\sum_{\ell = 0}^{\infty} D_n(\ell, 0) \lambda_{\ell}(t_{j}) \right| + \E\left | Z_n(t_{j+1}) -\sum_{\ell = 0}^{\infty} D_n(\ell, 0) \lambda_{\ell}(t_{j+1})\right |\right) \le 2Ce^{C'a}\epsilon^{-1} n^{ \frac{1}{2}- \omega}
 % \le \frac{C'''}{\epsilon} n^{-\omega} n^{\theta} 
 \end{multline*}
 for $C,C'$ not depending on $j$, which proves the lemma.\end{proof}
 
 \begin{lemma}\label{lem:Nt-sumNlam-close}
 There exist positive constants $C, C'$ such that for each $k$ and $\epsilon \in (0,1)$, 
 $$ \pr\left(\sup_{t \in [0, a]} \left | D_n(k, t) - \sum_{\ell = 0}^{\infty} D_n(\ell, 0) \lambda_{\ell}^{(k)}(t)  \right | > \epsilon(k+1) n^{\omega}\right) \le Ce^{C'a}\epsilon^{-2} n^{\widetilde{\theta} + \frac{1}{2} - \omega}$$
 $$ \pr\left(\sup_{t \in [0, a]} \left |Z_n(t) - \sum_{\ell = 0}^{\infty} D_n(\ell, 0) \lambda_{\ell}(t)  \right | > \epsilon n^{\omega}\right) \le Ce^{C'a}\epsilon^{-2} n^{\widetilde{\theta} + \frac{1}{2} - \omega }.$$
 \end{lemma}

 \begin{proof}

 Fix $k$ and $\epsilon \in (0,1)$. Note that
 \begin{align}\label{main}
& \pr\left( \sup_{t \in [0, a]} \left | D_n(k, t) - \sum_{\ell = 0}^{\infty} D_n(\ell, 0) \lambda_{\ell}^{(k)}(t)  \right | > \epsilon  n^{ \omega} \right)\nonumber\\
& \le \sum_{j=0}^{n^{\widetilde{\theta}}-1}  \pr\left( \sup_{t \in [t_j, t_{j+1}]} \left | D_n(k, t) - \sum_{\ell = 0}^{\infty} D_n(\ell, 0) \lambda_{\ell}^{(k)}(t)  \right | > \epsilon  n^{ \omega} \right) \nonumber\\
 & \le\sum_{j=0}^{n^{\widetilde{\theta}} -1}  \left[ \pr\left( \sup_{t \in [t_j, t_{j+1}]} \left | D_n(k, t) - D_n(k, t_j) \right| > \frac{\epsilon}{3}n^{ \omega}  \right) 
   + \pr\left( \left |D_n(k, t_j) - \sum_{\ell = 0}^{\infty} D_n(\ell, 0) \lambda_{\ell}^{(k)}(t_j) \right| > \frac{\epsilon}{3}n^{ \omega}  \right)\right. \nonumber\\
 & \qquad \left.+ \pr\left( \sup_{t \in [t_j, t_{j+1}]} \sum_{\ell = 0}^{\infty} D_n(\ell, 0) \left | \lambda_{\ell}^{(k)}(t) - \lambda_{\ell}^{(k)}(t_j) \right| > \frac{\epsilon}{3}n^{ \omega}  \right) \right].\nonumber\\
 \end{align}
 %From previous lemmas, we get can immediately control the first and third terms in the sum as follows.  
 By Lemma \ref{lem:sup-sup-Nt-Ntj},
 \begin{equation}\label{main1}
 \sum_{j=0}^{n^{\widetilde{\theta}} -1} \pr\left( \sup_{t \in [t_j, t_{j+1}]} \left | D_n(k, t) - D_n(k, t_j) \right| > \frac{\epsilon}{3}n^{ \omega}  \right) \le Ce^{C'a} \epsilon^{-2} n^{\widetilde{\theta} + \frac{1}{2} - \omega }.
 \end{equation}
 By Corollary \ref{cor:sumNl-lam-t-ts-diff},
$
 \sup_{j \le n^{\widetilde{\theta}}-1}\sup_{t \in [t_j, t_{j+1}]} \sum_{\ell = 0}^{\infty} D_n(\ell, 0) \left | \lambda_{\ell}^{(k)}(t) - \lambda_{\ell}^{(k)}(t_j) \right| \le Ce^{C'a}(k+\gamma + 2) n^{1-\widetilde{\theta}}
$
 and hence, as $\omega > 1 - \widetilde{\theta}$, there exists $n_0$ not depending on $k$ such that for all $n \ge n_0$,
 \begin{equation}\label{main2}
 \sum_{j=0}^{n^{\widetilde{\theta}}-1}\pr\left( \sup_{t \in [t_j, t_{j+1}]} \sum_{\ell = 0}^{\infty} D_n(\ell, 0) \left | \lambda_{\ell}^{(k)}(t) - \lambda_{\ell}^{(k)}(t_j) \right| > \frac{\epsilon (k+1)}{3}n^{ \omega}  \right) = 0.
 \end{equation}
 Finally we control the second term appearing in the sum \eqref{main}. It is sufficient to show
 \begin{equation}\label{main3}
 \sup_{j \le n^{\widetilde{\theta}}}\pr\left( \left |D_n(k, t_j) - \sum_{\ell = 0}^{\infty} D_n(\ell, 0) \lambda_{\ell}^{(k)}(t_j) \right| > (\epsilon/3)n^{ \omega}  \right) \le  Ce^{C'a}\epsilon^{-2} n^{\frac{1}{2} - \omega}.
 \end{equation}
By the triangle inequality and definitions of $D_n(k, t)$, and $\lambda_{\ell}^{(k)}(t)$, we see that for each fixed $j, k$,
 \begin{align}\label{eq:Ntj-2terms}
 \left |D_n(k, t_j) - \sum_{\ell = 0}^{\infty} D_n(\ell, 0) \lambda_{\ell}^{(k)}(t_j) \right| &  \le  \left| D_n^{BC}(k,t_j) - \sum_{\ell = 0}^k  D_n(\ell, 0) \pr \left( \xi_{f_1}^{(\ell)}(t_j) = k - \ell \right)  \right| \nonumber\\
 &  +   \left| D_n^{AC}(k,t_j) - \sum_{\ell=0}^{\infty} D_n(\ell, 0) \lambda^{AC,(k)}_\ell(t_j) \right|.
 \end{align}
 %so it is sufficient to control these two terms.
 By \eqref{eq:ZAC-Nlam-close2} and Markov's inequality,
 \begin{equation}\label{sub1}
 \sup_{j \le n^{\widetilde{\theta}}}\pr\left(\left | D_n^{AC}(k,t_j) - \sum_{\ell=0}^{\infty} D_n(\ell, 0) \lambda^{AC,(k)}_\ell(t_j) \right|  > \frac{\epsilon}{6} n^{\omega} \right) \le 6 Ce^{C'a}\epsilon^{-1} n^{\frac{1}{2} - \omega}.
 \end{equation}
 We now control the first term appearing in the bound in equation \eqref{eq:Ntj-2terms} by showing
 \begin{equation}\label{sub2}
 \sup_{t \in [0, a]}\E \left[ \left( D_n^{BC}(k,t) - \sum_{\ell = 0}^k  D_n(\ell, 0)  \pr \left( \xi_{f_1}^{(\ell)}(t) = k - \ell \right) \right)^2 \right] \le Cn.
 \end{equation}
 Fix $k$ and $t \in [0,a]$. Define a collection of mutually independent random variables\\ $\left \{ \xi_{f_1, m}^{(\ell)}(t)  \mid 1 \le m \le D_n(\ell, 0), 0 \le \ell \le k \right \}$ where $\xi_{f_1, m}^{(\ell)}(t) \sim \xi_{f_1}^{(\ell)}(t)$. Note that
 \begin{equation*}
 D_n^{BC}(k,t) \overset{d}{=} \sum_{\ell=0}^k \sum_{m=1}^{D_n(\ell, 0)} \mathds{1} \left( \xi_{f_1, m}^{(\ell)}(t) = k - \ell \right),
 \end{equation*}
 i.e. a vertex that was born before the change point and was of degree $\ell$ at the change point has to add $k - \ell$ new births to reach degree $k$ at time $t$. Therefore,
 \begin{multline*}\label{eq:Elen}
 \E \left[ \left( D_n^{BC}(k,t) - \sum_{\ell = 0}^k  D_n(\ell, 0)  \pr \left( \xi_{f_1}^{(\ell)}(t) = k - \ell \right) \right)^2 \right]\\
  = \E  \left[ \left ( \sum_{\ell=0}^k \sum_{m=1}^{D_n(\ell, 0)} \mathds{1} \left( \xi_{f_1, m}^{(\ell)}(t) = k - \ell \right) -  \sum_{\ell = 0}^k D_n(\ell, 0) \pr \left( \xi_{f_1}^{(\ell)}(t) = k - \ell \right)  \right)^2 \right ]\\
  = \E  \left[ \left \lbrace \sum_{\ell=0}^k \sum_{m=1}^{D_n(\ell, 0)} \left(\mathds{1} \left( \xi_{f_1, m}^{(\ell)}(t) = k - \ell \right) - \pr \left( \xi_{f_1}^{(\ell)}(t) = k - \ell \right)\right)  \right\rbrace^2 \right ].
 \end{multline*}
 Note that
 \begin{equation*}
  \sum_{\ell=0}^k  \sum_{m=1}^{D_n(\ell, 0)} \left( \mathds{1} \left( \xi_{f_1, m}^{(\ell)}(t) = k - \ell \right) -   \pr \left( \xi_{f_1}^{(\ell)}(t) = k - \ell \right) \right) \overset{d}
 {=} \sum_{\ell=0}^k  \sum_{m=1}^{D_n(\ell, 0)} Y_{\ell, m}
 \end{equation*}
 Where the random variables $\left \{Y_{\ell, m} \mid  1 \le m \le D_n(\ell, 0), 0 \le \ell \le k \right \}$ are mutually independent, supported on $[-1, 1]$ and $\E Y_{\ell, m} = 0$. Thus,

 \begin{equation*}
 \E \left[ \left( \sum_{\ell=0}^k  \sum_{m=1}^{D_n(\ell, 0)} Y_{\ell, m}\right)^2 \right] = \sum_{\ell=0}^k  \sum_{m=1}^{D_n(\ell, 0)} \E \left[ Y_{\ell, m}^2\right] \le C \sum_{\ell=0}^k D_n(\ell, 0) = C\gamma n
 \end{equation*}
 which proves \eqref{sub2}. Using \eqref{sub2} and Chebychev's inequality, we get
 \begin{equation}\label{sub3}
 \sup_{j \le n^{\widetilde{\theta}}}\pr \left(\left| D_n^{BC}(k,t_j) - \sum_{\ell = 0}^k  D_n(\ell, 0) \pr \left( \xi_{f_1}^{(\ell)}(t_j) = k - \ell \right)\right| > (\epsilon/6)n^{\omega}\right) \le C \epsilon^{-2} n^{1 - 2\omega}.
 \end{equation}
 Using \eqref{sub1} and \eqref{sub3} in \eqref{eq:Ntj-2terms}, we obtain \eqref{main3}. The first assertion in the lemma follows by using \eqref{main1}, \eqref{main2} and \eqref{main3} in \eqref{main}.
 %Since $ \sum_{\ell=0}^k D_n(\ell, \tcp) \le \gamma n$, this shows Equation \ref{eq:Elen} and thus the result. 
 The second assertion follows similarly upon noting that $Z_{AC,n}(t)$ is increasing in $t$ and using \eqref{eq:ZAC-Nlam-close}, Lemma \ref{lem:Z_t_js_close} and the first bound in Lemma \ref{lem:lam-t-ts-diff}. \end{proof}

 Now, we proceed towards removing the conditioning on $\cF_n(0)$ to complete the proof of Theorem \ref{thm:2}. We need the following Corollary to Lemma \ref{lem:9}.
 \begin{corollary} \label{cor:9.5}
 Fix $k\ge 0$, $\epsilon > 0$ and let $s_1, \dots, s_m \in [0, a]$ be $m$ fixed time points. Then, almost surely, there exists $n_0 \ge 1$ such that that for all $n \ge n_0$,
\begin{align*}
&\sup_{1 \le j \le m} \left | \frac{1}{n} \sum_{\ell = 0}^{\infty} D_n(\ell, 0) \lambda_{\ell}^{(k)} (s_j) - \gamma \sum_{\ell = 0}^{\infty} p^0_{\ell} \lambda_{\ell}^{(k)} (s_j)  \right | \le \epsilon,\\
 & \sup_{1 \le j \le m} \left | \frac{1}{n} \sum_{\ell = 0}^{\infty} D_n(\ell, 0) \lambda_{\ell}(s_j) - \gamma \sum_{\ell = 0}^{\infty} p^0_{\ell} \lambda_{\ell}(s_j)  \right | \le \epsilon.
  \end{align*}
 \end{corollary}

 \begin{proof}  Follows from Lemma \ref{lem:9} and the union bound. \end{proof}

 \begin{lemma} \label{lem:pk_has_finite_exp}
 Let $\set{p_k(f):k\geq 0}$ as in \eqref{eqn:pk-zero-def} be the asymptotic degree distribution using attachment function $f$  satisfying Assumption \ref{ass:attach-func}. Then $  \sum_{k=0}^{\infty} k p_k(f) =1.$
 \end{lemma}

 \begin{proof}
 % Recall $p_k = \frac{\lambda}{\lambda + f(k)} \prod_{i=0}^{k-1} \frac{f(k)}{\lambda + f(k)}$
 Recall that $p_k(f) = t_{k-1} - t_k$ where $t_k := \prod_{i=0}^{k} \frac{f(i)}{\lambda^* + f(i)}$ and $\lambda^*$ is the Malthusian parameter for the corresponding preferential attachment branching process. Therefore, $\sum_{k=1}^{\infty} k p_k(f) = \sum_{k=0}^n k ( t_{k-1} - t_k) = \sum _{k=0}^{\infty} t_k.$  By the definition of $\lambda^*$ and $t_k$ we see $\sum_{k=1}^{\infty} t_k = 1$, proving the lemma. \end{proof}

 % Nerman proof
 %Let $N_k^n$ be the number of vertices with degree $k$ after $f$ preferential attachment is run to $n$ vertices. Following the proof of $\ref{lem:9}$, we see that
 %$$
 %\frac{1}{n}\sum_{k=1}^{\infty} k N_k^n \overset{a.s.}{\longrightarrow} \sum_{k=1}^{\infty} k p_k
 %$$
 %Note for all $n$, $\sum_{k=1}^{\infty} k N_k^n  = n -1$ so $\frac{1}{n}\sum_{k=1}^{\infty} k N_k^n  < 1$. Since the LHS is bounded by 1, so must the RHS.
 %\end{proof}

 \begin{lemma}\label{lem:sumNlam-sumPllam-close}
 For any $k \ge 0$,
 $$ \sup_{t \in [0, a]} \left| \frac{1}{n} \sum_{\ell=0}^{\infty} D_n(\ell, 0) \lambda_{\ell}^{(k)}(t) - \gamma \sum_{\ell=0}^{\infty} p_{\ell}^0 \lambda_{\ell}^{(k)}(t)  \right| \overset{a.s.}{\longrightarrow} 0, \  \  \
\sup_{t \in [0, a]} \left| \frac{1}{n} \sum_{\ell=0}^{\infty} D_n(\ell, 0) \lambda_{\ell}(t) - \gamma \sum_{\ell=0}^{\infty} p_{\ell}^0 \lambda_{\ell}(t)  \right| \overset{a.s.}{\longrightarrow} 0.$$
 \end{lemma}

 \begin{proof}
 Fix $\epsilon >0$. Let $0 = s_1 < s_2 < \dots < s_m = a$ be a partition such that $|s_{j+1} - s_j| \le \epsilon$.
 % We select $C$ to be a constant large enough so that Corollary \ref{cor:lamLK-t-ts-diff} holds, $\frac{1}{n}\sum_{\ell=1}^{\infty} D_n(\ell, \tcp) (\ell + 1) \le C$, and $\gamma \sum_{\ell=1}^{\infty} p_{\ell}^0 (\ell + 1) \le C$ by Lemma \ref{lem:pk_has_finite_exp}. 
 By Corollary \ref{cor:sumNl-lam-t-ts-diff},
 \begin{equation*}
 \sup_{1 \le j \le m} \sup_{t \in [s_j, s_{j+1}]} \left| \frac{1}{n} \sum_{\ell=0}^{\infty} D_n(\ell, 0) \lambda_{\ell}^{(k)}(t) -  \frac{1}{n} \sum_{\ell=0}^{\infty} D_n(\ell, 0) \lambda_{\ell}^{(k)}(s_j) \right |
  \le Ce^{C'a}(k+3) \epsilon.
 \end{equation*}
 %where the final inequality comes from the construction of $C$ above.
 Similarly, using Corollary \ref{cor:lamLK-t-ts-diff},
 \begin{multline*}
   \sup_{1 \le j \le m-1} \sup_{t \in [s_j, s_{j+1}]}\left| \gamma \sum_{\ell=0}^{\infty} p_{\ell}^0 \lambda_{\ell}^{(k)}(t) - \gamma \sum_{\ell=0}^{\infty} p_{\ell}^0 \lambda_{\ell}^{(k)}(s_j)  \right |  \le \sup_{1 \le j \le k-1} \sup_{[s_j, s_{j+1}]}\gamma \sum_{\ell=0}^{\infty} p_{\ell}^0  \left| \lambda_{\ell}^{(k)}(t)   - \lambda_{\ell}^{(k)}(s_j)  \right| \\
   \le  Ce^{C'a}\epsilon \gamma \sum_{\ell=0}^{\infty} p_{\ell}^0  (k + \ell + 2)
  = Ce^{C'a}\gamma (k+3)\epsilon.
 \end{multline*}

 By Corollary \ref{cor:9.5}, almost surely, there exists $n_0 \ge 1$ such that that for all $n \ge n_0$,
 \begin{equation*}
  \sup_{1 \le j \le m} \left | \frac{1}{n} \sum_{\ell = 0}^{\infty} D_n(\ell, 0) \lambda_{\ell}^{(k)} (s_j) - \gamma \sum_{\ell = 0}^{\infty} p_{\ell}^0 \lambda_{\ell}^{(k)} (s_j)  \right | \le \epsilon.
 \end{equation*}
 From the above, we now have that for $n \ge n_0$,
 \begin{align*}
 &\sup_{t \in [0,a]}\left| \frac{1}{n} \sum_{\ell=0}^{\infty} D_n(\ell, 0) \lambda_{\ell}^{(k)}(t)  - \gamma \sum_{\ell=0}^{\infty} p_{\ell}^0 \lambda_{\ell}^{(k)}(t)  \right|\\
  \le 
  &\sup_{1 \le j \le m-1}\sup_{t \in [s_j, s_{j+1}]} \left| \frac{1}{n} \sum_{\ell=0}^{\infty} D_n(\ell, 0) \lambda_{\ell}^{(k)}(t) -  \frac{1}{n} \sum_{\ell=0}^{\infty} D_n(\ell, 0)\lambda_{\ell}^{(k)}(s_j) \right |  \\
 & +  \sup_{1 \le j \le m-1}\sup_{t \in [s_j, s_{j+1}]}\left| \gamma \sum_{\ell=0}^{\infty} p_{\ell}^0 \lambda_{\ell}^{(k)}(t) - \gamma \sum_{\ell=0}^{\infty} p_{\ell}^0 \lambda_{\ell}^{(k)}(s_j) \right |\\
  &+ \sup_{1 \le j \le m}\left | \frac{1}{n} \sum_{\ell = 0}^{\infty} D_n(\ell, 0)\lambda_{\ell}^{(k)} (s_j) - \gamma \sum_{\ell = 0}^{\infty} p_{\ell}^0 \lambda_{\ell}^{(k)} (s_j)  \right |
  \le Ce^{C'a}(k+3) \epsilon
 \end{align*}
 which proves the first assertion of the lemma. The second assertion follows similarly using Corollary \ref{cor:9.5} and the first bound in Lemma \ref{lem:lam-t-ts-diff}. \end{proof}
 
 \begin{proof}[Proof of Theorem \ref{thm:2}]
 The theorem follows from Lemmas \ref{lem:Nt-sumNlam-close} and \ref{lem:sumNlam-sumPllam-close}. 
 %The second claim follows similarly upon noting that $Z_{AC,n}(t)$ is increasing in $t$ and using \eqref{eq:ZAC-Nlam-close}, Lemma \ref{lem:Z_t_js_close} and the first bound in Lemma \ref{lem:lam-t-ts-diff}.
 \end{proof}

% \todo[inline]{SBana says: The $2+\epsilon$ moment assumption is NOT required for the pointwise and sup-norm convergence owing to the sharper estimate on the second moment of $N_n(i,j)$. But the current method does not give any rate of convergence owing to the soft argument in Lemma \ref{lem:sumNlam-sumPllam-close}. In order to have rates for change-point analysis or quick big bang, the moment assumption will be useful as this will give a quantified version of Lemma \ref{lem:sumNlam-sumPllam-close}. }

\begin{proof}[Proof of Corollary \ref{cor:int-wins}]
The essential message of this Corollary \ref{cor:int-wins} is that the tail of the distribution prescribed by the initializer function always wins.  Recall that the limit random variable $D_{\mvtheta}$ is a mixture of the distributions of $X_{\bc}$ and $X_{\ac}$. 
\end{proof}
\begin{lemma}\label{lem:exp-tail-ac}
    The random variable $X_{\ac}$ always has an exponential tail.  
\end{lemma}

% MEOW: which assumption 
\noindent {\bf Proof:} By construction, note that $X_{\ac}\stod \xi_{f_1}[0,\alpha]$. \chnr{Assumption \ref{ass:attach-func}} on the attachment functions implies that there exists $\bar \kappa > 0$ such that $\max(f_0(i), f_1(i))\leq \bar \kappa(i+1)$ for all $i$. In particular $\xi_{f_1}[0,\alpha]\stod Y_{\bar\kappa}[0,\alpha]$ where $Y_{\bar\kappa}(\cdot)$ is a rate $\bar\kappa$ Yule process (Definition \ref{def:yule-process}). Using Lemma \ref{lem:yule-prop} completes the proof. 
\qed

Thus is is enough to consider $X_{\bc}$ and show that this random variable has the same tail behavior as the random variable $D\sim \set{p_k^0: k\geq 1}$. Once again by construction, 
$
X_{\bc} \stod D + \sum_{i=1}^{D} Y_{\bar\kappa,i}[0,\alpha],
$
where $\set{Y_{\bar\kappa,i}(\cdot): i\geq 1}$ is an infinite collection of independent Yule processes (independent of $D$) \chsb{having the same distribution as $Y_{\bar\kappa}(\cdot)$}. Let $\mu := \E(Y_{\bar\kappa,i}[0,\alpha])$. Note $\mu > 1$. Conditioning on the value of $D$ we see that for $x\geq 1$, $\pr(X_{\bc} > x) \leq \cE$ where $\cE =\sum_{j=1}^{x/2\mu} \pr(D=j) \pr(\sum_{i=1}^{j} Y_{\bar\kappa,i}[0,\alpha] > x - j ) + \pr(D > x/2\mu)$. Further for $x\geq 1$,
\begin{equation}
	\label{eqn:d-y-kap-bound}
\cE   \leq \pr \left(\sum_{i=1}^{x/2\mu} Y_{\bar\kappa,i}[0,\alpha] > x \left(1 - \frac{1}{2\mu} \right) \right) + \pr(D > x/2\mu). 
\end{equation}

Standard large deviation bounds for the law of $Y_{\bar\kappa,i}$ implies that there exists constants $C_1, C_2$ such that for all $x$, 
$
\pr\left(\sum_{i=1}^{x/2\mu} Y_{\bar\kappa,i}[0,\alpha] > x\left(1 - \frac{1}{2\mu} \right) \right) \leq C_1\exp(-C_2 x).
$
Thus in the setting of Corollary \ref{cor:int-wins}(i), assuming $D$ has exponential tails, one finds using \eqref{eqn:d-y-kap-bound} that there exist finite constants  $C_1^\prime, C_2^\prime$ such that 
$
\pr(X_{\bc} > x) \leq C_1^\prime \exp(-C_2^\prime x).
$
This completes the proof of Corollary \ref{cor:int-wins}(i).  A similar argument, \chsb{along with} the obvious inequality $\pr(D> x) \leq \pr(X_{\bc} > x)$, verifies Corollary \ref{cor:int-wins}(ii). 
% This completes the proof.
\qed

\section{Proofs: Quick Big bang}

\label{sec:proofs-qbb}

\subsection{Proof of Theorem \ref{thm:qbbm}}
\chsb{Throughout this section we assume that $f_0$ satisfies Assumption \ref{ass:attach-func} and $f_1$ satisfies Assumptions \ref{ass:attach-func}, \ref{kickass} and \ref{xlogx}}. For notational convenience, \chsb{instead of considering the change point at $n^{\gamma}$ and evolving the tree till size $n$, we will consider the problem of the change point being at $n$ and evolving the tree till size $n^{1+\lambda_1^*\theta}$} for some $\theta>0$ (where $\lambda_1^*$ is the Malthusian rate corresponding to $f_1$). For this section, $t=0$ represents time $T_n$ (\chsb{the first time the total population size of the associated continuous time branching process has $n$ vertices}). It is easy to see that Theorem \ref{thm:qbbm} is equivalent to Theorem \ref{qbbmain} proved below. 

\chnr{We first give a proof outline. We again use the embedding of the \bbb{discrete time} network model \bbb{into} the associated continuous time branching process. Recall the notation from Section \ref{supdoc}. From Lemma \ref{lem:Nt-sumNlam-close}, for $k\ge 0$, there exists $\eta_0 >0 $ such that for $\eta \le \eta_0$,
\begin{equation}\label{bb1}
\frac{1}{n}\sup_{t \in [0, \eta \log n]} \left | D_n(k, t) - \sum_{\ell = 0}^{\infty} D_n(\ell, 0) \lambda_{\ell}^{(k)}(t)  \right |  \overset{P}{\longrightarrow} 0, \ \ \text{ as } n \rightarrow \infty.
\end{equation} 
Similarly, using Lemma \ref{lem:Nt-sumNlam-close}, we obtain \chnr{$\eta_0>0$} such that for all $\eta \le \eta_0$,
\begin{equation}\label{bb2}
\frac{1}{n}\sup_{t \in [0, \eta \log n]} \left | Z_n(t) - \sum_{\ell = 0}^{\infty} D_n(\ell, 0) \lambda_{\ell}(t)  \right |  \overset{P}{\longrightarrow} 0, \ \ \text{ as } n \rightarrow \infty.
\end{equation} 
\eqref{bb1} and \eqref{bb2} immediately imply for any $\eta \le \eta_0$,
\begin{align}\label{bb3}
\frac{1}{n^{1+\eta\lambda_1^*}}D_n(k, \eta \log n) - \frac{1}{n^{1+\eta\lambda_1^*}}\sum_{\ell = 0}^{\infty} D_n(\ell, 0) \lambda_{\ell}^{(k)}(\eta \log n) \overset{P}{\longrightarrow} 0,\\ \nonumber
\frac{1}{n^{1+\eta\lambda_1^*}}Z_n(\eta \log n) - \frac{1}{n^{1+\eta\lambda_1^*}}\sum_{\ell = 0}^{\infty} D_n(\ell, 0) \lambda_{\ell}(\eta \log n)  \overset{P}{\longrightarrow} 0
\end{align}
as $n \rightarrow \infty$. Thus, before the total population has grown too big, i.e. is of size $n^{1+ \eta \lambda_1^*}$ for some $\eta \le \eta_0$, one can approximate the empirical degree distribution and rescaled total population size by the normalized sums appearing in \eqref{bb3}. For each $\ell \ge 0$, $n^{-1}D_n(\ell, 0)$, converges to the classical limit degree distribution of the system without change point \bbb{i.e. $p_l^0 = p_l(f_0)$ as in \eqref{eqn:pk-zero-def}}. Thus, in lieu of \eqref{bb3}, one needs to understand how the quantities $n^{-\eta\lambda_1^*}\lambda_{\ell}^{(k)}(\eta \log n)$ and $n^{-\eta\lambda_1^*}\lambda_{\ell}(\eta \log n)$ behave for large $n$. Lemmas \ref{bbt0} to \ref{convdet} use techniques from renewal theory to quantify rates of convergence and characterize properties of the limits of these quantities in this general setup. This can be used to prove an analogue of Theorem \ref{thm:qbbm} for the branching process in the regime where the approximation \eqref{bb3} is valid \bbb{i.e. for $\eta\leq \eta_0$}. To extend this proof to the general case, we develop a `bootstrapping procedure' laid out in Lemma \ref{convd1} where we use results from Section \ref{sec:proofs-sp} and the \chsb{lemmas proved in this section} to show that for each $j \ge 0$, the `quick big bang' phenomenon holds when the population is of size $n^{1+ \eta \lambda_1^*}$ for some $\eta \le (j+1)\eta_0$ if it holds for all $\eta \le j\eta_0$. The rest of the section translates these results to the \bbb{network model in discrete time}.}

Define for each $\ell \ge 0$ and $\beta >0$, \chnr{the $\beta$-Laplace transform of the measure $\mu_{f_1}^{(\ell)}$ given by}
$$
w_{\ell}(\beta) := \int_0^{\infty}e^{- \beta s} \mu_{f_1}^{(\ell)}(ds).
$$
We will simply write $w_{\ell}$ for $w_{\ell}(\lambda_1^*)$.
We need the following technical lemmas. Recall from Assumption \ref{ass:attach-func} (ii) that there exists $\beta_1 \in (0, \lambda_1^*)$ such that $\hat{\rho}(\beta_1) < \infty$. Recall $C^*$ from Assumption \ref{kickass} applied to $f_1$.
\begin{lemma}\label{bbt0}
$\beta_1 \ge C^*$.
\end{lemma}
\begin{proof}
If $C^* = 0$, there is nothing to prove. So we assume $C^*>0$. For any $\epsilon \in (0, C^*)$, by Assumption \ref{kickass}, there exists $j_0 \ge 1$ such that for all $j \ge j_0$, $f_1(j) \ge (C^* - \epsilon)j$. Finiteness of $\hat{\rho}(\beta_1)$ implies that
\begin{equation}\label{finrho}
\sum_{k=1}^{\infty} \prod_{i=0}^{k-1}\frac{f_1(i + j_0)}{\beta_1 + f_1(i + j_0)} < \infty.
\end{equation} 
For any $k \ge 1$, noting that $x \mapsto \frac{x}{\beta_1 + x}$ is a strictly increasing function and, $\log(1+x) \le x$ for any $x\ge 0$, and $\sum_{j=j_1}^{j_2}\frac{1}{j} \le \int_{j_1 - 1}^{j_2}\frac{dx}{x}$ for any $j_2 \ge j_1 \ge 1$,
\begin{multline*}
\log\left[\prod_{i=0}^{k-1}\frac{f_1(i + j_0)}{\beta_1 + f_1(i + j_0)}\right] \ge \log\left[\prod_{i=0}^{k-1}\frac{i + j_0}{\frac{\beta_1}{C^* - \epsilon} + i + j_0}\right]  = - \sum_{i=0}^{k-1}\log\left[1 + \frac{\beta_1}{(C^* - \epsilon)(i+j_0)}\right]\\
\ge - \frac{\beta_1}{C^* - \epsilon}\sum_{i=0}^{k-1} \frac{1}{i+j_0} \ge - \frac{\beta_1}{C^* - \epsilon} \int_{j_0-1}^{j_0 + k -1}\frac{dx}{x} = - \frac{\beta_1}{C^* - \epsilon}\log\left(\frac{j_0 + k-1}{j_0-1}\right)
\end{multline*}
and thus
$$
\prod_{i=0}^{k-1}\frac{f_1(i + j_0)}{\beta_1 + f_1(i + j_0)} \ge \left(\frac{j_0-1}{j_0 + k-1}\right)^{\frac{\beta_1}{C^* - \epsilon}}.
$$
Thus, \eqref{finrho} holds only if $\beta_1 > C^* - \epsilon$. As $\epsilon>0$ is arbitrary, this proves the lemma.
\end{proof}
\begin{rem}\label{malthusC}
\chsb{Lemma \ref{bbt0} shows that if $f$ satisfies Assumptions \ref{ass:attach-func} and \ref{kickass}, then $\lambda^* > C^*$. In addition, if $f$ satisfies $\inf_{i \ge 0}f(i)>0$, then \cite[Proposition 5.7]{banerjee2020root}, implies
$\mathbb{E}\left(\exp\left\lbrace \delta\int_0^{\infty}e^{-\lambda^*t}\xi_{f}(dt)\right\rbrace\right) < \infty$
for some $\delta>0$ and, in particular, Assumption \ref{varass}.}
\end{rem}
\begin{lemma}\label{bbt1}
For any $\beta \in (\beta_1, \lambda_1^*]$, there exists a constant $C(\beta) > 0$ such that $w_{\ell}(\beta) \le C(\beta)(\ell +1)$ for $\ell \ge 0$.
\end{lemma} 
\begin{proof}
Fix any $\beta \in (\beta_1, \lambda_1^*]$ and $\ell \ge 0$. Since
$
\int_0^{\infty}e^{- \beta s} \mu_{f_1}(ds) = \sum_{k=1}^{\infty} \prod_{i=0}^{k-1} \frac{f_1(i)}{\beta + f_1(i)},
$
the sum on the right hand side is finite. Note that
 \begin{equation*}
 w_{\ell}(\beta) = \int_0^{\infty}e^{- \beta s} \mu^{(\ell)}_{f_1}(ds) = \sum_{k=1}^{\infty} \prod_{i=\ell}^{\ell + k-1}\frac{f_1(i )}{\beta + f_1(i)} = \frac{\sum_{k=1}^{\infty} \prod_{i=0}^{\ell + k-1}\frac{f_1(i)}{\beta + f_1(i)} }{\prod_{i=0}^{\ell-1}\frac{f_1(i)}{\beta + f_1(i )}} < \infty.
 \end{equation*}
 Choose and fix $\epsilon>0$ such that $C^* + 2\epsilon < \beta$ (which is possible by Lemma \ref{bbt0}). By Assumption \ref{kickass}, there exists $j_0 \ge 1$ such that for all $j \ge j_0$, $f_1(j) \le (C^* + \epsilon)j$. For any $ \ell \ge j_0$, using the facts that $x \mapsto \frac{x}{\beta + x}$ is a strictly increasing function and, $\log(1+x) \ge \frac{x}{1+x}$ for any $x\ge 0$, and $\sum_{j=j_1}^{j_2}\frac{1}{j} \ge \int_{j_1}^{j_2+1}\frac{dx}{x}$ for any $j_2 \ge j_1 \ge 1$, we obtain for any $\ell \ge j_0$,
 \begin{multline*}
 \log\left[\prod_{i=\ell}^{2\ell-1}\frac{f_1(i)}{\beta + f_1(i)} \right] \le \log\left[\prod_{i=\ell}^{2\ell-1}\frac{i}{\frac{\beta}{C^* + \epsilon} + i}\right] = - \sum_{i=\ell}^{2\ell-1}\log\left[1 + \frac{\beta}{(C^* +\epsilon)i}\right]\\
\le - \sum_{i=\ell}^{2\ell-1}\frac{\frac{\beta}{(C^* +\epsilon)i}}{1 + \frac{\beta}{(C^* +\epsilon)i}} \le  - \frac{\frac{\beta}{C^* + \epsilon}}{1 + \frac{\beta}{(C^* +\epsilon)\ell}}\sum_{i=\ell}^{2\ell-1}\frac{1}{i} \le - \frac{\frac{\beta}{C^* + \epsilon}}{1 + \frac{\beta}{(C^* +\epsilon)\ell}}\int_{\ell}^{2\ell}\frac{dx}{x} = - \frac{\frac{\beta}{C^* + \epsilon}}{1 + \frac{\beta}{(C^* +\epsilon)\ell}}\log 2.
 \end{multline*}
 Take $\ell_1 \ge j_0$ such that $\frac{\frac{\beta}{C^* + \epsilon}}{1 + \frac{\beta}{(C^* +\epsilon)\ell_1}} \ge \frac{\beta}{C^* + 2\epsilon}$. From the above calculation, for all $\ell \ge \ell_1$,
 $
 \prod_{i=\ell}^{2\ell-1}\frac{f_1(i)}{\beta + f_1(i)}  \le 2^{-\frac{\beta}{C^* + 2\epsilon}}.
 $
 Using this bound iteratively, we obtain for any $j \ge 1$,
 $
  \prod_{i=\ell}^{2^j\ell-1}\frac{f_1(i)}{\beta + f_1(i)}  \le 2^{-\frac{\beta j}{C^* + 2\epsilon}}.
 $
 Thus, for all $\ell \ge \ell_1$,
 \begin{multline*}
 w_{\ell}(\beta) = \sum_{k=1}^{\infty} \prod_{i=\ell}^{\ell + k-1}\frac{f_1(i)}{\beta + f_1(i)} \le \ell + \sum_{j=0}^{\infty} \sum_{k=2^j\ell}^{2^{j+1}\ell-1}\prod_{i=\ell}^{\ell + k -1}\frac{f_1(i)}{\beta + f_1(i)} \le  \ell + \sum_{j=0}^{\infty} 2^j\ell\prod_{i=\ell}^{2^j\ell-1}\frac{f_1(i)}{\beta + f_1(i)}\\
  = \ell\left[1 + \sum_{j=0}^{\infty} 2^{\left(1-\frac{\beta}{C^* + 2\epsilon}\right)j}\right] = \left(\frac{2-2^{\left(1-\frac{\beta}{C^* + 2\epsilon}\right)}}{1-2^{\left(1-\frac{\beta}{C^* + 2\epsilon}\right)}}\right)\ell
 \end{multline*}
 where the sum converges as  $C^* + 2\epsilon < \beta$. This proves the lemma.
\end{proof}
%\begin{lemma}\label{bbt2}
%There exist $C, \vartheta>0$ such that
%$$
%\left|e^{-\lambda_1^* t} m_{f_1}(t) - 1\right| \le Ce^{-\vartheta t}, \ \ t \ge 0.
%$$
%\end{lemma}
%\begin{proof}
%We use the renewal equation description of $m_{f_1}$:
%$$
%e^{-\lambda_1^* t} m_{f_1}(t) = e^{-\lambda_1^* t} + \int_0^{t}e^{-\lambda_1^* (t-s)} m_{f_1}(t-s) e^{-\lambda_1^* s} \mu_{f_1}(ds).
%$$
%Recalling \eqref{eqn:prop-under-lamb}, choose $\vartheta>0$ such that $\int_0^{\infty}e^{- (\lambda_1^* - 2\vartheta) s} \mu_{f_1}(ds)< \infty$. Define
%$
%h(t) = e^{\vartheta t}\left(e^{-\lambda_1^* t} m_{f_1}(t) - 1\right).
%$
%Then $h$ satisfies the renewal equation
%$$
%h(t) = h^*(t) + \int_0^{t}h(t-s) e^{-(\lambda_1^* - \vartheta) s} \mu_{f_1}(ds)
%$$
%where $h^*(t) = e^{-(\lambda_1^* - \vartheta)t} - e^{\vartheta t}\int_t^{\infty}e^{-\lambda_1^* s} \mu_{f_1}(ds)$. We claim that $\lim_{t \rightarrow \infty} h(t)$ exists and is finite. By Theorem 5.2.6 of \cite{jagers-ctbp-book}, it suffices to show that $\sum_{k=0}^{\infty}\sup_{t \in [k, k+1]}|h^*(t)| < \infty$. To show this, first observe that for any $k \ge 0$,
%$$
%\int_{k}^{\infty}e^{-\lambda_1^* s} \mu_{f_1}(ds) \le e^{-2\vartheta k}\int_0^{\infty}e^{- (\lambda_1^* - 2\vartheta) s} \mu_{f_1}(ds).
%$$
%Therefore,
%$$
%\sum_{k=0}^{\infty}\sup_{t \in [k, k+1]}|h^*(t)| \le \sum_{k=0}^{\infty}\left(e^{-(\lambda_1^* - \vartheta)k} + e^{\vartheta (k+1)} e^{-2\vartheta k}\int_0^{\infty}e^{- (\lambda_1^* - 2\vartheta) s} \mu_{f_1}(ds)\right) < \infty
%$$
%proving the existence of a finite limit of $h(t)$ as $t \rightarrow \infty$ which implies the lemma.
%\end{proof}
Recall the class of characteristics $\cC$ defined in \eqref{charclass}. \chsb{For given $\phi \in \cC$ and \bbb{initial values} $\{\lambda^{\phi}_{\ell}(0) \in [0,1] : \ell \ge 0\}$, define for each $\ell \ge 0$,}
\begin{equation}\label{lpdef}
\lambda^{\phi}_{\ell}(t) = \lambda^{\phi}_{\ell}(0) + \int_0^{t}m^{\phi}_{f_1}(t-s) \mu^{(\ell)}_{f_1}(ds).
\end{equation}
Note that this definition generalizes the expected aggregate $\phi$-score of offsprings of a degree $\ell$ parent defined in Section \ref{sec:proofs-sp} (see \chsb{just before Theorem \ref{thm:1}}) in that we allow for a general initial value $\lambda^{\phi}_{\ell}(0) \in [0,1]$. Hence, we keep the same notation. \chsb{Two special instances of $\lambda^{\phi}_{\ell}(\cdot)$ that we have already used extensively are given by taking $\phi(t) = \ind\set{t \ge 0}$, $t \ge 0$, $\lambda^{\phi}_{\ell}(0) = 1, \ell \ge 0,$ which we denoted by $\lambda_{\ell}(\cdot)$, and $\phi(t) = \ind\set{\xi(t) =k}$, $t \ge 0$, $\lambda^{\phi}_{\ell}(0) = \pr\left(\xi^{(\ell)}_{f_1}(t)=k-\ell\right)$ for $\ell \ge 0, k \ge 0$, denoted by $\lambda_{\ell}^{(k)}(\cdot)$ (see \eqref{lambdadef})}.
\begin{lemma}\label{conrt}
Let $\phi \in \cC$ such that $\lim_{t \rightarrow \infty}e^{-\lambda_1^* t}m_{f_1}^{\phi}(t) = c_{\phi}$.
%For $\ell \ge 0$, let $\lambda^{\phi}_{\ell}(\cdot)$ be defined as in \eqref{lpdef}.
Recall $\lambda^{\phi}_{\ell}(\cdot)$ defined in \eqref{lpdef}.
There is a constant $C>0$ for which the following holds: for any $\epsilon>0$, there exists $t(\epsilon)>0$ such that for any $\ell \ge 0$,
$$
\chnr{\sup_{t \ge t(\epsilon)}\left|e^{-\lambda_1^* t}\lambda_{\ell}^{\phi}(t) - w_{\ell}c_{\phi}\right| \le C\epsilon(\ell + 1).}
$$
\end{lemma}
\begin{proof}
In this proof, $C, C', C''$ will denote generic positive constants not depending on $t, \ell$ whose values might change from line to line. From \eqref{lpdef} and the definition of $w_{\ell}$, we have for any $t \ge 0$,
\begin{multline}\label{step}
e^{-\lambda_1^* t}\lambda_{\ell}^{\phi}(t) -w_{\ell}c_{\phi} = \lambda_{\ell}^{\phi}(0)e^{-\lambda_1^* t} - c_{\phi}\int_t^{\infty}e^{-\lambda_1^* s}\mu_{f_1}^{(\ell)}(ds)+  \int_0^{t}\left(e^{-\lambda_1^* (t-s)}m^{\phi}_{f_1}(t-s) - c_{\phi}\right)e^{-\lambda_1^* s} \mu^{(\ell)}_{f_1}(ds).
\end{multline}
Choose any $\epsilon>0$. \chnr{Fix} any $\vartheta>0$ such that $\lambda_1^* - \vartheta>\beta_1$. As $\lim_{t \rightarrow \infty}e^{-\lambda_1^* t}m^{\phi}_{f_1}(t) = c_{\phi}$ and \\ 
$\sup_{t< \infty}e^{-\lambda_1^* t}m^{\phi}_{f_1}(t) < \infty$ (which holds because the limit as $t \rightarrow \infty$ exists and as $\phi \in \cC$, therefore for each $a>0$, $\sup_{t \in [0,a]}m^{\phi}_{f_1}(t) \le C\sup_{t \in [0,a]}m_{f_1}(t)< \infty$ by virtue of \eqref{eqn:vf1-bound}), there exists $t_0 >0$ such that for all $t \ge t_0$, $\left|e^{-\lambda_1^* t}m^{\phi}_{f_1}(t) - c_{\phi}\right| \le \epsilon$ and $e^{-\vartheta t}\left(\sup_{z< \infty}e^{-\lambda_1^* z}m^{\phi}_{f_1}(z) + c_{\phi}\right) \le \epsilon$. Thus, for any $t \ge 2t_0$,
$$
\sup_{s \le t}e^{-\vartheta s}\left|e^{-\lambda_1^* (t-s)}m^{\phi}_{f_1}(t-s) - c_{\phi}\right| \le \epsilon.
$$
Thus, applying Lemma \ref{bbt1} with $\beta = \lambda_1^* - \vartheta$, we conclude that for any $t \ge 2t_0$,
\begin{multline*}
\int_0^{t}\left|e^{-\lambda_1^* (t-s)}m^{\phi}_{f_1}(t-s) - c_{\phi}\right| e^{-\lambda_1^* s}\mu^{(\ell)}_{f_1}(ds)
 = \int_0^t e^{-\vartheta s}\left|e^{-\lambda_1^* (t-s)}m^{\phi}_{f_1}(t-s) - c_{\phi}\right|e^{-(\lambda_1^* - \vartheta)s} \mu^{(\ell)}_{f_1}(ds) \\
 \le \epsilon w_{\ell}(\lambda_1^*-\vartheta) \le C \epsilon(\ell + 1).
\end{multline*}
Moreover, as $\int_0^{\infty}e^{- (\lambda_1^* - \vartheta) s} \mu_{f_1}^{(\ell)}(ds) \le C(\ell + 1)$, for $t \ge 0$,
$
c_{\phi}\int_t^{\infty}e^{-\lambda_1^* s}\mu_{f_1}^{(\ell)}(ds) \le C'(\ell + 1)e^{-\vartheta t}.
$
Using these in \eqref{step} and recalling $\lambda_{\ell}(0) \in [0,1]$ for each $\ell$, we obtain for $t \ge 2t_0$,
\begin{equation*}
\left|e^{-\lambda_1^* t}\lambda_{\ell}^{\phi}(t) - w_{\ell}c_{\phi}\right| \le e^{-\lambda_1^* t} + C'(\ell + 1)e^{-\vartheta t} + C \epsilon(\ell + 1).
\end{equation*}
Thus, there exists $t_1 \ge 2t_0$ such that for all $\ell \ge 0$ and all $t \ge t_1$,
$
\left|e^{-\lambda_1^* t}\lambda_{\ell}^{\phi}(t) - w_{\ell}c_{\phi}\right| \le C''\epsilon(\ell + 1).
$\end{proof}

\begin{lemma}\label{limconv}
Let $\phi \in \cC$ such that $\lim_{t \rightarrow \infty}e^{-\lambda_1^* t}m_{f_1}^{\phi}(t) = c_{\phi}$.
%For $\ell \ge 0$, let $\lambda^{\phi}_{\ell}(\cdot)$ be defined as in \eqref{lpdef}.
Fix any $\eta>0, a \in \mathbb{R}$. Then as $n \rightarrow \infty$,
$$
n^{-(1+\eta\lambda_1^*)}\sum_{\ell = 0}^{\infty} D_n(\ell, 0) \lambda^{\phi}_{\ell}(\eta \log n + a) \overset{P}{\longrightarrow} c_{\phi}e^{\lambda_1^* a}\sum_{\ell=0}^{\infty} p_{\ell}^0 w_{\ell}.
$$
\end{lemma}
\begin{proof}
In this proof, \bbb{once again} $C, C', C''$ will denote generic positive constants not depending on $n,t, \ell$ whose values might change from line to line. Note that
\begin{multline}\label{breaklim}
\left|n^{-(1+\eta\lambda_1^*)}\sum_{\ell = 0}^{\infty} D_n(\ell, 0) \lambda_{\ell}^{\phi}(\eta \log n + a) -  c_{\phi}e^{\lambda_1^* a }\sum_{\ell=0}^{\infty} p_{\ell}^0 w_{\ell}\right|\\
 \le n^{-1}\sum_{\ell=0}^{\infty}D_n(\ell, 0)\left|\lambda_{\ell}^{\phi}(\eta \log n + a) n^{- \eta\lambda_1^*} - w_{\ell}c_{\phi}e^{\lambda_1^* a}\right|
 + c_{\phi}e^{\lambda_1^* a}\left|\sum_{\ell=0}^{\infty} n^{-1}D_n(\ell, 0)w_{\ell} - \sum_{\ell=0}^{\infty} p_{\ell}^0 w_{\ell}\right|.
\end{multline}
To show that the second term goes to zero in probability, consider the characteristic $\chi(t) = \sum_{\ell=0}^{\infty}w_{\ell}\ind\set{\xi_{f_1}(t)=\ell}$. By Lemma \ref{bbt1}, $w_{\ell} \le C(\ell + 1)$ and hence, $\chi \in \cC$. Thus, by Lemma \ref{lem:deg_dist_quad_conv} (i),
\begin{equation}\label{p1}
 \left|\sum_{\ell=0}^{\infty}n^{-1}D_n(\ell, 0)w_{\ell} - \sum_{\ell=0}^{\infty} p_{\ell}^0 w_{\ell}\right| \overset{P}{\longrightarrow} 0 \ \ \text{ as } n \rightarrow \infty.
\end{equation}
To show that the first term in the bound \eqref{breaklim} goes to zero in probability, take any $\epsilon>0$. Recalling \\
$\sum_{\ell=0}^{\infty}D_n(\ell, 0) = n$ and $\sum_{\ell=0}^{\infty}(\ell + 1)D_n(\ell, 0) = 2n-1$, and taking $t = \eta \log n + a$ for any $n \ge e^{(t(\epsilon)-a)/\eta}$ in Lemma \ref{conrt}, we obtain
\begin{equation*}
\sum_{\ell=0}^{\infty}n^{-1} D_n(\ell, 0)\left| \lambda_{\ell}^{\phi}(\eta \log n + a) n^{- \eta\lambda_1^*} - w_{\ell}c_{\phi}e^{\lambda_1^* a}\right| \le n^{-1} C'' e^{\lambda_1^* a}\epsilon\sum_{\ell=0}^{\infty}(\ell + 1)D_n(\ell, 0) \le 2C''e^{\lambda_1^* a}\epsilon.
\end{equation*} 
As $\epsilon>0$ is arbitrary, the first term in \eqref{breaklim} converges to zero as $n \rightarrow \infty$ and completes the proof.
\end{proof}
Define $m_1^{\star} := \int_0^{\infty}u e^{-\lambda_1^* u} \mu_{f_1}(du)$. \chsb{For $\ell \ge 0, k \ge 0$, recall $\lambda_{\ell}(\cdot)$ and $\lambda_{\ell}^{(k)}(\cdot)$ from \eqref{lambdadef}.}
\begin{cor}\label{cordig}
Fix any $\eta>0$ and $k \ge 0$. Then as $n \rightarrow \infty$
$$
n^{-(1+\eta\lambda_1^*)}\sum_{\ell = 0}^{\infty} D_n(\ell, 0) \lambda_{\ell}(\eta \log n) \overset{P}{\longrightarrow} \sum_{\ell=0}^{\infty} p_{\ell}^0  w_{\ell}/\lambda_1^* m_1^{\star},
$$
$$
n^{-(1+\eta\lambda_1^*)}\sum_{\ell = 0}^{\infty} D_n(\ell, 0) \lambda^{(k)}_{\ell}(\eta \log n)  \overset{P}{\longrightarrow}p_{k}^{1}\sum_{\ell=0}^{\infty} p_{\ell}^0 w_{\ell} / \lambda_1^* m_1^{\star}.
$$
\end{cor}
\begin{proof}
Follows from Lemma \ref{limconv} upon using the explicit formulas
\begin{align*}
\lambda_{\ell}(t) = 1 + \int_0^{t}m_{f_1}(t-s) \mu^{(\ell)}_{f_1}(ds), \ \ \ 
\lambda^{(k)}_{\ell}(t) = \pr\left(\xi^{(\ell)}_{f_1}(t)=k-\ell\right) + \int_0^{t}m^{(k)}_{f_1}(t-s) \mu^{(\ell)}_{f_1}(ds), \ t \ge 0,
\end{align*}
and observing by Lemma \ref{lem:deg_dist_quad_conv} (ii) % TODO: make sure 6.7 includes this convergence
\begin{equation} \label{eq:m_lims}
\lim_{t \rightarrow \infty} e^{-\lambda_1^* t}m_{f_1}(t) =( \lambda_1^* m_1^{\star})^{-1}, \ \ \ \lim_{t \rightarrow \infty} e^{-\lambda_1^* t}m^{(k)}_{f_1}(t) = p^1_k (\lambda_1^* m_1^{\star})^{-1}.
\end{equation}
\end{proof}
\begin{lemma}\label{convdet}
There exists $\eta_0>0$ such that for any $\eta \le \eta_0$, the following limits hold as $n \rightarrow \infty$:
\begin{itemize}
\item[(i)] $n^{-(1+\eta\lambda_1^*)}Z_n(\eta \log n) \overset{P}{\longrightarrow}  \sum_{\ell=0}^{\infty} p_{\ell}^0 w_{\ell} / \lambda_1^* m_1^{\star}$,
\item[(ii)] For any $k \ge 0$, $n^{-(1+\eta\lambda_1^*)}D_n(k, \eta \log n)\overset{P}{\longrightarrow} p^1_k\sum_{\ell=0}^{\infty} p_{\ell}^0 w_{\ell}/ \lambda_1^* m_1^{\star}$.
\end{itemize}
\end{lemma}
\begin{proof}
(i) and (ii) follow from \eqref{bb2} and \eqref{bb1} respectively along with Corollary \ref{cordig}.
\end{proof}

\begin{cor}\label{soft}
$\sum_{\ell=0}^{\infty} p_{\ell}^1 w_{\ell} = \lambda_1^* m_1^{\star}$.
\end{cor}

\begin{proof}
Note that Lemma \ref{convdet} (i) holds in the special case where $f_0$ is taken to be $f_1$ (the model without change point). In this case, $p^0_{\ell} = p^1_{\ell}$ for all $\ell \ge 0$. By Lemma \ref{lem:deg_dist_quad_conv} (ii),
$
Z_n(\eta_0 \log n) e^{- \lambda_1^*(T_n + \eta_0 \log n)} \overset{a.s.}{\longrightarrow} W_{\infty} /\lambda_1^* m_1^{\star}.
$
Moreover, as $Z\left(T_n\right) = n$, therefore, applying Lemma \ref{lem:deg_dist_quad_conv} (ii) again,
$
n^{-1} e^{\lambda_1^* T_n}= e^{\lambda_1^* T_n}/Z\left(T_n\right) \overset{a.s.}{\longrightarrow} \lambda_1^* m_1^{\star} / W_{\infty}. 
$
Using these observations, we obtain
$$
n^{ -(1+\eta_0\lambda_1^*)}Z_n(\eta_0 \log n) = n^{-1}  e^{\lambda_1^* T_n}  Z_n(\eta_0 \log n) e^{- \lambda_1^*(T_n + \eta_0 \log n)} \overset{a.s.}{\longrightarrow} 1.
$$
Comparing this with Lemma \ref{convdet} (i) with $f_0=f_1$ gives the result.
\end{proof}

Recall that for any $k \ge 0$, $\xi_{f_1}^{(k)}(\cdot)$ is the point process denoting the distribution of birth times of children of a vertex which is of degree $k$ at time zero. The following lemma gives an estimate on the second moment of $\xi_{f_1}^{(k)}(t)$ under Assumption \ref{kickass}.
\begin{lemma}\label{secqbb}
There exists $C>0$ and $\beta' < \lambda_1^*$ such that for any $k \ge 0, t \ge 0$,
$
\mathbb{E}\left(\xi_{f_1}^{(k)}(t)\right)^2 \le C(k + 1)^2e^{2\beta' t}.
$
\end{lemma}
\begin{proof}
By Assumption \ref{kickass} and Lemma \ref{bbt0}, for any $\beta' \in (\beta_1, \lambda_1^*)$, there exists $\ell_0 \ge 0$ such that for all $\ell \ge \ell_0$, $f_1(\ell) \le \beta' \ell$. Let $m = \max_{\ell \le \ell_0}f_1(\ell)$. It is clear that $\xi_{f_1}^{(k)}(\cdot)$ is stochastically dominated by the offspring process of a continuous time branching process with linear attachment function $f^*(\ell) = \beta' \ell + 1 + (m + \beta'k), \ell \ge 0$. Applying the second moment obtained in Lemma \ref{lem:affine-pp-bound} (with $\nu =\beta'$ and $\kappa = 1 + m + \beta' (k -1)$) the lemma follows.

%It can be checked that the processes
%$$
%M'_1(t) := e^{-\beta't}\xi_{f^*}^{(k)}(t) - \frac{1+ m+ \beta'k}{\beta'}\left(1 - e^{-\beta' t}\right)
%$$
%and
%$$
%M'_2(t) := e^{- 2\beta't}\left(\xi_{f^*}^{(k)}\right)^2(t) - \int_0^t\left[2(1+ m + \beta'k) + \beta'\right]e^{-2\beta' s}\xi_{f^*}^{(k)}(s)ds - \frac{1+m+\beta'k}{2\beta'}\left(1 - e^{-2\beta' t}\right)
%$$
%are martingales, using which we conclude that
%$$
%\mathbb{E}\left(\xi_{f_1}^{(k)}(t)\right)^2 \le \mathbb{E}\left(\xi_{f^*}^{(k)}(t)\right)^2 = \frac{(1+m+\beta'k)\left[2(1+ m + \beta'k) + \beta'\right]}{2\beta'^2}\left(e^{\beta't} - 1\right)^2 + \frac{1+m+\beta'k}{2\beta'}\left(e^{2\beta't} - 1\right)
%$$
%from which, the lemma follows.
\end{proof}

For \chnr{$\eta>0, j \ge 0$}, let $D_n(k, j, \eta)$ denote the number of vertices of degree $k$ at time $(j+1)\eta \log n$ that were born before time $j \eta \log n$ \chsb{(including possibly the ones at time zero)}.
\begin{lemma}\label{prech}
For any $\eta>0$, $j\ge 0$, as $n \rightarrow \infty$,
$
\displaystyle{\sum_{k=0}^{\infty}(k+1)D_n(k, j, \eta) \Big/ \left(Z_n(j \eta \log n)n^{\lambda_1^* \eta}\right) \overset{P}{\longrightarrow} 0}.
$
\end{lemma}
\begin{proof}
We will condition on $\cF_n(j \eta \log n)$ throughout the proof. For $1 \le m \le D_n(\ell, j \eta \log n)$, denote by $\xi^{(\ell)}_{f_1,m}(t)$ the degree at time $t + j \eta \log n$ of the m-th vertex of degree $\ell$ at time $j \eta \log n$. Observe that
\begin{align*}
&\sum_{k=0}^{\infty}(k+1)D_n(k, j, \eta) = \sum_{k=0}^{\infty}(k+1)\sum_{\ell=0}^k\sum_{m=1}^{D_n(\ell, j\eta \log n)}\ind\set{\xi^{(\ell)}_{f_1,m}(\eta \log n) = k - \ell}\\
 &= \sum_{\ell=0}^{\infty}\sum_{m=1}^{D_n(\ell, j \eta \log n)}\sum_{k=\ell}^{\infty}(k+1)\ind\set{\xi^{(\ell)}_{f_1,m}(\eta \log n) = k - \ell} = \sum_{\ell=0}^{\infty}\sum_{m=1}^{D_n(\ell, j \eta \log n)}\left(\ell + 1 + \xi^{(\ell)}_{f_1,m}(\eta \log n)\right)\\
 & = \sum_{\ell=0}^{\infty}(\ell + 1)D_n(\ell, j \eta \log n) + \sum_{\ell=0}^{\infty}\sum_{m=1}^{D_n(\ell, j \eta \log n)}\xi^{(\ell)}_{f_1,m}(\eta \log n)\\
&  = 2Z_n(j \eta \log n) - 1 + \sum_{\ell=0}^{\infty}\sum_{m=1}^{D_n(\ell, j \eta \log n)}\xi^{(\ell)}_{f_1,m}(\eta \log n).
\end{align*}
Thus, it suffices to show that as $n \rightarrow \infty$,
\begin{equation}\label{mainprech}
\frac{1}{Z_n(j \eta \log n)}\sum_{\ell=0}^{\infty}\sum_{m=1}^{D_n(\ell,j \eta \log n)}\frac{1}{n^{\lambda_1^*\eta}}\xi^{(\ell)}_{f_1,m}(\eta \log n) \overset{P}{\longrightarrow} 0.
\end{equation}
Note that using Lemma \ref{secqbb},
\begin{multline*}
\operatorname{Var}\left(\frac{1}{Z_n(j \eta \log n)}\sum_{\ell=0}^{\infty}\sum_{m=1}^{D_n(\ell,j \eta \log n)}\frac{1}{n^{\lambda_1^*\eta}}\xi^{(\ell)}_{f_1,m}(\eta \log n)\right)\\
 \le \frac{1}{Z^2(j \eta \log n)n^{2\lambda_1^* \eta}}\sum_{\ell=0}^{\infty}\sum_{m=1}^{D_n(\ell,j \eta \log n)}\mathbb{E}\left(\xi^{(\ell)}_{f_1,m}(\eta \log n)\right)^2 \le \frac{Cn^{2\beta'\eta}}{Z^2(j \eta \log n)n^{2\lambda_1^* \eta}}\sum_{\ell=0}^{\infty}(\ell + 1)^2D_n(\ell,j \eta \log n).
\end{multline*}
Denoting the maximum degree at time $j \eta \log n$ of the branching process by $D^{\text{max}}$, note that $D^{\text{max}} + 1 \le Z_n(j \eta \log n)$ and hence,
$$
\sum_{\ell=0}^{\infty}(\ell + 1)^2D_n(\ell,j \eta \log n) \le (D^{\text{max}} + 1)\sum_{\ell=0}^{\infty}(\ell + 1)D_n(\ell,j \eta \log n) \le Z_n(j \eta \log n)(2Z_n(j \eta \log n) -1).
$$
Using this in the above variance bound, we get
\begin{equation*}
\operatorname{Var}\left(\frac{1}{Z_n(j \eta \log n)}\sum_{\ell=0}^{\infty}\sum_{m=1}^{D_n(\ell,j \eta \log n)}\frac{1}{n^{\lambda_1^*\eta}}\xi^{(\ell)}_{f_1,m}(\eta \log n)\right) \le \frac{2Cn^{2\beta'\eta}Z^2(j \eta \log n)}{Z^2(j \eta \log n)n^{2\lambda_1^* \eta}} = \frac{2C}{n^{2(\lambda_1^*- \beta')\eta}} \rightarrow 0
\end{equation*}
as $n \rightarrow \infty$ and hence,
\begin{multline}\label{mainpre1}
\frac{1}{Z_n(j \eta \log n)}\sum_{\ell=0}^{\infty}\sum_{m=1}^{D_n(\ell,j \eta \log n)}\frac{1}{n^{\lambda_1^*\eta}}\xi^{(\ell)}_{f_1,m}(\eta \log n)\\
 - \frac{1}{Z_n(j \eta \log n)}\sum_{\ell=0}^{\infty}\sum_{m=1}^{D_n(\ell,j \eta \log n)}\frac{1}{n^{\lambda_1^*\eta}}\mathbb{E}\left(\xi^{(\ell)}_{f_1,m}(\eta \log n)\right) \overset{P}{\longrightarrow} 0.
\end{multline}
By Lemma \ref{bbt1}, we obtain $\beta \in (\lambda_1^*-1, \lambda_1^*)$ such that $w_{\ell}(\beta) = \int_0^{\infty}e^{- \beta s} \mu_{f_1}^{(\ell)}(ds) \le C(\beta)(\ell + 1)$. This implies for any $m, \ell$, % TODO: this should be $\beta \in (\beta_1, \lambda_1^*)$ right?
$
\E\left(\xi^{(\ell)}_{f_1,m}(\eta \log n)\right) \le C(\beta)n^{\beta \eta}(\ell + 1)
$
and consequently,
\begin{multline}\label{mainpre2}
\frac{1}{Z_n(j \eta \log n)}\sum_{\ell=0}^{\infty}\sum_{m=1}^{D_n(\ell, j \eta \log n)}\frac{1}{n^{\lambda_1^*\eta}}\E\left(\xi^{(\ell)}_{f_1,m}(\eta \log n)\right)\\
\le \frac{1}{n^{(\lambda_1^* - \beta)\eta}}  \frac{C(\beta)}{Z_n(j \eta \log n)} \sum_{\ell=0}^{\infty}(\ell + 1)D_n(\ell, j \eta \log n)
\le \frac{2C(\beta)}{n^{(\lambda_1^* - \beta)\eta}} \rightarrow 0
\end{multline}
as $n \rightarrow \infty$. From \eqref{mainpre1} and \eqref{mainpre2}, the proof of \eqref{mainprech}, and hence the lemma, is complete.
\end{proof}

\begin{lemma}\label{convd1}
Let $\phi \in \cC$ such that $\lim_{t \rightarrow \infty}e^{-\lambda_1^* t}m^{\phi}(t) = c_{\phi}$. \chnr{Fix any $\eta_0 \in (0, 1/(2C'))$, where $C'$ is the constant appearing in Theorem \ref{thm:1}.
%For $\ell \ge 0$, let $\lambda^{\phi}_{\ell}(\cdot)$ be defined as in \eqref{lpdef}.
Then for any $j \ge 0$, $\eta \in (0, \eta_0]$ and $a \in \mathbb{R}$, as $n \rightarrow \infty$}:
\begin{equation}\label{indass}
\frac{1}{n^{1+(j\eta_0 +\eta)\lambda_1^*}}\sum_{\ell = 0}^{\infty} D_n(\ell, j \eta_0 \log n) \lambda^{\phi}_{\ell}(\eta \log n + a) \overset{P}{\longrightarrow} c_{\phi}e^{\lambda_1^* a}\sum_{\ell=0}^{\infty} p_{\ell}^0 w_{\ell}.
\end{equation}
\end{lemma}

\begin{proof}
\chnr{
We will proceed by induction \chnr{on $j \ge 0$}.
\chsb{Suppose for some $j' \ge 0$, \eqref{indass} holds for all $0 \le j \le j'$, $\eta \in (0, \eta_0]$ and $a \in \mathbb{R}$}.
%Suppose we can show that for some $j \ge 0$, the assertion of the lemma holds for all $j' \le j$.
Taking $\phi(t) = \ind\set{t \ge 0}$ and $\eta = \eta_0$ and recalling
$
\lim_{t \rightarrow \infty} e^{-\lambda_1^* t}m_{f_1}(t) = \frac{1}{\lambda_1^* m_1^{\star}},
$
we obtain for \chsb{any $0 \le j \le j'$ and} any $a \in \mathbb{R}$,
\begin{equation}\label{totnum}
\frac{1}{n^{1+(j+1)\eta_0\lambda_1^*}}Z_n((j+1) \eta_0 \log n + a) \overset{P}{\longrightarrow} \frac{1}{\lambda_1^* m_1^{\star}}e^{\lambda_1^*a}\sum_{\ell=0}^{\infty} p_{\ell}^0 w_{\ell}.
\end{equation}
}
% MEOW: is there a reason to do strong induction here?
%We will proceed by induction \chnr{on $j \ge 0$}. Suppose we can show that for some $j \ge 0$, the assertion of the lemma holds for all $j' \le j$. Taking $\phi(t) = \ind\set{t \ge 0}$ and $\eta = \eta_0$ and recalling
%$
%\lim_{t \rightarrow \infty} e^{-\lambda_1^* t}m_{f_1}(t) = \frac{1}{\lambda_1^* m^{\star}},
%$
%we obtain for any $j' \le j$ and $a \in \mathbb{R}$,
%\begin{equation}\label{totnum}
%\frac{1}{n^{1+(j'+1)\eta_0\lambda_1^*}}Z_n((j'+1) \eta_0 \log n + a) \overset{P}{\longrightarrow} \frac{1}{\lambda_1^* m^{\star}}e^{\lambda_1^*a}\sum_{\ell=0}^{\infty} p_{\ell}^0 w_{\ell}.
%\end{equation}
Fix any $\phi \in \cC$. Note that for any $\eta \le \eta_0$,
\begin{multline}\label{estbr}
\left|\frac{1}{n^{1+((j'+1)\eta_0 + \eta)\lambda_1^*}}\sum_{\ell=0}^{\infty}D_n(\ell, (j'+1) \eta_0 \log n)\lambda^{\phi}_{\ell}(\eta \log n + a) - c_{\phi}e^{\lambda_1^* a}\sum_{\ell=0}^{\infty} p_{\ell}^0 w_{\ell}\right|\\
 \le \sum_{\ell=0}^{\infty}\frac{D_n(\ell, (j'+1)\eta_0 \log n)}{n^{1+(j'+1)\eta_0\lambda_1^*}}\left|\frac{\lambda^{\phi}_{\ell}(\eta \log n + a)}{n^{\eta \lambda_1^*}} -  c_{\phi}e^{\lambda_1^* a} w_{\ell}\right|\\
  + \ c_{\phi} e^{\lambda_1^* a} \left|\sum_{\ell=0}^{\infty}\frac{D_n(\ell, (j'+1)\eta_0 \log n)}{n^{1+(j'+1)\eta_0\lambda_1^*}}w_{\ell} - \sum_{\ell=0}^{\infty} p_{\ell}^0 w_{\ell}\right|.
\end{multline}
For any $\epsilon>0$, by Lemma \ref{conrt}, there exists $n_0 \ge 1$ \chsb{and $C''>0$ such that} for all $n \ge n_0$, \chsb{$\ell \ge 0$},
$$
\left|\frac{\lambda^{\phi}_{\ell}(\eta \log n + a)}{n^{\eta \lambda_1^*}} -  c_{\phi} e^{\lambda_1^* a}w_{\ell}\right| \le C'' e^{\lambda_1^* a} \epsilon(\ell+1)
$$
and hence,
\begin{multline*}
 \sum_{\ell=0}^{\infty}\frac{D_n(\ell, (j'+1)\eta_0 \log n)}{n^{1+(j'+1)\eta_0\lambda_1^*}}\left|\frac{\lambda^{\phi}_{\ell}(\eta \log n + a)}{n^{\eta \lambda_1^*}} -  c_{\phi}e^{\lambda_1^* a} w_{\ell}\right|\\
  \le C''e^{\lambda_1^* a}\epsilon\sum_{\ell=0}^{\infty}\frac{(\ell + 1)D_n(\ell, (j'+1)\eta_0 \log n)}{n^{1+(j'+1)\eta_0\lambda_1^*}} \le 2C''e^{\lambda_1^* a}\epsilon\frac{Z_n((j'+1)\eta_0 \log n)}{n^{1+(j'+1)\eta_0\lambda_1^*}}.
\end{multline*}
Therefore, using \eqref{totnum} \chsb{with $j=j'$}, \chnr{and as $\epsilon>0$ is arbitrary}, the first term in the bound \eqref{estbr} converges to zero in probability. To estimate the second term in \eqref{estbr}, consider the characteristic $\chi(t) = \sum_{\ell=0}^{\infty}w_{\ell}\ind\set{\xi_{f_1}(t)=\ell}$ and note that by Lemma \ref{bbt1}, $\chi \in \cC$. \chnr{Recall $Z^{\chi}_n$ from Section \ref{sec:proofs-sp} (see Notation (iv)) with $\cF_n(0)$ replaced by $\cF_n(j' \eta_0 \log n)$ (that is, time starting at $T_n + j' \eta_0 \log n$) and take $a = \eta_0 \log n$}. As $Z_n^{\chi}$ denotes the \chnr{aggregate $\chi$-score of all vertices} born in the interval $[j'\eta_0 \log n, (j'+1)\eta_0 \log n]$,
\begin{multline}\label{num1}
\frac{1}{n^{1+(j'+1)\eta_0\lambda_1^*}}\left|\sum_{\ell=0}^{\infty}D_n(\ell, (j'+1)\eta_0 \log n) w_{\ell} - Z_n^{\chi}\right| \le \frac{C(\lambda_1^*)}{n^{1+ (j'+1)\eta_0\lambda_1^*}}\sum_{\ell=0}^{\infty}(\ell +1)D_n(\ell, j', \eta_0)\\
=\frac{Z_n(j' \eta_0 \log n)}{n^{1+j'\eta_0\lambda_1^*}}\frac{C(\lambda_1^*)}{Z_n(j' \eta_0 \log n)n^{\eta_0\lambda_1^*}}\sum_{\ell=0}^{\infty}(\ell +1)D_n(\ell, j', \eta_0) \overset{P}{\longrightarrow} 0
\end{multline}
as $n \rightarrow \infty$, \chsb{which follows from Lemma \ref{prech} and by \eqref{totnum} with $j=j'-1$ for $j' \ge 1$ and the trivial observation that $\frac{Z_n(j' \eta_0 \log n)}{n^{1+j'\eta_0\lambda_1^*}}  \overset{P}{\longrightarrow} 0$ when $j'=0$}. Here $C(\lambda_1^*)$ is the constant appearing in Lemma \ref{bbt1}. 
\chnr{Recall $\eta_0$ is chosen such that $\frac{Ce^{C'\eta_0 \log n}}{\sqrt{n}} \rightarrow 0$ as $n \rightarrow \infty$, where $C,C'$ are the constants appearing in Theorem \ref{thm:1}. Thus, \chsb{recalling $\lambda_{\ell}^{\chi}(t) = \int_0^{t} m_{f_1}^{\chi}(t - s)\mu_{f_1}^{(\ell)}(ds)$}, by Theorem \ref{thm:1} and \eqref{totnum}},
\begin{multline}\label{num2}
\frac{1}{n^{1+(j'+1)\eta_0\lambda_1^*}}\left|Z_n^{\chi} - \sum_{\ell=0}^{\infty}D_n(\ell, j' \eta_0 \log n) \lambda_{\ell}^{\chi}(\eta_0 \log n)\right|
% = \frac{Z_n(j \eta_0 \log n)}{n^{1+(j+1)\eta_0\lambda_1^*}}\frac{1}{Z_n(j \eta_0 \log n)}\left|Z_n^{\chi} - \sum_{\ell=0}^{\infty}D_n(\ell, j \eta_0 \log n) \lambda_{\ell}^{\chi}(\eta_0 \log n)\right|\\
 \le \frac{Ce^{C'\eta_0\log n}}{n^{1+(j'+1)\eta_0\lambda_1^*}}\sqrt{Z_n(j' \eta_0 \log n)}\\
 \le \frac{Ce^{C'\eta_0\log n}}{\sqrt{n}}\sqrt{\frac{Z_n(j' \eta_0 \log n)}{n^{1+(j'+1)\eta_0\lambda_1^*}}}
  \overset{P}{\longrightarrow} 0.
\end{multline}
%where we recall $\lambda_{\ell}^{\chi}(t) = \int_0^{t} m_{f_1}^{\chi}(t - s)\mu_{f_1}^{(\ell)}(ds)$. 
By \eqref{num1} and \eqref{num2}, we obtain
\begin{multline}\label{use1}
\left|\sum_{\ell=0}^{\infty}\frac{D_n(\ell, (j'+1)\eta_0 \log n)}{n^{1+(j'+1)\eta_0\lambda_1^*}}w_{\ell} - \sum_{\ell=0}^{\infty}\frac{D_n(\ell, j' \eta_0 \log n)}{n^{1+(j'+1)\eta_0\lambda_1^*}}\lambda_{\ell}^{\chi}(\eta_0 \log n)\right|\\
 \le \frac{1}{n^{1+(j'+1)\eta_0\lambda_1^*}}\left|\sum_{\ell=0}^{\infty}D_n(\ell, (j'+1)\eta_0 \log n) w_{\ell} - Z_n^{\chi}\right| \\
 + \frac{1}{n^{1+(j'+1)\eta_0\lambda_1^*}}\left|Z_n^{\chi} - \sum_{\ell=0}^{\infty}D_n(\ell, j' \eta_0 \log n) \lambda_{\ell}^{\chi}(\eta_0 \log n)\right|\overset{P}{\longrightarrow} 0.
\end{multline}
Next, we will show that
\begin{equation}\label{chi}
e^{-\lambda_1^* t}m_{f_1}^{\chi}(t) \rightarrow 1 \ \text{ as } t \rightarrow \infty.
\end{equation}
To see this, first note that it follows from Assumption \ref{ass:attach-func} (ii) that there exists $\beta < \lambda_1^*$ such that $\E\left(\xi_{f_1}(t)\right) \le Ce^{\beta t}$. Moreover, $w_{\ell} \le C(\ell + 1)$ for all $\ell \ge 0$. These observations imply
\begin{multline*}
\sum_{k=0}^{\infty}\sup_{t \in [k, k+1]}\left[e^{-\lambda_1^* t}E(\chi(t))\right] \le C \sum_{k=0}^{\infty}\sup_{t \in [k, k+1]}\left[e^{-\lambda_1^* t}\sum_{\ell=0}^{\infty}(\ell + 1)\pr\left(\xi_{f_1}(t)=\ell\right)\right]\\
=C \sum_{k=0}^{\infty}\sup_{t \in [k, k+1]}\left[e^{-\lambda_1^* t}\E\left(\xi_{f_1}(t) + 1\right)\right] \le C'\sum_{k=0}^{\infty}\sup_{t \in [k, k+1]}\left[e^{-\lambda_1^* t}e^{\beta t}\right] \le C'e^{\beta}\sum_{k=0}^{\infty}e^{-(\lambda_1^* - \beta)k} < \infty
\end{multline*}
where $C, C'>0$ are constants. Thus, by Proposition 2.2 of \cite{nerman1981convergence} and Corollary \ref{soft}, it follows that
$$
\lim_{t \rightarrow \infty}e^{-\lambda_1^* t}m_{f_1}^{\chi}(t) = \frac{1}{\lambda_1^* m_1^{\star}}\sum_{\ell = 0}^{\infty}w_{\ell}\lambda_1^*\int_0^{\infty}e^{-\lambda_1^* s}\pr\left(\xi_{f_1}(s)=\ell\right)ds = \frac{1}{\lambda_1^* m_1^{\star}}\sum_{\ell = 0}^{\infty}w_{\ell}p^1_{\ell} = 1.
$$
Using this, the definition of $\lambda_{\ell}^{\chi}$, the fact that $\chi \in \cC$ and the induction hypothesis, we obtain
\begin{equation}\label{use2}
\frac{1}{n^{1+(j'+1)\eta_0\lambda_1^*}}\sum_{\ell = 0}^{\infty} D_n(\ell, j' \eta_0 \log n) \lambda^{\chi}_{\ell}(\eta_0 \log n) \overset{P}{\longrightarrow} \sum_{\ell=0}^{\infty} p_{\ell}^0 w_{\ell} \ \ \ \text{ as } n \rightarrow \infty.
\end{equation}
From \eqref{use1} and \eqref{use2}, the second term in \eqref{estbr} goes to $0$ \chnr{in probability} as $n \rightarrow \infty$ which shows
$$
\left|\frac{1}{n^{1+((j'+1)\eta_0 + \eta)\lambda_1^*}}\sum_{\ell=0}^{\infty}D_n(\ell, (j'+1) \eta_0 \log n)\lambda^{\phi}_{\ell}(\eta \log n + a) - c_{\phi}e^{\lambda_1^* a}\sum_{\ell=0}^{\infty} p_{\ell}^0 w_{\ell}\right| \overset{P}{\longrightarrow} 0
$$
\chsb{establishing \eqref{indass} for all $j \le j'+1$}. \chsb{\eqref{indass} holds for $j=0$ by Lemma \ref{limconv}}. Thus, the lemma is proved.
\end{proof}

\begin{lemma}\label{convd2}
For any $k \ge 0, \theta >0$ and $a \in \mathbb{R}$, as $n \rightarrow \infty$:
$$
n^{-(1+\theta \lambda_1^*)}Z_n(\theta \log n + a) \overset{P}{\longrightarrow}e^{\lambda_1^*a}\sum_{\ell=0}^{\infty} p_{\ell}^0 w_{\ell} / \lambda_1^* m_1^{\star}, \ \ \ \   \    \ \frac{D_n(k, \theta \log n + a)}{Z_n(\theta \log n + a)} \overset{P}{\longrightarrow} p^1_k.
$$
\end{lemma}

\begin{proof}
%The first assertion follows by the exact argument used to derive \eqref{totnum} \chnr{after writing $\theta = j\eta_0 + \eta$ for some $j \ge 0$, $\eta \in [0, \eta_0)$}. To prove the second assertion, fix any $k \ge 0$. \chnr{Take $\eta_0>0$ in Lemma \ref{convd1} small enough so that $Ce^{C'\eta_0 \log n}\epsilon^{-2} n^{-(\omega - \widetilde{\theta} - \frac{1}{2})} \rightarrow 0$, where $C,C', \omega, \widetilde{\theta}$ are as in Lemma \ref{lem:Nt-sumNlam-close}}. Let $j \ge 0$, $\eta \in [0, \eta_0)$ such that $\theta = j\eta_0 + \eta$. Recall that the probability bound obtained in Lemma \ref{lem:Nt-sumNlam-close} conditionally on $\cF_n(0)$ was in terms of deterministic constants and $n$, the total number of vertices at time $0$. Thus, replacing $\cF_n(0)$ by $\cF_n(j \eta_0 \log n)$ and time starting from $T_n + j \eta_0 \log n$, Lemma \ref{lem:Nt-sumNlam-close} (\chnr{with $n$ replaced by $Z_n(j \eta_0 \log n)$, the total number of vertices at time $j \eta_0 \log n$}) implies,

\chnr{Note for any $\eta_0 > 0$ we can write  $\theta = j\eta_0 + \eta$ for some $j \ge 0$.}
The first assertion follows by the argument used to derive \eqref{totnum}. To prove the second assertion, fix any $k \ge 0$. \chnr{Take $\eta_0>0$ in Lemma \ref{convd1} small enough so that $Ce^{C'\eta_0 \log n}\epsilon^{-2} n^{-(\omega - \widetilde{\theta} - \frac{1}{2})} \rightarrow 0$, where $C,C', \omega, \widetilde{\theta}$ are as in Lemma \ref{lem:Nt-sumNlam-close}}. Recall that the bound obtained in Lemma \ref{lem:Nt-sumNlam-close} conditionally on $\cF_n(0)$ was in terms of deterministic constants and $n$, the total number of vertices at time $0$. Replacing $\cF_n(0)$ by $\cF_n(j \eta_0 \log n)$ and time starting from $T_n + j \eta_0 \log n$, Lemma \ref{lem:Nt-sumNlam-close} (\chnr{with $n$ replaced by $Z_n(j \eta_0 \log n)$, the total number of vertices at time $j \eta_0 \log n$}) implies,
$$
\frac{1}{Z_n(j \eta_0 \log n)}D_n(k, \theta \log n + a) - \frac{1}{Z_n(j \eta_0 \log n)}\sum_{\ell = 0}^{\infty} D_n(\ell, j \eta_0 \log n) \lambda_{\ell}^{(k)}(\eta \log n + a)  \overset{P}{\longrightarrow} 0, \ \ \text{ as } n \rightarrow \infty.
$$
From Lemma \ref{convd1} (taking $\phi(t) = \ind\set{t \ge 0}$), $Z_n(j \eta_0 \log n)/Z_n(\theta \log n + a) \overset{P}{\longrightarrow} 0$ if $\eta>0$, and\\
 $Z_n(j \eta_0 \log n) / Z_n(\theta \log n + a) \overset{P}{\longrightarrow} e^{-\lambda_1^* a}$ if $\eta=0$ and thus, multiplying both sides of the above by \\
 $Z_n(j \eta_0 \log n) / Z_n(\theta \log n + a)$, we obtain
\begin{equation}\label{deter1}
\frac{D_n(k, \theta \log n + a)}{Z_n(\theta \log n + a)} - \frac{1}{Z_n(\theta \log n + a)}\sum_{\ell = 0}^{\infty} D_n(\ell, j \eta_0 \log n) \lambda_{\ell}^{(k)}(\eta \log n + a)  \overset{P}{\longrightarrow} 0, \ \ \text{ as } n \rightarrow \infty.
\end{equation}
Taking $\phi(t) = \ind\set{\xi_{f_1}(t) = k}$, we see that $\lambda^{\phi}_{\ell} = \lambda^{(k)}_{\ell}$ for each $\ell \ge 0$. Moreover, recall from \eqref{eq:m_lims}\\
$
\lim_{t \rightarrow \infty} e^{-\lambda_1^* t}m^{(k)}_{f_1}(t) = p^1_k / \lambda_1^* m_1^{\star}.
$
Thus, from Lemma \ref{convd1},
\begin{equation}\label{deter2}
\frac{1}{n^{1+\theta\lambda_1^*}}\sum_{\ell = 0}^{\infty} D_n(\ell, j \eta_0 \log n) \lambda^{(k)}_{\ell}(\eta \log n + a) \overset{P}{\longrightarrow} \frac{p^1_k}{\lambda_1^* m_1^{\star}}e^{\lambda_1^* a}\sum_{\ell=0}^{\infty} p_{\ell}^0 w_{\ell}.
\end{equation}
Using \eqref{deter2} and the first assertion of the lemma in \eqref{deter1}, the second assertion follows.
\end{proof}
\noindent Let $a_0 := \frac{1}{\lambda_1^*} \log \left(\frac{\lambda_1^* m_1^{\star}}{\sum_{\ell=0}^{\infty} p_{\ell}^0 w_{\ell}}\right)$ and
$
T_n^{\theta} := T_{n^{1 + \lambda_1^*\theta}}
$
be the first time the branching process has $n^{1 + \lambda_1^*\theta}$ vertices. % Note that as we are conditioning on $T_n$ so $T_n^{\theta}$ is the time it takes to grow from size $n$ to $n^{1 + \lambda_1^*}$.
\begin{lemma}\label{convtime}
 $T_n^{\theta} - \theta \log n \overset{P}{\longrightarrow} a_0$. % $T_n^{\theta} - (1 + \lambda_1^*\theta) \log n \overset{P}{\longrightarrow} a_0$.
\end{lemma}
\begin{proof}
Follows immediately from the first assertion of Lemma \ref{convd2}.
\end{proof}
\begin{theorem}\label{qbbmain}
For any $k \ge 0$, $\theta > 0$, as $n \rightarrow \infty$,
$
n^{-(1 + \lambda_1^*\theta)}D_n(k, T_n^{\theta}) \overset{P}{\longrightarrow} p^1_k.
$
\end{theorem}
\begin{proof}
In the proof, we will abbreviate $z^*= \frac{1}{\lambda_1^* m_1^{\star}}\sum_{\ell=0}^{\infty} p_{\ell}^0 w_{\ell}$. Fix any $k \ge 0$, $\theta > 0$. Take any $\epsilon \in (0,1)$. By the same argument as in the proof of Lemma \ref{lem:sup-sup-Nt-Ntj},
\begin{equation}\label{bbm0}
\sup_{t \le 2\epsilon}|D_n(k, \theta \log n + a_0 - \epsilon + t) - D_n(k, \theta \log n + a_0 - \epsilon)| \le \left(Z_n(\theta \log n + a_0 + \epsilon) - Z_n(\theta \log n + a_0 - \epsilon)\right) + Y_n.
\end{equation}
where, conditionally on $\cF_n(\theta \log n + a_0 - \epsilon)$, $Y_n$ has the same distribution as the random variable $\sum_{\ell = 0}^{k} \text{Bin}\left(D_n\left(\ell, \theta \log n + a_0 - \epsilon \right), q_{\ell} \left(2\epsilon\right) \right)$.
Observe that by the first assertion in Lemma \ref{convd2}, for small enough $\epsilon$,
\begin{equation}\label{bbm1}
n^{-(1+ \lambda_1^*\theta)}\left(Z_n(\theta \log n + a_0 + \epsilon) - Z_n(\theta \log n + a_0 - \epsilon)\right) \overset{P}{\longrightarrow} e^{\lambda_1^*  \epsilon} - e^{-\lambda_1^*\epsilon} \le 4\lambda_1^*\epsilon.
\end{equation}
Note that for any $C>0$,
\begin{multline}\label{bbm2}
\pr\left(Y_n > C\sqrt{\epsilon}n^{1 + \lambda_1^*\theta}\right)
 \le \pr\left(Y_n > C\sqrt{\epsilon}n^{1 +  \lambda_1^*\theta}, Z_n(\theta \log n + a_0 - \epsilon) \le \epsilon^{-1/2}n^{1+  \lambda_1^*\theta}\right)\\
  + \pr\left(Z_n(\theta \log n + a_0 - \epsilon) > \epsilon^{-1/2}n^{1+  \lambda_1^*\theta}\right).
\end{multline} 
For $\epsilon$ sufficiently small, by the first assertion of Lemma \ref{convd2}, as $n \rightarrow \infty$,
\begin{equation}\label{bbm21}
\pr\left(Z_n(\theta \log n + a_0 - \epsilon) > \epsilon^{-1/2}n^{1+  \lambda_1^*\theta}\right) \rightarrow 0.
\end{equation}
Let $\cH_n := \cF_n(\theta \log n + a_0 - \epsilon)$. Using Lemma \ref{qdec},
\begin{multline*}
\E\left(Y_n \mid \cH_n\right) = \sum_{\ell=0}^kD_n\left(\ell, \theta \log n + a_0 - \epsilon \right)q_{\ell} \left(2\epsilon\right) \le C'\epsilon \sum_{\ell=0}^k (\ell + 1)D_n\left(\ell, \theta \log n + a_0 - \epsilon \right)\\
\le 2C'\epsilon Z_n(\theta \log n + a_0 - \epsilon).
\end{multline*}
Thus, choosing $C > 4C'$, using Chebychev's inequality, conditionally on $\cH_n$ on the event $\{ Z_n(\theta \log n + a_0 - \epsilon)\le \epsilon^{-1/2}n^{1+  \lambda_1^*\theta}\}$,
\begin{multline}\label{bbm22}
\pr\left(Y_n > C\sqrt{\epsilon}n^{1 +  \lambda_1^*\theta} \mid \cH_n\right)
 \le \pr\left(Y_n - \E\left(Y_n \mid \cH_n\right)> \frac{C}{2}\sqrt{\epsilon}n^{1 +  \lambda_1^*\theta}\mid \cH_n\right)\\
  \le \frac{4\operatorname{Var}\left(Y_n \mid \cH_n\right)}{C^2 \epsilon n^{2(1+  \lambda_1^*\theta)}}
  = \frac{4\sum_{\ell=0}^kD_n\left(\ell, \theta \log n + a_0 - \epsilon \right)q_{\ell} \left(2\epsilon\right)(1 -q_{\ell} \left(2\epsilon\right))}{C^2 \epsilon n^{2(1+  \lambda_1^*\theta)}}\\
  \le \frac{4C'\epsilon \sum_{\ell=0}^k (\ell + 1)D_n\left(\ell, \theta \log n + a_0 - \epsilon \right)}{C^2 \epsilon n^{2(1+  \lambda_1^*\theta)}}
  \le \frac{8C' Z_n(\theta \log n + a_0 - \epsilon)}{C^2 n^{2(1+  \lambda_1^*\theta)}}
   \le \frac{8C'}{C^2 \sqrt{\epsilon} n^{1+  \lambda_1^*\theta}} \rightarrow 0 \ \ \ \text{ as } n \rightarrow \infty.
\end{multline}
Using \eqref{bbm21} and \eqref{bbm22} in \eqref{bbm2}, we conclude
\begin{equation}\label{bbm3}
\pr\left(Y_n > C\sqrt{\epsilon}n^{1 + \lambda_1^*\theta}\right) \rightarrow 0 \ \ \ \text{ as } n \rightarrow \infty.
\end{equation}
Using \eqref{bbm1}, \eqref{bbm3} and \eqref{bbm0}, we conclude that there exist $C_0>0, \epsilon_0>0$ such that for all $\epsilon \in (0, \epsilon_0)$,
\begin{equation}\label{bbm4}
\pr\left(\sup_{t \le 2\epsilon}|D_n(k, \theta \log n + a_0 - \epsilon + t) - D_n(k, \theta \log n + a_0 - \epsilon)|  > C_0\sqrt{\epsilon}n^{1 + \lambda_1^* \theta}\right)\rightarrow 0 \ \ \ \text{ as } n \rightarrow \infty.
\end{equation}
From \eqref{bbm4} and Lemma \ref{convtime}, as $n \rightarrow \infty$,
\begin{multline}\label{bbm5}
\pr\left(|D_n(k, T^{\theta}_n) - D_n(k, \theta \log n + a_0 - \epsilon)|  > C_0\sqrt{\epsilon}n^{1 + \lambda_1^* \theta}\right) \le \pr\left(\left|T^{\theta}_n - \theta \log n - a_0\right| > 2\epsilon\right)\\
 + \pr\left(\sup_{t \le 2\epsilon}|D_n(k, \theta \log n + a_0 - \epsilon + t) - D_n(k, \theta \log n + a_0 - \epsilon)|  > C_0\sqrt{\epsilon}n^{1 + \lambda_1^* \theta}\right)
 \rightarrow 0.
\end{multline}
For any $\epsilon>0$,
\begin{multline}\label{bbm6}
\pr\left(\left|\frac{D_n(k, T^{\theta}_n)}{n^{1 + \lambda_1^* \theta}} - p^1_k\right| > 2C_0\sqrt{\epsilon}\right) \le \pr\left(\left|\frac{D_n(k, T^{\theta}_n)}{n^{1 + \lambda_1^* \theta}} - \frac{D_n(k, \theta \log n + a_0 - \epsilon)}{n^{1 + \lambda_1^* \theta}}\right| > C_0\sqrt{\epsilon}\right)\\
 + \pr\left(\left|\frac{D_n(k, \theta \log n + a_0 - \epsilon)}{n^{1 + \lambda_1^* \theta}}- p^1_k\right| > C_0\sqrt{\epsilon}\right).
\end{multline}
By Lemma \ref{convd2},
$$
\frac{D_n(k, \theta \log n + a_0 - \epsilon)}{n^{1 + \lambda_1^* \theta}} = \frac{D_n(k, \theta \log n + a_0 - \epsilon)}{Z_n(\theta \log n + a_0 - \epsilon)} \frac{Z_n(\theta \log n + a_0 - \epsilon)}{n^{1 + \lambda_1^* \theta}}  \overset{P}{\longrightarrow} p^1_k e^{-\lambda_1^* \epsilon},
$$
and therefore, there is \chnr{an} $\epsilon_1 \le \epsilon_0$ such that for all $\epsilon \in (0, \epsilon_1)$,
\begin{equation}\label{bbm61}
\left| D_n(k, \theta \log n + a_0 - \epsilon) n^{-(1 + \lambda_1^* \theta)} - p^1_k\right|\overset{P}{\longrightarrow} p^1_k (1 - e^{-\lambda_1^* \epsilon}) \le p^1_k \lambda_1^* \epsilon < C_0\sqrt{\epsilon}.
\end{equation}
For $\epsilon \in (0, \epsilon_1)$, using \eqref{bbm5} and \eqref{bbm61} in \eqref{bbm6}, we conclude
$
\pr\left(\left|\frac{D_n(k, T^{\theta}_n)}{n^{1 + \lambda_1^* \theta}} - p^1_k\right| > 2C_0\sqrt{\epsilon}\right) \rightarrow 0$ as $n \rightarrow \infty$ proving the theorem.
\end{proof}

% \noindent {\bf Proof idea:}
% For uniform $\leadsto$ linear: Initial embedding (Yule process) grows like $e^t$ so takes $\gamma \log{n}$ to get to size $n^\gamma$. After change point grows like $e^{(2+\alpha)t}$ starting with $n^\gamma$ at time zero so time to get to size $n$ should solve $n^\gamma e^{(2+\alpha)t} \approx n$ which gives $t_n\approx \frac{1-\gamma}{2+\alpha}\log{n}$. After change point, typical degree grows like $e^t$. So if we look at the root (whose degree by change point will be $\approx \gamma \log{n}$), we get that this should have degree $(\gamma \log{n})\cdot  e^{\frac{1-\gamma}{2+\alpha}\log{n}}$ which gives lower bound in the Theorem. Some sort of exponential tail bound argument needed for upper bound.

\subsection{Proof of Theorem \ref{thm:max-deg-quick-big-bang}} 
\color{banared}
We prove (a) of the theorem; (b) and (c) follow via straightforward modifications of these arguments. \chsb{For (a), construct the continuous time branching process $\BP_{\mvtheta}(\cdot)$ with change point as in Section \ref{cpembedding} with $\tau = n^{\gamma}$. \bbb{To ease notation later in the section, write $\BP_n(\cdot) := \BP_{\mvtheta}(\cdot)$. \bbb{Thus $\BP_n(T_{n^\gamma})$ is a random tree obtained by running a continuous time branching process with attachment function $f_0 \equiv 1$ till it reaches size $n^\gamma$ after which all vertices switch to reproducing using attachment function $f_1$ as in (a) of the Theorem.}  We are interested in the random tree $\cT_n = \BP_n(T_n) $}, where as before for any $m$, $T_m := \inf\{t \ge 0: |\BP_n(t)| = m\}$.} %recall that we first grow the tree using the uniform attachment scheme with $f_0 \equiv 1$ till it is of size $n^\gamma$ and then use the preferential attachment scheme.
%we assume that $\cT_n^{\mvtheta}$ has been constructed as follows:
% \begin{enumeratea}
%     \item
%{\bf (a)} Generate the genealogical tree according to a rate one Yule process $\set{\cT^{\sss \mathrm{Yule}}(t):t\geq 0}$ as in Definition  \ref{def:yule-process} run \chnr{forever}; 
%{\bf (b)} To obtain $\cT_n^{\mvtheta}$,  let $\cT_{n^{\gamma}} = \cT^{\sss \mathrm{Yule}}(T_{n^{\gamma}})$. Now every vertex in $\cT_{n^{\gamma}}$ switches to offspring dynamics giving birth to children at rate corresponding to the number of children $+1+\alpha$ (thus modulated by the function $f_1$). Write $\BP_n(\cdot)$ for the combined process and  stop this process at time $T_n$ and let $\cT_n^{\mvtheta} = \BP_n(T_n) $.   

%\end{enumeratea}
 % The following describes asymptotics for the above continuous time construction.
 
% \chsb{Fix any sequence $\omega_n \rightarrow \infty$ such that $\omega_n = o(\log{n})\uparrow \infty$.}
 \begin{prop}
    \label{prop:bb-asymp-gr}
    For the process $\BP_n(\cdot)$ as constructed above:
    \begin{enumeratea}
        \item The stopping time $T_{n^{\gamma}}$ satisfies,
        $T_{n^{\gamma}} - \gamma\log{n} \convas \tilde{W},$
        where $\tilde{W} = -\log{W}$ and $W \sim \exp(1)$. 
        \item Let $\omega_n\to \infty$ arbitrarily slowly. Then there exists a constant $C>0$ independent of $\omega_n$ such that 
        \[\pr\left(\sup_{t\geq 0}\left|n^{-\gamma} e^{-(2+\alpha)t}|\BP_n(t + T_{n^{\gamma}})| -1\right| > \omega_n n^{- \gamma/2}\right) \leq C/\omega_n^2. \]
        In particular whp as $n\to\infty$,
        \bbb{$\left|(T_n - T_{n^\gamma}) - (1-\gamma)\log{n} /(2+\alpha) \right|\leq \omega_n n^{-\gamma/2}.$}
    \end{enumeratea}
 \end{prop}

\begin{proof} 
Part (a) follows from Lemma \ref{lem:yule-prop} \chsb{upon noting that $T_{n^{\gamma}}$ has the same distribution as the hitting time of $n^{\gamma}$ by a Yule process with rate $1$}. To prove (b), recall that for $ t> T_{n^{\gamma}}$, all individuals switch to offspring dynamics modulated by $f_1$. For the rest of the proof, we proceed \bbb{conditional on $\BP_n(T_{n^\gamma})$}. % the history of the process till time $T_{n^{\gamma}}$. 
Using Proposition \ref{prop:mean-PA-model}, the following two processes are martingales
\begin{equation*}
    M_1(t) := \left(e^{-(2+\alpha)t}|\bbb{\BP_n(t+T_{n^\gamma})}| - n^{\gamma}\right) + \left(1-e^{-(2+\alpha)t} \right)/(2+\alpha), \qquad t\geq 0,
\end{equation*}
\begin{equation*}
    M_2(t) := e^{-2(2+\alpha)t}|\bbb{\BP_n(t + T_{n^\gamma})}|^2 - \int_0^{t}\alpha e^{-2(2+\alpha)s}|\bbb{\BP_n(s +T_{n^\gamma})}|ds - e^{-2(2+\alpha)t}/2(2+\alpha), \quad t\geq 0,
\end{equation*}
Using these expressions, it can be deduced that
$
\sup_{t\ge 0}\E\left(M_1^2(t)\right) \le Cn^{\gamma}
$
for some constant $C>0$.  Doob's $\bL^2$-maximal inequality then proves the first assertion of Proposition \ref{prop:bb-asymp-gr} (b) which then results in the second assertion \bbb{in (b)}. 
\end{proof}

\bbb{We now construct two approximating processes $\BP_n^+$ and $\BP_n^-$ for $\cT_n^{\mvtheta}$ (grown completely in continuous time).} Fix constant $B>0$ \chsb{and sequence $\omega_n \rightarrow \infty$ such that $\omega_n = o(\log{n})\uparrow \infty$}. For the rest of this Section let $t_n^{\pm}:= \frac{(1-\gamma)}{(2+\alpha)}\log{n} \pm \frac{\omega_n}{n^{\gamma/2}}$. Define the process $\set{\BP_n^+(t): 0\leq t\leq \gamma \log{n}+B + t_n^{+} }$ as follows: 
{\bf (a)} Run a \chsb{continuous time branching process driven by $f_0(\cdot) \equiv 1$} for time $\gamma\log{n}+B$;
{\bf(b) } \chsb{After this time, every vertex} switches dynamics so that it reproduces at rate equal to the number of children $+1+\alpha$. Run this process for {\bf an additional} time $t_n^+$. Write $\tilde{\cT}_n^+(B,\omega_n) = \BP_n^+(\gamma \log{n}+B + t_n^{+})$ for the random rooted tree at the end of this process.    Analogously define $\set{\BP_n^-(t): 0\leq t\leq \gamma \log{n} -B + t_n^{-} }$ and $\tilde{\cT}_n^-(B,\omega_n):=\BP_n^-(\log{n} -B + t_n^{-})$ where in the above construction we wait till time \chsb{$\gamma \log{n} -B$} before switching dynamics and run the new dynamics for additional time $t_n^{-}$. 

By Proposition \ref{prop:bb-asymp-gr}, given any $\eps> 0$ we can choose a constant $B = B(\eps)$ \chsb{for which} we can produce a coupling between $\cT_n$ and $\tilde{\cT}^+_n(B,\omega_n)$ such that for all large $n$, \chr{with probability at least $1-\eps$,} $\cT_n \subseteq \tilde{\cT}^+_n(B,\omega_n)$ where we see the object on the left as a subtree of the object on the right with the same root. A similar assertion holds with $ \tilde{\cT}^-_n(B,\omega_n)\subseteq\cT_n $. Using these couplings, the following proposition completes the proof of \chsb{part (a) of} Theorem \ref{thm:max-deg-quick-big-bang} with part (a) of the proposition proving the lower bound while part (b) proving the upper bound. \chsb{In the following, we will denote the root of the respective trees by $\rho^*$.}

\begin{prop}\label{noname}
    Fix $B >0$ and $\omega_n = o(\log{n})\uparrow \infty$. 
    \begin{enumeratea}
        \item Consider the degree of the root $D_n^{-}(\rho^*)$ in $\tilde{\cT}^-_n(B,\omega_n)$.  Then \chsb{$D_n^{-}(\rho^*) \ge \frac{\gamma}{4} n^{(1-\gamma)/(2+\alpha)}\log{n}$} whp as $n \rightarrow \infty$. 
        \item Consider the maximal degree $M_n^{+}(1)$ in $\tilde{\cT}^+_n(B,\omega_n)$. Then $\exists$ \bbb{constant} $ C >0$ such that whp as $n\to\infty$, $M_n^+(1) \ll C n^{(1-\gamma)/(2+\alpha)}(\log{n})^2 $. 
    \end{enumeratea}
\end{prop}

\noindent {\bf Proof:} We start with (a). Each individual in the \chsb{original branching process driven by $f_0(\cdot) \equiv 1$ before time $\gamma \log{n} -B$} reproduces according to a rate one Poisson process. In particular standard bounds for a Poisson random variable imply that the \chsb{degree of the root in the branching process at time $\gamma\log{n}-B$, denoted by $\deg_n(\rho^*,\gamma\log{n}-B)$,} satisfies \begin{equation}
\label{eqn:conc-root-deg-bef-cp}
\chsb{\deg_n(\rho^*,\gamma\log{n}-B) \ge \frac{3}{4}\gamma\log{n} \ \text{whp as} \ n \rightarrow \infty .}
\end{equation} 
Now let \chsb{$\set{Y_{(i)}(\cdot):i\geq 1}$} be a collection of independent rate one Yule processes.  Comparing rates \bbb{for the evolution of} the degree of the root after $\gamma\log{n}-B$ we get that 
\begin{equation}\label{domlow}
\chsb{D_n^{-}(\rho^*) \sustod \sum_{i=1}^{\deg_n(\rho^*,\gamma\log{n}-B)} Y_{(i)}(t_n^-).}
\end{equation}        
\chsb{Using \eqref{eqn:conc-root-deg-bef-cp}, \eqref{domlow}, Lemma \ref{lem:yule-prop} and standard lower tail bounds for the Geometric distribution \cite[Theorem 3.1]{janson2018tail} finishes the proof.}

Let us now prove (b). Recall that after the change point, dynamics are modulated by $f_1(\cdot) := \cdot+1+\alpha$. Let $A$ denote the smallest integer $\geq \alpha +1$. \bbb{Let $\xi_{f_1}$ be point process associated with $f_1$ as in \eqref{eqn:xi-f-def}}.  Comparing rates we see that 
$
\bbb{\xi_{f_1}(\cdot)\stod \sum_{i=1}^{A+2} Y_{(i)}(\cdot)},
$
where as before $\set{Y_{(i)}(\cdot):i\geq 1}$ is a collection of independent rate one Yule processes. For every vertex {$v \in \tilde{\cT}_n^+(B,\omega_n)$ }write $\deg_n(v)$ for the {final}  degree of the vertex at time \chsb{$\gamma\log{n}+B +t_n^+$} when we have finished constructing the process $\BP_n^+(\cdot)$. {As below Theorem \ref{ratewocp}, for any $v\in \BP_n^{+}$, let $\sigma_v$ denote the time of birth of vertex $v$ into the system}.  We split the proof of (b) into two cases ({loosely corresponding to the maximal degree of vertices after and before change point respectively}):

\noindent{\bf (b1) Maximal degree for vertices born after \chsb{$\gamma\log{n}+B$}:} Define \chnr{the following collection of vertices}
\[\dL_n = \set{v\in \tilde{\cT}_n^+(B,\omega_n) : \bbb{\sigma_v}\in [\gamma\log{n}+B ,\;\; \gamma\log{n}+B +t_n^+],\;\; \deg_n(v) > C(A+2) n^{\frac{1-\gamma}{2+\alpha}} (\log{n})^2 },\]
\bbb{where $A$ is above is the smallest integer $\geq \alpha +1$ and $C$ is an appropriate constant chosen later in the proof}. \bbb{We now show} that we can choose $C$ such that $\E(|\dA_n|) \to 0, \mbox{as } n\to\infty$. This would then imply
\begin{equation}
\label{eqn:ts-dbn-zero}
    \pr(\exists v \in \tilde{\cT}_n^+(B,\omega_n), \bbb{\sigma_v}\geq \gamma\log{n}+B,\;\; \deg_n(v) > C (A+2) n^{\frac{1-\gamma}{2+\alpha}} (\log{n})^2 ) \to 0.
\end{equation}

For the rest of the proof let $k_n^\prime := C(A+2) n^{\frac{1-\gamma}{2+\alpha}} (\log{n})^2$ and $k_n = k_n^\prime/(A+2)  = C n^{\frac{1-\gamma}{2+\alpha}} (\log{n})^2 $.  % and \chsb{for $t \in [0, \gamma \log n + B + t_n^+]$,} let $\tilde{\cT}_n^+(t)$ denote the \chsb{dominating tree process constructed before Proposition \ref{noname}} at time $t$. 
Fix $s\geq 0$ and consider a vertex born at some time $s+ \gamma\log{n} +B \in [\gamma\log{n}+B ,\;\; \gamma\log{n}+B +t_n^+]$. Thus this vertex has time $t_n^+ -s$ to evolve its degree.  Using the bound on $\xi_{f_1}$ namely the offspring process of each new vertex born at  $t > \gamma\log n + B$ by a sum of Yule process above, by Lemma \ref{lem:yule-prop} the \chr{probability that} such a vertex has degree greater than $(A+2) k_n$ by time $t_n^+$ is bounded by
$\mathbb{P}(\text{\chnr{geom}}(e^{-(t_n^+ - s)}) \ge (A+2) k_n) \le e^{-k_n e^{t_n^+ -s}}.$
Next note that for any $t\geq \gamma \log{n} +B$, new vertices are produced at rate $(2 + \alpha)|\BP_n^+(t)| - 1$. As in the proof of Proposition \ref{prop:bb-asymp-gr}, the process $M(s) := e^{-(2+\alpha)s}|\BP_n^+(s+\gamma \log{n} +B)| +{(2+\alpha)^{-1}}e^{-(2+\alpha)s},  s\geq 0$ is a martingale. Noting $\E |\BP_n^+(\gamma\log n + B)| = e^B n^{\gamma}$ we get that 
$\E |\BP_n^+(s+\gamma \log{n} +B)| \leq  C' n^{\gamma} e^{(2 + \alpha)s} \text{ for } 0\leq s\leq t_n^+$
where $C'$ is a constant depending only on $B, \alpha$.  Thus, 
$$\E(|\dL_n|) \le C'' n^{\gamma} \int_0^{t^+_n} e^{ -k_n e^{ -(t_n^+ -s)}} e^{(2 + \alpha) s } ds,$$
 where $C''$ depends only on $B, \alpha$. The following completes the proof of \eqref{eqn:ts-dbn-zero}.

\begin{lemma}\label{lem:in-to-zero} $
I_n := n^{\gamma} \int_0^{t^+_n} e^{-C (\log n)^2 n^{\frac{1-\gamma}{2+\alpha}} e^{-(t_n^+-s)}} e^{(2 + \alpha) s } ds
\to0 $ for sufficiently large $C$ as $n \to \infty$. 
\end{lemma}

% where $t_n^+ = \frac{1-\gamma}{2+\alpha} \log n + \frac{w_n}{n^{\gamma / 2}}$ and $w_n = o\left(n^{\gamma / 2} \right)$.

\begin{proof}

Writing  $a := \frac{1-\gamma}{2+\alpha}$ and $b := 2 + \alpha$, algebraic manipulations result in:
\begin{align*}
%I_n &= n^{\gamma} (\log n)^{-2b}  e^{b \frac{w_n}{n^{\gamma / 2}}}   \int_{C (\log n)^{2}  e^{-\frac{w_n}{n^{\gamma / 2}}}}^{C (\log n)^{2} (1 + n^a e^{\frac{w_n}{n^{\gamma / 2}}})} e^{-v} v^{b-1} dv \\
%&  \le  n^{\gamma} (\log n)^{-2b}  e^{b \frac{w_n}{n^{\gamma / 2}}} \Gamma\left(b,C (\log n)^{2}  e^{-\frac{w_n}{n^{\gamma / 2}}}\right):=\cE_n.
%I_n &= n^{\gamma} (\log n)^{-2b}  e^{b \frac{w_n}{n^{\gamma / 2}}}   \int_{C (\log n)^{2}  e^{-\frac{w_n}{n^{\gamma / 2}}}}^{C (\log n)^{2} (1 + n^a e^{\frac{w_n}{n^{\gamma / 2}}})} e^{-v} v^{b-1} dv \\
I_n \le  n^{\gamma} (\log n)^{-2b}  e^{b \frac{w_n}{n^{\gamma / 2}}} \Gamma\left(b,C (\log n)^{2}  e^{-\frac{w_n}{n^{\gamma / 2}}}\right):=\cE_n.
\end{align*}
where $\Gamma(b, z) = \int_{z}^{\infty} e^{-t}t^{b-1} dt$ is the upper incomplete Gamma function. It is known that\\
 $\Gamma(b, z) = \Omega(z^{b-1} e^{-z})  \text{ as } z \to \infty$. Thus $\cE_n  \sim n^{\gamma - C \log n e^{- \frac{w_n}{n^{\gamma / 2}}}} (\log n)^{-2} e^{- \frac{w_n}{n^{\gamma / 2}}} \to 0.$
%\begin{align*}
%\cE_n % &\le  n^{\gamma} (\log n)^{-2b}  e^{b \frac{w_n}{n^{\gamma / 2}}} \Omega\left(  \left((\log n)^{2}  e^{-\frac{w_n}{n^{\gamma / 2}}} \right)^{b-1} e^{-C (\log n)^{2}  e^{-\frac{w_n}{n^{\gamma / 2}}}}\right) \\
%% & \sim n^{\gamma} (\log n)^{-2} e^{- \frac{w_n}{n^{\gamma / 2}}} e^{-C (\log n)^{2} e^{- \frac{w_n}{n^{\gamma / 2}}} } \\
%& \sim n^{\gamma - C \log n e^{- \frac{w_n}{n^{\gamma / 2}}}} (\log n)^{-2} e^{- \frac{w_n}{n^{\gamma / 2}}} \to 0. 
%\end{align*}
% which goes to 0 as $n \to \infty$.
\end{proof}

\noindent{\bf (b2) Maximal degree for vertices born before $\log{n}+B$:}

To simplify notation let $\Delta_n:= \gamma \log n + B, \Upsilon_n:= \gamma \log n + B + t_n^+$.
% \begin{equation}
% \label{eqn:del-ups-def}
%     \Delta_n:= \gamma \log n + B, \qquad \Upsilon_n:= \gamma \log n + B + t_n^+.
% \end{equation}
For fixed vertex $v$ born into $\BP_n^+(\cdot)$ and for time $t\leq \Upsilon_n$, let  $\deg(v,t)$ denote the degree of this vertex $v$ in $\BP_n^+(t)$  with the convention that $\deg(v,t):=0$ for $t< \sigma_v$. Write $\deg_n(v):=\deg(v,\Upsilon_n)$ for the final degree of $v$ in $\tilde{\cT}^+_n(B,\omega_n)$.  Fix $C>0$ and let $\dB_n$ be the set of vertices born before $\gamma\log{n}+B$ whose final degree is too large i.e. 
$ \dB_n  := \{ v\in \tilde{\cT}^+_n(B,\omega_n): \sigma_v \le \gamma\log{n}+B, \deg_n(v) > C n^{\frac{1-\gamma}{2+\alpha}} (\log{n})^2 \}$, 
where as before, $\deg_n(v):=\deg(v,\Upsilon_n)$ is the degree of vertex $v$ in the final tree $\tilde{\cT}^+_n(B,\omega_n)$.

\begin{prop}
\label{prop:early_vert_max_deg_not_large}
We can choose $C<\infty$ such that  
$\pr( |\dB_n| \geq  1) \to 0$
as $n \to \infty$.
\end{prop}

% The plan is as follows: we control the maximal degree of vertices born in the early (pre $\Delta_n$) tree then show that none of these early vertices have time to accumulate too many edges in the remaining $\Upsilon_n-\Delta_n$ time period.

\begin{proof}
Consider the tree $\BP_n^+(\Delta_n)$. Let $M_n(\Delta_n) := \max_{v \in \BP_n^+(\Delta_n) } \deg(v, \Delta_n)$ be the maximal degree of vertices in $\BP_n^+(\Delta_n)$ at time $\Delta_n$. Let $\ell_n := 10 e \log n$ and fix a sequence $\omega_n\uparrow\infty$. 
% Note
% \begin{align*}
% \pr(A) &  = \pr(A \cap B \cap C) + \pr(A \cap (B \cap C)^c) \\
% & \le \pr(A \cap B \cap C) + \pr((B \cap C)^c) \\
% & \le \pr(A \cap B \cap C) + \pr(B^c) + \pr(C^c)
% \end{align*}
%
% So we can bound $\pr( \dB_n > 0)  = \pr( \dB_n \ge 1) $ by
By the union bound, 
\begin{align*}
\pr( |\dB_n| \ge 1) &\le \pr\big( |\dB_n| \ge 1, |\BP_n^{+}(\Delta_n)| < \omega_n n^{\gamma}, M_n(\Delta_n) \le \ell_n\big)\\
&\qquad + \pr(|\BP_n^+(\Delta_n)| \ge \omega_n n^{\gamma}) + \pr(M_n(\Delta_n) > \ell_n). 
\end{align*}
 Lemmas \ref{lem:early_vert_first_term} and \ref{lem:early_vert_second_terms} bound the three terms on the right and complete the proof of the Proposition. % say all three terms die in the limit.
\end{proof}

\begin{lemma}
\label{lem:early_vert_first_term}
Let $\omega_n = \log{n}$. We can choose constant $C <\infty$ such that as $n\to\infty$,
 \[\pr( |\dB_n| \ge 1, |\BP_n^+(\Delta_n)| < \omega_n n^{\gamma}, M_n(\Delta_n) \le \ell_n) \to 0 .\]
\end{lemma}

\begin{proof} 
Let $\dG_n =\{  |\BP_n^+(\Delta_n)| < \omega_n n^{\gamma}, M_n(\Delta_n) \le \ell_n \}$. % Note our goal is to show $\pr( \dB_n \ge 1, \dG_n )  \to 0$.
 It is sufficient to show we can choose constant $C$ such that  $\pr( |\dB_n| \ge 1 | \dG_n ) \to 0$.
  % since $\pr( \dB_n \ge 1 , \dG_n )  = \pr( \dB_n \ge 1 | \dG_n ) \pr(\dG_n )$ by Bayes rule.
%
% The event $\dB_n \ge 1 | \dG_n$ says there is at least one vertex born early with very large final degree conditioned on the event that by time $\gamma \log n + B$ there are not too many vertices and the maximal degree of these vertices is not too large.
Conditional on $\dG_n$, we will construct a  stochastic process that bounds the growth of the maximal degree of the vertices in $\BP_n^+(\Delta_n)$ for times $t\geq \Delta_n$. Let $\set{X_i(\cdot): 1\leq i\leq n^\gamma \omega_n}$ be a collection of i.i.d. stochastic processes with distribution % \begin{equation}
% \label{eqn:x-dist}
$	X(\cdot) = \sum_{j=1}^{\ell_n + A+ 2} Y_j(\cdot)$,
% \end{equation}
where $\set{Y_j(\cdot):j\geq 1}$ is collected of i.i.d. rate one Yule processes. Recall that $t_n^+ = \frac{1 - \gamma}{2 + \alpha} \log n + \frac{\omega_n}{n^{\gamma / 2}}$. Let $\cM_n := \max_{1 \le i \le \omega_n n^{\gamma}} X_i(t_n^+)$. 

 On the event $\dG_n$, the number of vertices $|BP_n^+(\Delta_n)| \leq \omega_n n^\gamma$ and further the maximal degree of any vertex at time $\Delta_n$ is $\leq \ell_n$. Thus on $\dG_n$, for any $v \in \BP_n^+(\Delta_n)$, comparing rates for the point process representing the evolution of degrees for $t>\Delta_n$, we see that $\deg(v,\cdot) \stod X(\cdot) $ with $X$ as above. The time translation makes the precise formulation clunky but in brief, on the set $\dG_n$, for any $v\in \BP_n^+(\Delta_n)$, we can construct $\set{(\deg(v,\Delta_n+s), X(s)): 0\leq s\leq t_n^+}$ on a common probability space so that for all $0\leq s\leq t_n^+$, $\deg(v,\Delta_n+s) \leq X(s) $.
  % in the system at time $\Delta_n$ the evolution of vertices in $\BP_n^+$ after time $t\geq \Delta_n$ is as follows:  \chnr{sample} $\tilde{\cT}_n^+(\Delta_n)$ conditional on  $\dG_n$ i.e. the event that there are fewer than $\omega_n n^{\gamma}$ vertices and the maximal degree is less than $\ell_n$. For each vertex, $v$, in $\tilde{\cT}_n^+(\Delta_n)$ we run an independent, rate 1 Yule process starting with $\text{deg}(v,\Delta_n) + \alpha$ individuals for time $t_n^+$. Our new process starts each Yule process as if each individual has maximal degree at time $\gamma \log n + B$.
   Thus on the event $\dG_n$, the maximal degree at time $\Upsilon_n$ of vertices born before time $\Delta_n$  satisfies 
$\max_{v\in \BP_n^+(\Delta_n)}\deg(v,\Upsilon_n)\stod \cM_n$. The rest of the proof analyzes $\cM_n$. The union bound gives, 
\begin{align}
\pr\left( |\dB_n| \ge 1 | \dG_n \right)   \le \pr\left(\cM_n \ge C n^{\frac{1-\gamma}{2 + \alpha}} (\log n)^2 \right ) 
% & = \pr\left( \cup_{i=1}^{\omega_n n^{\gamma}} X_i(t_n^+) \ge C n^{\frac{1-\gamma}{2 + \alpha}} (\log n)^2\right) \\
 \le \omega_n n^{\gamma} \pr \left( X(t_n^+)  \ge C n^{\frac{1-\gamma}{2 + \alpha}} (\log n)^2\right). \label{eqn:blahh}
\end{align}
% The first inequality comes from the stochastic domination of the new process. The second inequality comes from the union bound. We can bound this last term using the following facts about the Yule process.
% Now for a rate one Yule process started with $m$ individuals at time zero say  $Y^m(\cdot)$ for fixed $t$,  $Y^m(t)$ is distributed as the sum of $m$ iid geometric random variables with $p = e^{-t}$.
By Lemma \ref{lem:yule-prop} for any $t\geq 0$ and $\lambda >0$, with $m=\ell_n + A+2$,
$$
\pr\left(X(t) > \lambda \right)  \le  m\pr\left(\text{geom}(e^{-t}) > (\lambda / m) \right) 
 \le m \exp \left[ - (\lambda / m) e^{-t} \right].
$$
Plugging in $ t = t_n^+ , \lambda=C n^{\frac{1-\gamma}{2 + \alpha}} (\log n)^2$ we get that the last term in \eqref{eqn:blahh} can be bounded by
$
 K \omega_n n^{\gamma} n^{-C} \log n 
$
which goes to zero for sufficiently large $C$. 
\end{proof}
%Let $Y^m(t)$ be a rate one Yule process started with $m$ individuals at time 0 then $Y^m(t) \sim \text{geom}(\frac{1}{m} e^{-t})$ (TODO:
%this is not correct, fix this). Then for any $t, \lambda > 0$,
%\begin{align*}
%\pr\left(\frac{1}{m} Y^m(t) > \lambda \right)  & = \pr\left(Y^m(t) > m\lambda \right) \\
%& = \left (1 - \frac{1}{m}e^{-t} \right)^{m \lambda} \\
%& \le e^{- \frac{1}{m}e^{-t} \cdot m \lambda} = e^{-\lambda e^{-t}}
%\end{align*}
%where the second equality comes from the CDF of the geometric distribution and the inequality comes from $1-x \le e^{-x}$. 
%
%Applying this fact we get
%
%$$\pr \left( \frac{1}{10e \log n + \lceil \alpha \rceil} X_i(t_n^+)  \ge \lambda \right)  \le \exp \left( - \lambda e^{t_n^+}\right) $$
%where $\lambda = C n^{\frac{1-\gamma}{2 + \alpha}} \log n$. Plugging in for $\lambda$ and $t_n^+$ leaves us with 
%$$ \pr \left( X_i(t_n^+)  \ge C n^{\frac{1-\gamma}{2 + \alpha}} (\log n)^2\right) \le n^{-C' e^{\frac{\omega_n}{n^{\gamma /2}}}}$$
%so
%$$\pr\left( \dB_n \ge 1 | \dG_n \right) \le \omega_n n^{\gamma - C' e^{\frac{\omega_n}{n^{\gamma /2}}} }$$
%which goes to 0 as $n\to \infty$ for sufficiently large $C'$.
%\end{proof}

\begin{lemma}
\label{lem:early_vert_second_terms}
For $C$ large enough as $n \to \infty$, $ \pr(|\BP_n^+(\Delta_n)| \ge \omega_n n^{\gamma}) \to 0$, and $\pr(M_n(\Delta_n) > \ell_n)  \to 0.$
\end{lemma}

\begin{proof} 
	The second assertion follows from standard bounds for the maximal degree of the random recursive tree \cite{devroye:1995}. We omit the proof. 
We prove the first  assertion. The size of the tree grows according to a rate one Yule process. Thus by Lemma \ref{lem:yule-prop}, $|\BP_n(\Delta_n)| \sim  \text{\chnr{geom}}\left(e^{-(\gamma log n + B)} \right)$. Thus % By looking at the CDF of the geometric distribution and using the inequality $1 - x \le e^{-x}$ we get $\pr(Y(t) \ge \lambda) \le e^{-\lambda e^{-t}}$ for any $t, \lambda > 0$. Applying this bound we get
\begin{align*}
\pr\left(|\BP_n^+(\Delta_n)| \ge \omega_n n^{\gamma} \right) & \le \exp\left[- \omega_n n^{\gamma} e^{- \gamma \log n - B}  \right] \to 0, \qquad  \text{ as } n\to\infty. 
%& =  \exp\left[- \omega_n  e^{ - B}  \right] 
\end{align*}
% which goes to $0$ as $n \to \infty$.
% For the second assertion, note that for any $0\leq t\leq \Delta_n$, the rate at which a new vertex is born is $|\tilde{\cT}_n^+(t)|$. Since the offspring process of each new vertex (before time $\Delta_n$) is a Poisson process,  the probability that this new vertex has degree greater than $\ell_n$ conditional on $\tilde{\cT}_n^+(t)$ is
% $\pr(\text{Poisson}(\Delta_n -t)\geq \ell_n)\leq \pr(\text{Poisson}(\Delta_n)\geq \ell_n).
% $
% Thus writing $N_n(\Delta_n)$ for the number of vertices with degree at least $\ell_n$ by time $\Delta_n$ and recalling that for $t\leq \Delta_n$, $\E(\tilde{\cT}_n^+(t)) = e^t$ we have,
% \[\E(N_n(\Delta_n)) = \int_0^{\Delta_n}\pr(\text{Poisson}(\Delta_n -t)\geq \ell_n) e^t dt\leq  e^B \pr(\text{Poisson}(\Delta_n)\geq \ell_n) / n^\gamma .\] % MEOW: could put this equation inline but looks bad
% Since $\Delta_n = \gamma \log{n}+B$ with $\gamma <1$, Poisson exponential tail bounds complete the proof.
\end{proof}

\color{black}

\section{Proofs: Convergence rates for model without change point}\label{prwocp}
This section is dedicated to proving Theorem \ref{ratewocp} and Theorem \ref{l1conv}. \chnr{ We need the following lemma which quantifies the rate of convergence of solutions of renewal equations to their limit as time goes to infinity.}
 
\begin{lemma}\label{renr}
Consider a continuous time branching process with attachment function $f$ that satisfies Assumption \ref{ass:attach-func}. Fix $\beta \in (0, \lambda^*)$. There exist positive constants $C_1$, $C_2$ such that \chsb{the following holds: }if $h$ solves the renewal equation
$$
h(t) = e^{-\lambda^* t}\phi(t) + \int_0^th(t-s) e^{-\lambda^* s}\mu_f(ds)
$$
with any $\phi$ satisfying $|\phi(s)| \le C_{\phi}e^{\beta s}$ for all $s \ge 0,$ for some \chnr{$C_{\phi}>0$}, \chsb{then} $h(\infty) := \lim_{t \rightarrow \infty} h(t)$ \chsb{exists and} we have, for all $t \ge 0$,
$
|h(\infty) - h(t)| \le C_1 C_{\phi} e^{-C_2 t}.
$
\end{lemma}
\begin{proof}
\chnr{In the proof, $C,C'$ will denote generic positive constants, not depending on $C_{\phi}$ or the choice of $\phi$, whose values might change from line to line.}
We will use estimates about quantitative rates of convergence for renewal measures derived in \cite{bardet2015quantitative} in the setting of the point process with i.i.d. inter-arrival times having distribution $e^{-\lambda^*s}\mu_f(ds)$. By Assumption \ref{ass:attach-func} (ii), it is clear that the measure $e^{-\lambda^*s}\mu_f(ds)$ satisfies $\int_0^{\infty} e^{\beta' s}  e^{-\lambda^*s}\mu_f(ds) < \infty$ for some $\beta' >0$ and thus, Assumption 1 of \cite{bardet2015quantitative} is satisfied. Moreover, for any Borel set $A$ in $[0,1]$, denoting by $E$ the first time the root reproduces (which has an exponential distribution with rate $f(0)$), note that
$$
\mu_f(A) \ge \E\left(\ind\set{E \in A}\right) = \int_A f(0) e^{-f(0)x} dx \ge f(0) e^{-f(0)}\int_A dx
$$
and consequently, the distribution of the inter-arrival time is \textit{spread out} in the sense of Assumption 2 of \cite{bardet2015quantitative} taking $c= 1/2, L=1/2$ and $\widetilde{\eta} = f(0) e^{-(\lambda^* + f(0))}$. Thus, Corollary 1 of \cite{bardet2015quantitative} holds for the point process under consideration. For any $x \ge 0$, denote by $U^x$ the renewal measure corresponding to the associated point process with time started at $x$. The stationary version of this point process corresponds to a random starting time whose law is $\mu^*(ds) = m^{\star-1}s e^{-\lambda^*s}\mu_f(ds)$ (called the \textit{stationary delay distribution}), where $m^{\star} = \int_0^{\infty} u e^{-\lambda^*u}\mu_f(du)$. From translation invariance, it follows that the renewal measure associated to this stationary version is given by $U^*(ds) = m^{\star-1} ds$. By Corollary 1 of \cite{bardet2015quantitative}, there exist constants $C, C'>0$ and $\beta'' < \beta'$ such that for any Borel set $D \subset (0,\infty)$ and any $x, t \ge 0$,
\begin{equation*}
|U^x(D + t) - U^0(D+t)| \le C e^{\beta''x}e^{-C't} (U^0((0, \sup D)) + 1).
\end{equation*}
Integration both sides of the above relation over $x$ with respect to the stationary delay distribution $\mu^*(dx)$ and using Fubini's theorem and the fact that $\int_0^{\infty} e^{\beta' s}  e^{-\lambda^*s}\mu_f(ds) < \infty$, we obtain
\begin{equation*}
|U^*(D + t) - U^0(D+t)| \le C e^{-C't} (U^0((0, \sup D)) + 1).
\end{equation*}
This, in turn, implies that for ant $t \ge 0$, if $U^*_{M,t}$ and $U^0_{M,t}$ denote the measures defined by $U^*_{M,t}(D) = U^*(D + t)$ and $U^0_{M,t}(D) = U^0(D + t)$ for any Borel set $D \subset [0,M]$, then using the fact that $\lim_{t \rightarrow \infty} t^{-1} U^0([0, t]) = \frac{1}{m^{\star}}$ (which follows from the elementary renewal theorem),
\begin{equation}\label{renbd}
||U^*_{M,t} - U^0_{M,t}||_{TV} \le CMe^{-C't}.
\end{equation}
From standard results in renewal theory, \chsb{$h(t) =\int_0^t e^{-\lambda^* (t-s)}\phi(t-s)U^0(ds), t \ge 0,$ and $h(\infty) := \lim_{t \rightarrow \infty} h(t)$ exists with $h(\infty) = \int_0^{\infty}e^{-\lambda^* s}\phi(s)U^*(ds)$.} Thus, for $t \ge 0$,
\begin{multline}\label{ren0}
|h(\infty) - h(t)| = \left|\int_0^{\infty}e^{-\lambda^* s}\phi(s)U^*(ds) - \int_0^t e^{-\lambda^* (t-s)}\phi(t-s)U^0(ds)\right|\\
 \le \left|\int_0^{t}e^{-\lambda^* s}\phi(s)U^*(ds) - \int_0^t e^{-\lambda^* (t-s)}\phi(t-s)U^0(ds)\right| + \int_t^{\infty}e^{-\lambda^* s}\phi(s)U^*(ds).
\end{multline}
As $|\phi(s)| \le C_{\phi}e^{\beta s}$ for all $s$,
\begin{equation}\label{ren1}
\int_t^{\infty}e^{-\lambda^* s}\phi(s)U^*(ds) \le C_{\phi}m^{\star-1}\int_t^{\infty}e^{-(\lambda^*-\beta) s}ds = \frac{C_{\phi}}{m^{\star}(\lambda^* - \beta)}e^{-(\lambda^*-\beta) t}.
\end{equation}
To estimate the first term in the bound \eqref{ren0}, note that for $t \ge 0$,
\begin{multline}\label{ren2}
\left|\int_0^{t}e^{-\lambda^* s}\phi(s)U^*(ds) - \int_0^t e^{-\lambda^* (t-s)}\phi(t-s)U^0(ds)\right|\\
%= \left|\int_0^{t}e^{-\lambda^*(t-s)}\phi(t-s)U^*(ds) - \int_0^t e^{-\lambda^* (t-s)}\phi(t-s)U^0(ds)\right| \\
 \le \int_0^{t/2}e^{-\lambda^*(t-s)}\phi(t-s)U^*(ds) +  \int_0^{t/2} e^{-\lambda^* (t-s)}\phi(t-s)U^0(ds)\\
 + \left|\int_{t/2}^{t}e^{-\lambda^* (t-s)}\phi(t-s)U^*(ds) - \int_{t/2}^t e^{-\lambda^* (t-s)}\phi(t-s)U^0(ds)\right|\\
 \le C_{\phi}e^{-(\lambda^* - \beta)t/2}U^*([0, t/2]) + C_{\phi}e^{-(\lambda^* - \beta)t/2}U^0([0, t/2]) + C_{\phi}||U^*_{t/2,t/2} - U^0_{t/2,t/2}||_{TV} \le C'_1C_{\phi} e^{-C'_2t}
\end{multline}
for constants $C_1', C_2'>0$ not depending on $\phi$, where we used \eqref{renbd} and the observations that $U^*([0, t/2] = t /2m^{\star}$ and $\lim_{t \rightarrow \infty} t^{-1} U^0([0, t/2]) = 1 / (2m^{\star})$. The lemma follows using \eqref{ren1} and \eqref{ren2} in \eqref{ren0}.
\end{proof}
%The next key ingredient will be Theorem \ref{l1conv} which we now prove.
\begin{proof}[Proof of Theorem \ref{l1conv}]
\chnr{We bound $\left|e^{-\lambda^* t} Z^{\phi}_f(t) - W_{\infty}M^{\phi}_f(\infty)\right|$ using the same techniques as in the proof of Theorem 3.1 of \cite{nerman1981convergence}. For each term appearing in the bound, we show that they are small in a suitable sense using renewal theoretic methods and variance computations.}

In the proof, $C, C', C'', C_1,C_2, \beta', \beta$ denote generic positive constants depending neither on $b_{\phi}$ nor the choice of $\phi$. Following \cite{nerman1981convergence}, write $x=(x',i)$ when $x$ is the $i$-th child of $x'$ and define for any $t, c \ge 0$,
\begin{align*}
\mathcal{I}(t) = \{x= (x', i): \sigma_{x'} \le t \text{ and } t < \sigma_x < \infty\}, \ \
\mathcal{I}(t,c) = \{x= (x', i): \sigma_{x'} \le t \text{ and } t  + c < \sigma_x < \infty\}.
\end{align*}
Let $\bar T_t$ denote the number of vertices born by time $t$ and let $\mathcal{A}_n$ be the filtration generated by the entire \chsb{biographies} of the first $n$ vertices (see \cite{nerman1981convergence} for detailed definitions). Define $\cF_t = \mathcal{A}_{\bar T_t}$. For any $s >0$, write $\phi = \phi_s + \phi_s'$ where $\phi_s(u) = \phi(u) \ind\set{u < s}$ and $\phi_s'(u) = \phi(u) \ind\set{u \ge s}$. Note that
\begin{multline}\label{trundec}
\E\left|e^{-\lambda^* t} Z^{\phi}_f(t) - W_{\infty}M^{\phi}_f(\infty)\right| \le \E\left|e^{-\lambda^* t} \left(Z^{\phi}_f(t) - Z^{\phi_s}_f(t)\right)\right| + \E\left|e^{-\lambda^* t} Z^{\phi_s}_f(t) - W_{\infty}M^{\phi_s}_f(\infty)\right|\\
+ \E\left(\left|M^{\phi_s}_f(\infty) - M^{\phi}_f(\infty)\right|W_{\infty}\right).
\end{multline}
\chr{Recall that, by \eqref{eqn:prop-under-lamb} appearing in Assumption \ref{ass:attach-func} (ii), $\underline{\lambda} < \lambda^*$ and hence, there exists $\beta' \in (\underline{\lambda}, \lambda^*)$ such that
%\vspace{-.2in}
\begin{equation}\label{rholess}
e^{-\beta' t}\E\left(\xi_f(t)\right) \chsb{=} \E\left(\xi_f(t)\right)\int_t^{\infty}\beta'e^{-\beta'u}du \le  \int_0^{\infty}\beta'e^{-\beta'u}\E\left(\xi_f(u)\right)du = \hat{\rho}(\beta') < \infty.
\end{equation}
Using this, the third term in the bound \eqref{trundec} can be bounded as}
\begin{multline}\label{trun1}
\E\left(\left|M^{\phi_s}_f(\infty) - M^{\phi}_f(\infty)\right|W_{\infty}\right) = M^{\phi_s'}_f(\infty) = \frac{1}{m^{\star}}\int_s^{\infty}e^{-\lambda*u}\E\left(\phi(u)\right)du \\
\le \frac{b_{\phi}}{m^{\star}}\int_s^{\infty}e^{-\lambda*u}\E\left(\xi_f(u) + 1\right)du \le Cb_{\phi}e^{-(\lambda^* - \beta')s}.
\end{multline}
%for some $\beta' < \lambda^*$ by virtue of Assumption \ref{ass:attach-func} (ii).
The first term in the bound \eqref{trundec} can be bounded as
\begin{equation}\label{trun2}
\E\left|e^{-\lambda^* t} \left(Z^{\phi}_f(t) - Z^{\phi_s}_f(t)\right)\right| = \E\left(e^{-\lambda^* t} Z^{\phi_s'}_f(t)\right) \le \left| M^{\phi_s'}_f(t) - M^{\phi_s'}_f(\infty)\right| + M^{\phi_s'}_f(\infty).
\end{equation}
By the fact that $M^{\phi_s'}_f(t)$ satisfies the renewal equation \eqref{renmean} (with $\phi_s'$ in place of $\phi$) and Lemma \ref{renr}, for $t \ge 0$,
$
\left| M^{\phi_s'}_f(t) - M^{\phi_s'}_f(\infty)\right| \le C_1b_{\phi} e^{-C_2 t}.
$
Using this estimate and \eqref{trun1} in \eqref{trun2}, we obtain
\begin{equation}\label{trun3}
\E\left|e^{-\lambda^* t} \left(Z^{\phi}_f(t) - Z^{\phi_s}_f(t)\right)\right| \le C_1b_{\phi} e^{-C_2 t} + Cb_{\phi}e^{-(\lambda^* - \beta')s}.
\end{equation}
Using \eqref{trun1} and \eqref{trun3} in \eqref{trundec}, for any $t, s \ge 0$,
\begin{equation}\label{trunfin}
\E\left|e^{-\lambda^* t} Z^{\phi}_f(t) - W_{\infty}M^{\phi}_f(\infty)\right| \le \E\left|e^{-\lambda^* t} Z^{\phi_s}_f(t) - W_{\infty}M^{\phi_s}_f(\infty)\right| + C_1b_{\phi} e^{-C_2 t} + 2Cb_{\phi}e^{-(\lambda^* - \beta')s}.
\end{equation}
Now, we estimate the first term in the above bound.
Observe that as $\phi_s(u) = 0$ for all $u \ge s$, every individual that contributes to $Z^{\phi_s}_f(t+s)$ must be born after time $t$. Therefore,
$
Z^{\phi_s}_f(t+s) = \sum_{x \in \mathcal{I}(t)}Z^{\phi_s}_{f,x}(t+s - \sigma_x)
$
where for any vertex $x$ and any $u \ge 0$, $Z^{\phi_s}_{f,x}(u)$ denotes the aggregate $\phi$-score at time $\sigma_x + u$ treating the vertex $x$ as the root. For $t, c \ge 0$ such that $s \ge c$, write
$$
X(t,s,c) = \sum_{x \in \mathcal{I}(t) \setminus \mathcal{I}(t,c)} e^{- \lambda^* \sigma_x} \left(e^{-\lambda^* (t+s - \sigma_x)}Z^{\phi_s}_{f,x}(t+s - \sigma_x) - M^{\phi_s}_f(t+s - \sigma_x)\right).
$$
and write $W_t = \sum_{x \in \mathcal{I}(t)}e^{-\lambda^* \sigma_x}$, $W_{t,c} = \sum_{x \in \mathcal{I}(t,c)}e^{-\lambda^* \sigma_x}$.
Following equation (3.36) in \cite{nerman1981convergence}, we obtain
\begin{align}\label{bddexp}
&\left|e^{-\lambda^* (t+s)} Z^{\phi_s}_f(t+s) - W_{\infty}M^{\phi_s}_f(\infty)\right|\nonumber\\
&\le \left| X(t,s,c)\right| + \sum_{x \in \mathcal{I}(t) \setminus \mathcal{I}(t,c)} e^{- \lambda^* \sigma_x} \left|M^{\phi_s}_f(t+s - \sigma_x) - M^{\phi_s}_f(\infty)\right|\nonumber\\
&\qquad + \ \left|\sum_{x \in \mathcal{I}(t,c)} e^{- \lambda^* \sigma_x} \left(e^{-\lambda^* (t+s - \sigma_x)}Z^{\phi_s}_{f,x}(t+s - \sigma_x) - M^{\phi_s}_f(\infty)\right)\right| + M^{\phi_s}_f(\infty)\left|W_t - W_{\infty}\right|.\nonumber\\
\end{align}
Note that
\begin{equation}\label{vcom}
\operatorname{Var}(X(t,s,c) | \cF_t) = \sum_{x \in \mathcal{I}(t) \setminus \mathcal{I}(t,c)} e^{- 2\lambda^* \sigma_x}V^{\phi_s}_{f}(t+s - \sigma_x)
\end{equation}
where $V^{\phi_s}_{f}(t) = \operatorname{Var}\left(e^{-\lambda^*t}Z^{\phi_s}_f(t)\right)$. Recall $m^{\phi_s}_f(t) = \E\left(Z^{\phi_s}_f(t)\right)$ and $v^{\phi_s}_f(t) = \operatorname{Var}\left(Z^{\phi_s}_f(t)\right)$. From Theorem 3.2 of \cite{jagers1984growth}, $v^{\phi_s}_f(t) = h \star U(t)$, where
$
h(t) = \operatorname{Var}\left(\phi_s(t) + \int_0^tm^{\phi_s}_f(t-u) \xi_f(du)\right)
$
and $U(\cdot) = \sum_{\ell=0}^{\infty}\mu_f^{\star \ell}(\cdot)$ denotes the renewal measure. As $\phi_s(t) \le b_{\phi}(\xi_f(t) + 1)$ for all $t$, using Assumption \ref{varass},
\begin{align}\label{mm}
e^{-2\lambda^* t} \E(\phi_s(t))^2 &\le (b_{\phi})^2\E\left(e^{-\lambda^* t}(1+\xi_f(t))\right)^2 \nonumber\\
&\le 2(b_{\phi})^2\E\left(e^{-2\lambda^* t} + \lambda^{*2}\left(\int_t^{\infty}e^{-\lambda^*u} \xi_f(u)du\right)^2\right) \le C(b_{\phi})^2.
\end{align}
%Therefore, by Theorem 3.5 of \cite{jagers1984growth}, 
%$$
%\lim_{t \rightarrow \infty}V^{\phi}_{f}(t) = \frac{\left(\E\left(\hat{\phi}(\lambda^*)\right)\right)^2 \operatorname{Var}\left(\hat{\xi}_f(\lambda^*)\right)}{\lambda^{*2}m^{\star 2}(1 - \hat{\mu}_f(2\lambda^*))} \le C(b_{\phi})^2.
%$$
As $\E\left(\xi_f(t) + 1\right) \le Ce^{\beta't}$ \chr{by \eqref{rholess}}, therefore $\E\left(\phi_s(t)\right) \le b_{\phi}\E\left(\xi_f(t) + 1\right) \le b_{\phi}Ce^{\beta't}$. Hence, by the fact that $M^{\phi_s}_f(t)$ satisfies the renewal equation \eqref{renmean} and Lemma \ref{renr}, for $t \ge 0$,
\begin{equation}\label{me}
\left|M^{\phi_s}_f(t) - M^{\phi_s}_f(\infty)\right| \le C_1b_{\phi} e^{-C_2 t}.
\end{equation}
Moreover,
\begin{equation}\label{me0}
M^{\phi_s}_{f}(\infty) = (m^{\star})^{-1} \int_0^{\infty} e^{-\lambda^*u}\E(\phi_s(u))du \le  (m^{\star})^{-1} b_{\phi}\int_0^{\infty}\E\left(e^{-\lambda^* u}(1+\xi_f(u))\right)du \le Cb_{\phi}.
\end{equation}
Using \eqref{me} and \eqref{me0}, we obtain for all $t \ge 0$,
\begin{equation}\label{me1}
M^{\phi_s}_f(t) \le C'b_{\phi}.
\end{equation}
From \eqref{mm} and \eqref{me1}, we conclude for all $t \ge 0$,
\begin{align*}
e^{-2\lambda^* t}h(t) &= \operatorname{Var}\left(e^{-\lambda^* t}\phi_s(t) + \int_0^te^{-\lambda^* (t-u)}m^{\phi_s}_f(t-u) e^{-\lambda^*u}\xi_f(du)\right)\\
&\le 2e^{-2\lambda^* t} \E(\phi_s(t))^2 + 2\E\left(\int_0^tM^{\phi_s}_f(t-u) e^{-\lambda^*u}\xi_f(du)\right)^2\\
&\le 2C(b_{\phi})^2 + 2(Cb_{\phi})^2\E\left(\int_0^{\infty}e^{-\lambda^*u}\xi_f(du)\right)^2 \le C'(b_{\phi})^2.
\end{align*}
Thus, for all $t \ge 0$,
\begin{multline}\label{vbd}
V^{\phi_s}_{f}(t) = \int_0^{\infty}e^{-2\lambda^*(t-u)}h(t-u)e^{-2\lambda^*u}U(du)\\
\le C'(b_{\phi})^2\int_0^{\infty}e^{-2\lambda^*u}U(du) = C'(b_{\phi})^2 \sum_{\ell = 0}^{\infty}\hat{\mu}_f(2\lambda^*)^{\ell}
 = \frac{C'(b_{\phi})^2}{1 - \hat{\mu}_f(2\lambda^*)} = C''(b_{\phi})^2.
\end{multline}
Using this bound in \eqref{vcom}, we obtain
$$
\E\left(\operatorname{Var}(X(t,s,c) | \cF_t) \right) \le C''(b_{\phi})^2 \E\left(\sum_{x \in \mathcal{I}(t) \setminus \mathcal{I}(t,c)} e^{- 2\lambda^* \sigma_x} \right) \le C''(b_{\phi})^2e^{-\lambda^* t}\E(W_t) = C''(b_{\phi})^2e^{-\lambda^* t}.
$$
Moreover, $\E\left(X(t,s,c) | \cF_t\right) = 0$. Thus, we obtain
\begin{equation}\label{bdd1}
\E |X(t,s,c)| \le \sqrt{\E(X(t,s,c))^2} = \sqrt{\operatorname{Var}(X(t,s,c))} \le \sqrt{C''}b_{\phi}e^{-\lambda^* t/2}.
\end{equation}

Using \eqref{me},
\begin{multline}\label{bdd2}
\E\left(\sum_{x \in \mathcal{I}(t) \setminus \mathcal{I}(t,c)} e^{- \lambda^* \sigma_x} \left|M^{\phi_s}_f(t+s - \sigma_x) - M^{\phi_s}_f(\infty)\right|\right) \le C_1b_{\phi}e^{-C_2(s-c)}\E(W_t) = C_1b_{\phi}e^{-C_2(s-c)}.
\end{multline}
To estimate the third term in the bound \eqref{bddexp}, observe that upon conditioning on $\cF_t$ and noting that $\sup_{t < \infty}M^{\phi_s}_f(t) \le C'b_{\phi}$,
\begin{multline}\label{lul0}
 \E\left(\left|\sum_{x \in \mathcal{I}(t,c)} e^{- \lambda^* \sigma_x} \left(e^{-\lambda^* (t+s - \sigma_x)}Z^{\phi_s}_{f,x}(t+s - \sigma_x) - M^{\phi_s}_f(\infty)\right)\right|\right) \\
 \le  \E\left(\sum_{x \in \mathcal{I}(t,c)} e^{- \lambda^* \sigma_x}\left(M^{\phi_s}_f(t+s - \sigma_x) + M^{\phi_s}_f(\infty)\right)\right) \le C'b_{\phi}\E(W_{t,c}).
\end{multline}
Consider the characteristic $\phi^{c}(v) = e^{\lambda^*v}\left(\int_{v+c}^{\infty}e^{-\lambda^* u}\xi_f(du)\right)$, $v \ge 0$. Then $W_{t,c} = e^{-\lambda^* t}Z^{\phi^{c}}_f(t)$. Note that
\begin{multline*}
\E(\phi^{c}(t)) = e^{\lambda^*t}\E\left(\int_{t+c}^{\infty}e^{-\lambda^* u}\xi_f(du)\right) = e^{\lambda^*t}\E\left(\int_{t+c}^{\infty}\lambda^* e^{-\lambda^* v}(\xi_f(v) - \xi_f(t+c))dv\right)\\
\le e^{\lambda^*t}\E\left(\int_{t+c}^{\infty}\lambda^* e^{-\lambda^* v}\xi_f(v)dv\right) \le Ce^{\lambda^* t} \left(\int_{t+c}^{\infty}\lambda^* e^{-\lambda^* v}e^{\beta'v}dv\right) \le \frac{C\lambda^*e^{\lambda^*t}}{\lambda^* - \beta'}e^{-(\lambda^* - \beta')t} = \frac{C\lambda^*e^{\beta't}}{\lambda^* - \beta'}.
\end{multline*}
Hence, by Lemma \ref{renr},
\begin{equation}\label{lul1}
\left|M^{\phi^c}_f(t) - M^{\phi^c}_f(\infty)\right| \le C_1 e^{-C_2 t}.
\end{equation}
Moreover, by Lemma 3.5 of \cite{nerman1981convergence},
$
M^{\phi^c}_f(\infty) = \int_c^{\infty}(1 - \mu_{f,\lambda^*}(u))du \big / \int_0^{\infty}(1 - \mu_{f, \lambda^*}(u))du
$
where $\mu_{f,\lambda^*}(u) = \int_0^{u}e^{-\lambda^*v}\mu_f(dv)$. Now, for any $u \ge 0$,
\begin{align*}
1 - \mu_{f,\lambda^*}(u) = \int_u^{\infty}e^{-\lambda^*v}\mu_f(dv) &\le \int_u^{\infty}\lambda^*e^{-\lambda^*v}\mu_f(v)dv\\
&\le C\int_u^{\infty}\lambda^*e^{-\lambda^*v}e^{\beta'v}dv = \frac{C\lambda^*}{\lambda^* - \beta'}e^{-(\lambda^* - \beta')u}
\end{align*}
and hence,
$$
\int_c^{\infty}(1 - \mu_{f,\lambda^*}(u))du \le \int_c^{\infty}\frac{C\lambda^*}{\lambda^* - \beta'}e^{-(\lambda^* - \beta')u}du = \frac{C\lambda^*}{(\lambda^* - \beta')^2}e^{-(\lambda^* - \beta')c}.
$$
This bound implies that there exists $C>0$ such that for all $c>0$,
\begin{equation}\label{lul2}
M^{\phi^c}_f(\infty) \le Ce^{-(\lambda^* - \beta')c}.
\end{equation}
Combining \eqref{lul1} and \eqref{lul2}, we have 
$
\E(W_{t,c}) = M^{\phi^c}_f(t) \le C_1e^{-C_2 t} + Ce^{-(\lambda^* - \beta')c}.
$
Using this in \eqref{lul0},
\begin{equation}\label{bdd3}
 \E\left(\left|\sum_{x \in \mathcal{I}(t,c)} e^{- \lambda^* \sigma_x} \left(e^{-\lambda^* (t+s - \sigma_x)}Z^{\phi_s}_{f,x}(t+s - \sigma_x) - M^{\phi_s}_f(\infty)\right)\right|\right) \le C'b_{\phi}\left(e^{-C_2 t} + e^{-(\lambda^* - \beta')c}\right).
\end{equation}
To estimate the last term in the bound \eqref{bddexp}, observe that for any $t \ge 0$, $W_{\infty} = \sum_{x \in \mathcal{I}(t)}e^{-\lambda^* \sigma_x} W^x_{\infty}$, where $W^x_{\infty}$ corresponds to $W_{\infty}$ treating vertex $x$ as the root (and hence are i.i.d. and have the same distribution as $W_{\infty}$). Moreover, by Theorem 4.1 of \cite{jagers1984growth}, $\operatorname{Var}\left(W_{\infty}\right) < \infty$. Using these observations,
\begin{align*}
\E\left(W_t - W_{\infty}\right)^2 = \E\left(\sum_{x \in \mathcal{I}(t)}e^{-\lambda^* \sigma_x} (1-W^x_{\infty})\right)^2 &= \operatorname{Var}\left(W_{\infty}\right)\E\left(\sum_{x \in \mathcal{I}(t)}e^{-2\lambda^* \sigma_x}\right)\\
&\le \operatorname{Var}\left(W_{\infty}\right) e^{-\lambda^* t}\E(W_t) = \operatorname{Var}\left(W_{\infty}\right) e^{-\lambda^* t}.
\end{align*}
Together with the fact that $\sup_{t < \infty}M^{\phi_s}_f(t) \le C'b_{\phi}$, this implies that for $t \ge 0$,
\begin{equation}\label{bdd4}
\E\left|M^{\phi_s}_f(\infty)\left|W_t - W_{\infty}\right|\right| \le \sqrt{\E\left(M^{\phi_s}_f(\infty)\left|W_t - W_{\infty}\right|\right)^2} \le C'b_{\phi}e^{-\lambda^*t/2}.
\end{equation}
Using \eqref{bdd1}, \eqref{bdd2}, \eqref{bdd3} and \eqref{bdd4} and the bound \eqref{bddexp}, we obtain $D, D_1, D_2, D_3>0$ not \chnr{depending} on $b_{\phi}, t,s,c$ such that
\begin{equation}\label{trunfin2}
\E\left(\left|e^{-\lambda^* (t+s)} Z^{\phi_s}_f(t+s) - W_{\infty}M^{\phi_s}_f(\infty)\right|\right) \le Db_{\phi}\left(e^{-D_1 t} + e^{-D_2c} + e^{-D_3(s-c)}\right).
\end{equation}
On taking $t-s$ in place of $t$ in \eqref{trunfin2}, we obtain for any $s,t,c \ge 0$ such that $t \ge s \ge c$,
\begin{equation}\label{trunfin3}
\E\left(\left|e^{-\lambda^*t} Z^{\phi_s}_f(t) - W_{\infty}M^{\phi_s}_f(\infty)\right|\right) \le Db_{\phi}\left(e^{-D_1 (t-s)} + e^{-D_2c} + e^{-D_3(s-c)}\right).
\end{equation}
Using \eqref{trunfin3} in \eqref{trunfin}, we obtain for any $s,t,c \ge 0$ such that $t \ge s \ge c$,
\begin{equation*}
\E\left|e^{-\lambda^* t} Z^{\phi}_f(t) - W_{\infty}M^{\phi}_f(\infty)\right| \le D b_{\phi}\left(e^{-D_1 (t-s)} + e^{-D_2c} + e^{-D_3(s-c)}\right) + C_1b_{\phi} e^{-C_2 t} + 2Cb_{\phi}e^{-(\lambda^* - \beta')s}.
\end{equation*}
The theorem now follows by taking $s=t/2$ and $c=t/4$.
\end{proof}

\noindent Recall $\lambda_{\ell}, \lambda_{\ell}^{(k)}$ for $k, \ell \ge 0$ from \eqref{lambdadef}, with $f_1$ replaced by $f$ (this section considers the model without change point). \chnr{The following lemma uses the exponential convergence rate established in Theorem \ref{l1conv} along with some continuity estimates to furnish a quantitative sup-norm bound on appropriate statistics on suitably chosen intervals.}
\begin{lemma}\label{quant}
Consider a continuous time branching process with attachment function \chsb{$f$ that satisfies Assumptions \ref{ass:attach-func}, \ref{kickass} and \ref{varass}}. There exist \chsb{$\omega_1 \in (0,1), \epsilon^* \in(0,1)$} and positive constants $C, \omega_2$ such that for all $\epsilon \le \epsilon^*$ and all $T \in \left[\frac{1-\epsilon}{\lambda^*}\log n, \frac{1+\epsilon}{\lambda^*}\log n\right]$,
$$
\E\left(n^{\omega_1}\sup_{t \in [0,2\epsilon \log n / \lambda^*]}\left| e^{-\lambda^* T}\sum_{\ell =0}^{\infty} \lambda_{\ell}(t)D\left(\ell,T\right) - \frac{1}{\lambda^* m^{\star}}\sum_{\ell =0}^{\infty} \lambda_{\ell}(t)p_{\ell}W_{\infty}\right| \right) \le Cn^{-\omega_2}
$$
and for any $k \ge 0$,
$$
\E\left(n^{\omega_1}\sup_{t \in [0,2\epsilon \log n / \lambda^*]}\left| e^{-\lambda^* T}\sum_{\ell =0}^{\infty} \lambda_{\ell}^{(k)}(t)D\left(\ell,T\right) - \frac{1}{\lambda^* m^{\star}}\sum_{\ell =0}^{\infty} \lambda_{\ell}^{(k)}(t)p_{\ell}W_{\infty}\right|\right) \le C(k+1)n^{-\omega_2}.
$$
\end{lemma}

\begin{proof}
For any $t$, consider the characteristic $\phi(s) = \sum_{\ell=0}^{\infty}\lambda_{\ell}(t)\ind\set{\xi_f(s) = \ell}$. Then $Z^{\phi}_f(s) = \sum_{\ell=0}^{\infty}\lambda_{\ell}(t)D(\ell, s)$. \chsb{By Lemma \ref{lem:deg_dist_quad_conv} (ii), $\lim_{t \rightarrow \infty} e^{-\lambda^* t} m_f(t) = \frac{1}{\lambda^* m^{\star}}$. Moreover, as Assumption \ref{kickass} holds, by Lemma \ref{bbt1}, there exists a constant $C>0$ such that for each $\ell \ge 0$, $w_{\ell} \le C(\ell+1)$. Thus, there exists a constant $C'>0$ such that for any $\ell \ge 0$,
$$
\sup_{t \ge 0} e^{-\lambda^* t}\lambda_{\ell}(t) \le 1 + w_{\ell}\left(\sup_{t \ge 0}e^{-\lambda^* t}m_{f}(t)\right) \le C' (\ell + 1).
$$
Hence, the hypotheses of Theorem \ref{l1conv} hold with $b_{\phi} = C'e^{\lambda^* t}$}. Consequently, for any $\epsilon \in (0, 1)$, any $t \in [0,2\epsilon \log n / \lambda^*]$ and any $T \in \left[\frac{1-\epsilon}{\lambda^*}\log n, \frac{1+\epsilon}{\lambda^*}\log n\right]$,
\begin{multline*}
\E\left(\left| e^{-\lambda^* T}\sum_{\ell =0}^{\infty} \lambda_{\ell}(t)D\left(\ell,T\right) - \frac{1}{\lambda^* m^{\star}}\sum_{\ell =0}^{\infty} \lambda_{\ell}(t)p_{\ell}W_{\infty}\right|\right)\\
\le C_1Ce^{\lambda^*t}e^{-\frac{C_2(1-\epsilon)}{\lambda^*}\log n} \le C_1Ce^{2\epsilon \log n}e^{-\frac{C_2(1-\epsilon)}{\lambda^*}\log n}.
\end{multline*}
Therefore, choosing $\epsilon^*$ small enough, there exists $\theta_1>0$ such that for any $\epsilon \le \epsilon^*$, any $t \in [0,2\epsilon \log n / \lambda^*]$ and any $T \in \left[\frac{1-\epsilon}{\lambda^*}\log n, \frac{1+\epsilon}{\lambda^*}\log n\right]$,
\begin{equation}\label{quant1}
\E\left(\left| e^{-\lambda^* T}\sum_{\ell =0}^{\infty} \lambda_{\ell}(t)D\left(\ell,T\right) - \frac{1}{\lambda^* m^{\star}}\sum_{\ell =0}^{\infty} \lambda_{\ell}(t)p_{\ell}W_{\infty}\right|\right) \le n^{- \theta_1}.
\end{equation}
Take any $\theta_2 \in (0, \theta_1)$ and a partition of $[0,2\epsilon \log n / \lambda^*]$ into $t_0<t_1<\dots < t_{\lfloor (2\epsilon \log n / \lambda^*)n^{\theta_2}\rfloor + 1}$ of mesh $n^{-\theta_2}$. By Lemma \ref{lem:lam-t-ts-diff}, for any $j$ and any $t \in [t_j, t_{j+1}]$, there exist constants $C, C'>0$ independent of $\epsilon, n$ such that
\begin{multline}\label{quant2}
\left|\left| e^{-\lambda^* T}\sum_{\ell =0}^{\infty} \lambda_{\ell}(t)D\left(\ell,T\right) - \frac{1}{\lambda^* m^{\star}}\sum_{\ell =0}^{\infty} \lambda_{\ell}(t)p_{\ell}W_{\infty}\right|\right.\\
 \left.- \left| e^{-\lambda^* T}\sum_{\ell =0}^{\infty} \lambda_{\ell}(t_j)D\left(\ell,T\right) - \frac{1}{\lambda^* m^{\star}}\sum_{\ell =0}^{\infty} \lambda_{\ell}(t_j)p_{\ell}W_{\infty}\right|\right|\\
\le e^{-\lambda^* T}\sum_{\ell =0}^{\infty}\left|\lambda_{\ell}(t) - \lambda_{\ell}(t_j)\right|D\left(\ell,T\right) + \frac{1}{\lambda^* m^{\star}}\sum_{\ell =0}^{\infty}\left|\lambda_{\ell}(t) - \lambda_{\ell}(t_j)\right|p_{\ell}W_{\infty}\\
\le \frac{Cn^{C'\epsilon}}{n^{1-\epsilon +\theta_2}}\sum_{\ell =0}^{\infty}(\ell + 1)D\left(\ell,T\right) + \frac{Cn^{C'\epsilon}}{n^{\theta_2}}\sum_{\ell =0}^{\infty}(\ell + 1)p_{\ell}W_{\infty} \le \frac{2C}{n^{1 -(1+C')\epsilon +\theta_2}}Z\left(T\right) + \frac{2C}{n^{\theta_2-C'\epsilon}}W_{\infty}.
\end{multline}
Using \eqref{quant1}, \eqref{quant2} and the union bound, we obtain for any $\omega'>0$,
\begin{multline*}
\E\left(n^{\omega'}\sup_{t \in [0,2\epsilon \log n / \lambda^*]}\left| e^{-\lambda^* T}\sum_{\ell =0}^{\infty} \lambda_{\ell}(t)D\left(\ell,T\right) - \frac{1}{\lambda^* m^{\star}}\sum_{\ell =0}^{\infty} \lambda_{\ell}(t)p_{\ell}W_{\infty}\right|\right)\\
\le \E\left(n^{\omega'}\sup_{1 \le j \le \lfloor (2\epsilon \log n / \lambda^*)n^{\theta_2}\rfloor + 1}\left| e^{-\lambda^* T}\sum_{\ell =0}^{\infty} \lambda_{\ell}(t_j)D\left(\ell,T\right) - \frac{1}{\lambda^* m^{\star}}\sum_{\ell =0}^{\infty} \lambda_{\ell}(t_j)p_{\ell}W_{\infty}\right|\right)\\
+ \E\left(\frac{2Cn^{\omega'}}{n^{1 - (1+C')\epsilon +\theta_2}}Z\left(T\right) + \frac{2Cn^{\omega'}}{n^{\theta_2-C'\epsilon}}W_{\infty} \right)\\
\le n^{\omega'}\sum_{j=0}^{\lfloor (2\epsilon \log n / \lambda^*)n^{\theta_2}\rfloor + 1}\E\left(\left| e^{-\lambda^* T}\sum_{\ell =0}^{\infty} \lambda_{\ell}(t_j)D\left(\ell,T\right) - \frac{1}{\lambda^* m^{\star}}\sum_{\ell =0}^{\infty} \lambda_{\ell}(t_j)p_{\ell}W_{\infty}\right|\right)\\
+ n^{\omega'}\E\left(\frac{2C}{n^{1- (1+C')\epsilon +\theta_2}}Z\left(T\right) + \frac{2C}{n^{\theta_2-C'\epsilon}}W_{\infty}\right)
\le \frac{C''\epsilon \log n}{n^{\theta_1- \theta_2 - \omega'}} + \frac{C''}{n^{\theta_2 - (2+C')\epsilon - \omega'}} + \frac{C''}{n^{\theta_2 - C'\epsilon - \omega'}}
\end{multline*}
for some constant $C''>0$. Taking $\epsilon^* < \theta_2/(2+C')$ and any $\omega' < \min\{\theta_1- \theta_2,\theta_2 - (2+C')\epsilon^*,1\}$, this proves the first assertion in the lemma. The second assertion follows similarly upon noting that $\lambda_{\ell}^{(k)} \le \lambda_{\ell}$ for each $k \ge 0$ (and thus the constant $C$ in the expectation bound can be chosen uniformly over $k$) and using Corollary \ref{cor:lamLK-t-ts-diff} in place of Lemma \ref{lem:lam-t-ts-diff} (which accounts for the $(k+1)$ in the bound).
\end{proof}
% Now, we have all the ingredients to prove Theorem \ref{ratewocp}.
\begin{proof}[Proof of Theorem \ref{ratewocp}]
Take $\epsilon^{**} \le \epsilon^*$ (where $\epsilon^*$ is as in Lemma \ref{quant}) and any $\epsilon \le \epsilon^{**}$. We abbreviate
\begin{align*}
\mathcal{S}_n &:= \sup_{t \in [0,2\epsilon \log n / \lambda^*]}\left|\sum_{\ell =0}^{\infty} \lambda_{\ell}(t)D\left(\ell,\frac{1-\epsilon}{\lambda^*}\log n\right) - \frac{n^{1-\epsilon}}{\lambda^* m^{\star}}\sum_{\ell =0}^{\infty} \lambda_{\ell}(t)p_{\ell}W_{\infty}\right|,\\
\mathcal{S}_n^{(k)} &:= \sup_{t \in [0,2\epsilon \log n / \lambda^*]}\left|\sum_{\ell =0}^{\infty} \lambda_{\ell}^{(k)}(t)D\left(\ell,\frac{1-\epsilon}{\lambda^*}\log n\right) - \frac{n^{1-\epsilon}}{\lambda^* m^{\star}}\sum_{\ell =0}^{\infty} \lambda_{\ell}^{(k)}(t)p_{\ell}W_{\infty}\right|.
\end{align*}
Observe that for any $k \ge 0$, using the fact that $\lambda_{\ell}(\cdot)$ is an increasing function and $\lambda_{\ell}(0)=1$ for each $\ell \ge 0$,
\begin{multline*}
\sup_{t \in [0,2\epsilon \log n / \lambda^*]}\left| \frac{\sum_{\ell =0}^{\infty} \lambda_{\ell}^{(k)}(t)D\left(\ell,\frac{1-\epsilon}{\lambda^*}\log n\right)}{\sum_{\ell =0}^{\infty} \lambda_{\ell}(t)D\left(\ell,\frac{1-\epsilon}{\lambda^*}\log n\right)} - \frac{\sum_{\ell = 0}^{\infty} \lambda_{\ell}^{(k)}(t) p_{\ell}}{\sum_{\ell = 0}^{\infty} \lambda_{\ell}(t) p_{\ell}} \right|\\
\le \frac{\mathcal{S}_n^{(k)}}{\sum_{\ell =0}^{\infty} \lambda_{\ell}(t)D\left(\ell,\frac{1-\epsilon}{\lambda^*}\log n\right)} + \frac{\mathcal{S}_n\left(\sum_{\ell =0}^{\infty} \lambda_{\ell}^{(k)}(t)p_{\ell}W_{\infty}\right)}{\left(\sum_{\ell = 0}^{\infty} \lambda_{\ell}(t) p_{\ell}W_{\infty}\right)\left(\sum_{\ell =0}^{\infty} \lambda_{\ell}(t)D\left(\ell,\frac{1-\epsilon}{\lambda^*}\log n\right)\right)}\\
\le \frac{\mathcal{S}_n^{(k)}}{\sum_{\ell =0}^{\infty} \lambda_{\ell}(0)D\left(\ell,\frac{1-\epsilon}{\lambda^*}\log n\right)} + \frac{\mathcal{S}_n}{\left(\sum_{\ell =0}^{\infty} \lambda_{\ell}(0)D\left(\ell,\frac{1-\epsilon}{\lambda^*}\log n\right)\right)} = \frac{\mathcal{S}_n^{(k)}}{Z\left(\frac{1-\epsilon}{\lambda^*}\log n\right)} + \frac{\mathcal{S}_n}{Z\left(\frac{1-\epsilon}{\lambda^*}\log n\right)}.
\end{multline*}
Recalling $\omega_1$ from Lemma \ref{quant},
\begin{multline*}
n^{\omega_1}\sum_{k=0}^{\infty}2^{-k}\left(\sup_{t \in [0,2\epsilon \log n / \lambda^*]}\left| \frac{\sum_{\ell =0}^{\infty} \lambda_{\ell}^{(k)}(t)D\left(\ell,\frac{1-\epsilon}{\lambda^*}\log n\right)}{\sum_{\ell =0}^{\infty} \lambda_{\ell}(t)D\left(\ell,\frac{1-\epsilon}{\lambda^*}\log n\right)} - \frac{\sum_{\ell = 0}^{\infty} \lambda_{\ell}^{(k)}(t) p_{\ell}}{\sum_{\ell = 0}^{\infty} \lambda_{\ell}(t) p_{\ell}} \right|\right)\\
\le \frac{n^{1-\epsilon}}{Z\left(\frac{1-\epsilon}{\lambda^*}\log n\right)}\sum_{k=0}^{\infty}2^{-k}\left(\frac{\mathcal{S}_n^{(k)}}{n^{1-\epsilon - \omega_1}} + \frac{\mathcal{S}_n}{n^{1-\epsilon -\omega_1}}\right).
\end{multline*}
Using Lemma \ref{quant}, for any $\eta>0$,
\begin{align*}
\pr\left(\sum_{k=0}^{\infty}2^{-k}\left(\frac{\mathcal{S}_n^{(k)} + \mathcal{S}_n}{n^{1-\epsilon -\omega_1}}\right) > \eta\right) &\le \eta^{-1}\sum_{k=0}^{\infty} 2^{-k}\frac{1}{n^{1-\epsilon -\omega_1}}\E\left(\mathcal{S}_n^{(k)} + \mathcal{S}_n\right)\\
 &\le \eta^{-1}\sum_{k=0}^{\infty} 2^{-k}(k+2)Cn^{-\omega_2} \le C'\eta^{-1}n^{-\omega_2}
\end{align*}
for positive constants $C, C'$. Moreover, $\frac{n^{1-\epsilon}}{Z\left(\frac{1-\epsilon}{\lambda^*}\log n\right)} \overset{P}{\longrightarrow} \frac{\lambda^* m^{\star}}{W_{\infty}}$ as $n \rightarrow \infty$ by Lemma \ref{lem:deg_dist_quad_conv}. Combining these,
\begin{equation}\label{supb1}
n^{\omega_1}\sum_{k=0}^{\infty}2^{-k}\left(\sup_{t \in [0,2\epsilon \log n / \lambda^*]}\left| \frac{\sum_{\ell =0}^{\infty} \lambda_{\ell}^{(k)}(t)D\left(\ell,\frac{1-\epsilon}{\lambda^*}\log n\right)}{\sum_{\ell =0}^{\infty} \lambda_{\ell}(t)D\left(\ell,\frac{1-\epsilon}{\lambda^*}\log n\right)} - \frac{\sum_{\ell = 0}^{\infty} \lambda_{\ell}^{(k)}(t) p_{\ell}}{\sum_{\ell = 0}^{\infty} \lambda_{\ell}(t) p_{\ell}} \right|\right) \overset{P}{\longrightarrow} 0.
\end{equation}
Moreover, it is straightforward to check that
\begin{multline}\label{intm1}
\sup_{t \in [0,2\epsilon \log n / \lambda^*]}\left| \frac{D\left(k,\frac{1-\epsilon}{\lambda^*}\log n + t\right)}{Z\left(\frac{1-\epsilon}{\lambda^*}\log n + t\right)}  - \frac{\sum_{\ell =0}^{\infty} \lambda_{\ell}^{(k)}(t)D\left(\ell,\frac{1-\epsilon}{\lambda^*}\log n\right)}{\sum_{\ell =0}^{\infty} \lambda_{\ell}(t)D\left(\ell,\frac{1-\epsilon}{\lambda^*}\log n\right)}\right|\\
 \le \frac{1}{Z\left(\frac{1-\epsilon}{\lambda^*}\log n\right)}\sup_{t \in [0,2\epsilon \log n / \lambda^*]}\left|D\left(k,\frac{1-\epsilon}{\lambda^*}\log n + t\right) - \sum_{\ell =0}^{\infty} \lambda_{\ell}^{(k)}(t)D\left(\ell,\frac{1-\epsilon}{\lambda^*}\log n\right)\right|\\
 + \frac{1}{Z\left(\frac{1-\epsilon}{\lambda^*}\log n\right)}\sup_{t \in [0,2\epsilon \log n / \lambda^*]}\left|Z\left(\frac{1-\epsilon}{\lambda^*}\log n + t\right) - \sum_{\ell =0}^{\infty} \lambda_{\ell}(t)D\left(\ell,\frac{1-\epsilon}{\lambda^*}\log n\right)\right|.
\end{multline}
Abbreviate
\begin{align*}
\hat{\mathcal{S}}_n^{(k)} &:= \sup_{t \in [0,2\epsilon \log n / \lambda^*]}\left|D\left(k,\frac{1-\epsilon}{\lambda^*}\log n + t\right) - \sum_{\ell =0}^{\infty} \lambda_{\ell}^{(k)}(t)D\left(\ell,\frac{1-\epsilon}{\lambda^*}\log n\right)\right|,\\
\hat{\mathcal{S}}_n &:= \sup_{t \in [0,2\epsilon \log n / \lambda^*]}\left|Z\left(\frac{1-\epsilon}{\lambda^*}\log n + t\right) - \sum_{\ell =0}^{\infty} \lambda_{\ell}(t)D\left(\ell,\frac{1-\epsilon}{\lambda^*}\log n\right)\right|.
\end{align*}
By conditioning on $\cF_n\left(\frac{1-\epsilon}{\lambda^*}\log n\right)$ and applying Lemma \ref{lem:Nt-sumNlam-close}, we obtain $\omega_1' \in (0,1) , \omega_2'>0$ not depending on $\epsilon$ such that for any $\eta>0$,
\begin{multline}\label{intm2}
\pr\left(\sum_{k=0}^{\infty}2^{-k}\left(\frac{\hat{\mathcal{S}}_n^{(k)}}{Z\left(\frac{1-\epsilon}{\lambda^*}\log n\right)^{1-\omega_1'}}\right) > \eta \ \Big| \ \cF_n\left(\frac{1-\epsilon}{\lambda^*}\log n\right)\right)\\
=\pr\left(\sum_{k=0}^{\infty}2^{-k}\left(\frac{\hat{\mathcal{S}}_n^{(k)}}{Z\left(\frac{1-\epsilon}{\lambda^*}\log n\right)^{1-\omega_1'}}\right) > \sum_{k=0}^{\infty}\left(\frac{3}{2}\right)^{-k}\frac{\eta}{3} \ \Big| \ \cF_n\left(\frac{1-\epsilon}{\lambda^*}\log n\right)\right)\\
\le \sum_{k=0}^{\infty}\pr\left(\frac{\hat{\mathcal{S}}_n^{(k)}}{Z\left(\frac{1-\epsilon}{\lambda^*}\log n\right)^{1-\omega_1'}} > \left(\frac{4}{3}\right)^{k}\frac{\eta}{3} \ \Big| \ \cF_n\left(\frac{1-\epsilon}{\lambda^*}\log n\right) \right)\\
\le Ce^{C'2\epsilon \log n / \lambda^*}\eta^{-2}Z\left(\frac{1-\epsilon}{\lambda^*}\log n\right)^{-\omega_2'}\sum_{k=0}^{\infty}(k+1)^2\left(\frac{3}{4}\right)^{2k} = C'n^{2C'\epsilon/\lambda^*} \eta^{-2}Z\left(\frac{1-\epsilon}{\lambda^*}\log n\right)^{-\omega_2'}
\end{multline}
for positive constants $C, C'$. As $\frac{n^{1-\epsilon}}{Z\left(\frac{1-\epsilon}{\lambda^*}\log n\right)} \overset{P}{\longrightarrow} \frac{\lambda^* m^{\star}}{W_{\infty}}$, the bound above converges to zero almost surely if $\epsilon^{**}$ is chosen sufficiently small and $\epsilon \le \epsilon^{**}$.
Similarly,
\begin{equation}\label{intm3}
\pr\left(\sum_{k=0}^{\infty}2^{-k}\left(\frac{\hat{\mathcal{S}}_n}{Z\left(\frac{1-\epsilon}{\lambda^*}\log n\right)^{1-\omega_1'}}\right) > \epsilon \ \Big| \ \cF_n\left(\frac{1-\epsilon}{\lambda^*}\log n\right)\right) \le C' n^{2C'\epsilon/\lambda^*}\epsilon^{-2}Z\left(\frac{1-\epsilon}{\lambda^*}\log n\right)^{-\omega_2}.
\end{equation}
Using \eqref{intm1}, \eqref{intm2}, \eqref{intm3} and recalling that $\frac{n^{1-\epsilon}}{Z\left(\frac{1-\epsilon}{\lambda^*}\log n\right)} \overset{P}{\longrightarrow} \frac{\lambda^* m^{\star}}{W_{\infty}}$ as $n \rightarrow \infty$, we conclude
\begin{equation}\label{supb2}
n^{(1-\epsilon)\omega_1'}\sum_{k=0}^{\infty}2^{-k}\left(\sup_{t \in [0,2\epsilon \log n / \lambda^*]}\left| \frac{D\left(k,\frac{1-\epsilon}{\lambda^*}\log n + t\right)}{Z\left(\frac{1-\epsilon}{\lambda^*}\log n + t\right)}  - \frac{\sum_{\ell =0}^{\infty} \lambda_{\ell}^{(k)}(t)D\left(\ell,\frac{1-\epsilon}{\lambda^*}\log n\right)}{\sum_{\ell =0}^{\infty} \lambda_{\ell}(t)D\left(\ell,\frac{1-\epsilon}{\lambda^*}\log n\right)}\right|\right) \overset{P}{\longrightarrow} 0.
\end{equation}
Choosing $\omega^* = \min\{\omega_1, (1-\epsilon)\omega_1'\}$, we conclude from \eqref{supb1} and \eqref{supb2} that
\begin{equation}\label{supb4}
n^{\omega^*}\sum_{k=0}^{\infty}2^{-k}\left(\sup_{t \in [0,2\epsilon \log n / \lambda^*]}\left| \frac{D\left(k,\frac{1-\epsilon}{\lambda^*}\log n + t\right)}{Z\left(\frac{1-\epsilon}{\lambda^*}\log n + t\right)} - \frac{\sum_{\ell = 0}^{\infty} \lambda_{\ell}^{(k)}(t) p_{\ell}}{\sum_{\ell = 0}^{\infty} \lambda_{\ell}(t) p_{\ell}}\right|\right) \overset{P}{\longrightarrow} 0.
\end{equation}
Finally, we claim that for each $k \ge 0$, $t \ge 0$,
\begin{equation}\label{supb5}
\sum_{\ell = 0}^{\infty} \lambda_{\ell}^{(k)}(t) p_{\ell} \big /\sum_{\ell = 0}^{\infty} \lambda_{\ell}(t) p_{\ell} = p_k.
\end{equation}
To see this, observe that the following limits hold as $n \rightarrow \infty$:
$
\frac{Z\left(\frac{1-\epsilon}{\lambda^*}\log n + t\right)}{n^{1-\epsilon}} \overset{P}{\longrightarrow} \frac{e^{\lambda^*t} W_{\infty}}{\lambda^* m^{\star}},$ and \\ $\frac{D(k,\frac{1-\epsilon}{\lambda^*}\log n + t)}{n^{1-\epsilon}} \overset{P}{\longrightarrow} \frac{p_k e^{\lambda^*t} W_{\infty}}{\lambda^* m^{\star}}
$. Thus
$
\frac{D\left(k,\frac{1-\epsilon}{\lambda^*}\log n + t\right)}{Z\left(\frac{1-\epsilon}{\lambda^*}\log n + t\right)} \overset{P}{\longrightarrow} p_k.
$
But from \eqref{supb4},
$$
\frac{D\left(k,\frac{1-\epsilon}{\lambda^*}\log n + t\right)}{Z\left(\frac{1-\epsilon}{\lambda^*}\log n + t\right)} \overset{P}{\longrightarrow} \frac{\sum_{\ell = 0}^{\infty} \lambda_{\ell}^{(k)}(t) p_{\ell}}{\sum_{\ell = 0}^{\infty} \lambda_{\ell}(t) p_{\ell}}.
$$
\eqref{supb5} follows from the above two observations. The \chnr{theorem} now follows from \eqref{supb4} and \eqref{supb5}.
\end{proof}

\section{Proofs: Change point detection}
\label{sec:proofs-cpd}
\chsb{Throughout this section, we assume that $f_0$ satisfies Assumptions \ref{ass:attach-func}, \ref{kickass} and \ref{varass}, and $f_1$ satisfies Assumptions \ref{ass:attach-func} and \ref{xlogx}.}
Recall $\lambda_{\ell}, \lambda_{\ell}^{(k)}$ for $k, \ell \ge 0$ defined in \eqref{lambdadef}
and the functional
$
\Phi_a : \mathcal{P} \rightarrow \mathcal{P}
$
defined for each $a>0$ in \eqref{phidef}.
%Let $\mathcal{P}$ denote the collection of all probability measures on $\mathbb{N} \cup \{0\}$. For each $a>0$, consider the functional
%$
%\Phi_a : \mathcal{P} \rightarrow \mathcal{P}
%$
%given by
%$$
%\Phi_a(\mathbf{p}) = \left(\frac{\sum_{\ell = 0}^{\infty} p_{\ell}\lambda_{\ell}^{(k)}(a)}{\sum_{\ell = 0}^{\infty} p_{\ell}\lambda_{\ell}(a)}\right)_{k \ge 0}
%$$
%where $\mathbf{p}=(p_0, p_1, \dots) \in \mathcal{P}$. Let $\mathbf{p}^{1} = (p^{1}_0, p^{1}_1, \dots)$. 
\begin{lemma}\label{limop}
$\lim_{a \rightarrow \infty} \Phi_a(\mathbf{p}) = \mathbf{p}^1$ (where the limit is taken in the coordinate-wise sense).
\end{lemma}
\begin{proof}
\chsb{As $f_1$ satisfies Assumptions \ref{ass:attach-func} and \ref{xlogx},} for each $k \ge 0$, by Lemma \ref{lem:deg_dist_quad_conv} (ii), % TODO: make sure 6.7 includes this convergence
$
\lim_{t \rightarrow \infty} e^{-\lambda_1^* t}m_{f_1}(t) = (\lambda_1^* m_1^{\star})^{-1}$ and
$\lim_{t \rightarrow \infty} e^{-\lambda_1^* t}m^{(k)}_{f_1}(t) = p^1_k/ (\lambda_1^* m_1^{\star})
$
and consequently,
\begin{equation}\label{lime}
\lim_{t \rightarrow \infty} e^{-\lambda_1^* t}\lambda_{\ell}(t) = w_{\ell} / (\lambda_1^* m_1^{\star}), \ \ \ \lim_{t \rightarrow \infty} e^{-\lambda_1^* t}\lambda^{(k)}_{\ell}(t) = p^1_k w_{\ell} / (\lambda_1^* m_1^{\star}).
\end{equation}
Moreover, it is easy to see from \eqref{lambdadef} that for any $\ell, k \ge 0$, $e^{-\lambda_1^* t}\lambda_{\ell}(t) \le 1 + \left(\sup_{u \ge 0}e^{-\lambda_1^* u}m_{f_1}(u)\right)w_{\ell}$ and $e^{-\lambda_1^* t}\lambda_{\ell}^{(k)}(t) \le 1 + \left(\sup_{u \ge 0}e^{-\lambda_1^* u}m_{f_1}(u)\right)w_{\ell}$ for all $t \ge 0$ and this bound is finite.
By this observation, we can apply the dominated convergence theorem and \eqref{lime} in the formula of $\Phi_a(\mathbf{p})$ to obtain the lemma.
\end{proof}
\begin{lemma}\label{MP}
For any $s,t \ge 0$ and any $j, k \ge 0$,
$$
\sum_{\ell=0}^{\infty}\lambda_{j}^{(\ell)}(t) \lambda_{\ell}(s) = \lambda_{j}(s + t), \ \ \ \sum_{\ell=0}^{\infty}\lambda_{j}^{(\ell)}(t) \lambda_{\ell}^{(k)}(s) = \lambda_{j}^{(k)}(s + t).
$$
Consequently, for any $ \mathbf{p} \in \mathcal{P}$, we have
$
\Phi_s(\Phi_t(\mathbf{p})) = \Phi_{s+t}(\mathbf{p}).
$
\end{lemma}
\begin{proof}
We will only prove the first assertion. The second one follows similarly. Denote by $\BP^{(j)}(\cdot)$ the continuous time branching process with attachment function $i \mapsto f_1(i+j)$ and denote by $D_n^{(j)}(\ell, t)$ the corresponding number of vertices of degree $\ell$ at time $t$ (excluding the root). Then
\begin{multline*}
\E\left(\left|\BP^{(j)}(t+s) \right| \mid \cF_n(t)\right) = \sum_{\ell = j}^{\infty}\ind\set{\xi_{f_1}^{(j)}(t) = \ell - j}\left(1 + \int_0^s m_{f_1}(s-v)\mu_{f_1}^{(\ell)}(dv)\right) \\
+ \sum_{\ell = 0}^{\infty}D_n^{(j)}(\ell, t)\left(1 + \int_0^s m_{f_1}(s-v)\mu_{f_1}^{(\ell)}(dv)\right)
\end{multline*}
where the first term denotes the expected number of vertices born to the root (counting the root itself) in the time interval $[t, t+s]$ and the second term denotes the expected number of vertices born in the time interval $[t, t+s]$ to those vertices born in the time interval $(0,t]$, \chsb{both expectations conditional on $\cF_n(t)$}. Taking expectation on both sides of the above expression and noting that $\lambda_{j}(t+s) = \E\left(\left|\BP^{(j)}(t+s)\right|\right)$ and $\E\left(D_n^{(j)}(\ell, t)\right) = \int_0^t m_{f_1}^{(\ell)}(t-u)\mu_{f_1}^{(j)}(du)$, we obtain
\begin{align*}
\lambda_{j}(t+s) &= \sum_{\ell = 0}^{\infty}\left(\pr \left( \xi_{f_1}^{(j)}(t) = \ell - j\right) + \int_0^t m_{f_1}^{(\ell)}(t-u)\mu_{f_1}^{(j)}(du)\right)\left(1 + \int_0^s m_{f_1}(s-v)\mu_{f_1}^{(\ell)}(dv)\right)\\
&= \sum_{\ell=0}^{\infty}\lambda_{j}^{(\ell)}(t) \lambda_{\ell}(s).
\end{align*}
To prove the semigroup property, note that for each $k \ge 0$,
\begin{multline*}
\left(\Phi_s(\Phi_t(\mathbf{p}))\right)_k =\left(\frac{\sum_{\ell = 0}^{\infty} \left(\Phi_t(\mathbf{p})\right)_{\ell}\lambda_{\ell}^{(k)}(s)}{\sum_{\ell = 0}^{\infty} \left(\Phi_t(\mathbf{p})\right)_{\ell}\lambda_{\ell}(s)}\right) =\left(\frac{\sum_{\ell = 0}^{\infty} \left(\sum_{j = 0}^{\infty} p_{j}\lambda_{j}^{(\ell)}(t)\right)\lambda_{\ell}^{(k)}(s)}{\sum_{\ell = 0}^{\infty} \left(\sum_{j = 0}^{\infty} p_{j}\lambda_{j}^{(\ell)}(t)\right)\lambda_{\ell}(s)}\right)\\
= \frac{\sum_{j=0}^{\infty}p_j\left(\sum_{\ell=0}^{\infty}\lambda_{j}^{(\ell)}(t) \lambda_{\ell}^{(k)}(s)\right)}{\sum_{j=0}^{\infty}p_j\left(\sum_{\ell=0}^{\infty}\lambda_{j}^{(\ell)}(t) \lambda_{\ell}(s)\right)} = \frac{\sum_{j=0}^{\infty}p_j\lambda_{j}^{(k)}(s + t)}{\sum_{j=0}^{\infty}p_j\lambda_{j}(s + t)} = \left(\Phi_{s+t}(\mathbf{p})\right)_k.
\end{multline*}
\end{proof}
\begin{lemma}\label{diffop}
For any $a >0$ and any $\mathbf{p} \in \mathcal{P}$ such that $\mathbf{p} \neq \mathbf{p}^1$, we have $\Phi_a(\mathbf{p}) \neq \mathbf{p}$.
\end{lemma}
\begin{proof}
Suppose there exists $a>0$ and $\mathbf{p} \neq \mathbf{p}_1$ such that $\Phi_a(\mathbf{p}) = \mathbf{p}$. Then by Lemma \ref{MP}, for any $n \ge 1$, $\Phi_{na}(\mathbf{p}) = \mathbf{p}$. Letting $n \rightarrow \infty$ and using Lemma \ref{limop}, we obtain $\mathbf{p}^1 = \mathbf{p}$ which gives a contradiction.
\end{proof}

% Now we are ready to prove Theorem \ref{cpconv}.

\begin{proof}[Proof of Theorem \ref{cpconv}]
Recall $\omega^*$, $\epsilon^{**}$ from Theorem \ref{ratewocp} applied to the branching process with attachment function $f_0$ and fix any $\epsilon \le \epsilon^{**}$. Let $\lambda_0^*$ denote the associated Malthusian rate. Take any $n_0 \ge 1$ such that $h_{n} \ge 1/\gamma$ for all $n \ge n_0$. Observe that for any $\eta>0$ and any $n \ge n_0$,
\begin{multline*}
\pr\left(n^{\omega^*}\sum_{k=0}^{\infty}2^{-k}\sup_{1/h_n \le t \le \gamma}\left|\frac{D(k, T_{\lfloor nt \rfloor})}{\lfloor nt \rfloor} - p_k^0\right| > \eta\right)\\
\le \pr\left( n^{\omega^*}\sum_{k=0}^{\infty}2^{-k}\left(\sup_{t \in [0,2\epsilon \log n / \lambda_0^*]}\left| \frac{D\left(\ell,\frac{1-\epsilon}{\lambda_0^*}\log n + t\right)}{Z\left(\frac{1-\epsilon}{\lambda_0^*}\log n + t\right)} - p_k^0\right|\right)> \eta\right)\\
 + \pr\left(T_{\lfloor n/ h_n \rfloor} < \frac{1-\epsilon}{\lambda_0^*}\log n\right) + \pr\left(T_{\lfloor n\gamma \rfloor} > \frac{1+\epsilon}{\lambda_0^*}\log n\right). 
\end{multline*}
The first term in the above bound converges to zero by Theorem \ref{ratewocp}. Further,
$$
\pr\left(T_{\lfloor n/ h_n \rfloor} < \frac{1-\epsilon}{\lambda_0^*}\log n\right)  \rightarrow 0
$$
because $\lambda_0^* T_{\lfloor n/ h_n \rfloor} / \log\left( n/ h_n \right) \probc 1$ as $n \rightarrow \infty$ by Lemma \ref{lem:deg_dist_quad_conv} (ii) and by assumption, $\log h_n / \log n \rightarrow 0$. Similarly,
$
\pr\left(T_{\lfloor n\gamma \rfloor} > \frac{1+\epsilon}{\lambda_0^*}\log n\right) \rightarrow 0
$
because $\lambda_0^* T_{\lfloor n\gamma \rfloor} /  \log(n\gamma) \probc 1$ as $n \rightarrow \infty$. Thus, we conclude
\begin{equation}\label{cpe0}
n^{\omega^*}\sum_{k=0}^{\infty}2^{-k}\sup_{1/h_n \le t \le \gamma}\left|\frac{D(k, T_{\lfloor nt \rfloor})}{\lfloor nt \rfloor} - p_k^0\right| \probc 0
\end{equation}
as $n \rightarrow \infty$ which, along with the fact that $\omega^* \in (0,1)$, implies
$$
n^{\omega^*}\sum_{k=0}^{\infty}2^{-k}\sup_{1/h_n \le t \le \gamma}\left|\frac{D(k, T_{\lfloor nt \rfloor})}{nt} - \frac{D(k, T_{\lfloor n/ h_n \rfloor})}{n/ h_n}\right| \probc 0.
$$
As $\frac{\log b_n}{\log n} \rightarrow 0$ as $n \rightarrow \infty$, the above implies
$
b_n\sum_{k=0}^{\infty}2^{-k}\sup_{1/h_n \le t \le \gamma}\left|\frac{D(k, T_{\lfloor nt \rfloor})}{nt} - \frac{D(k, T_{\lfloor n/ h_n \rfloor})}{n/ h_n}\right| \probc 0.
$
From this observation and the definition of $\hat{T}_n$, we conclude that
\begin{equation}\label{cpe1}
\pr\left(\hat{T}_n \ge \gamma\right) \rightarrow 1 \ \text{ as } n \rightarrow \infty.
\end{equation}
Moreover, by Theorem \ref{supthm}, for any $t > \gamma$ and any $k \ge 0$, $\left| \frac{D(k, T_{\lfloor tn \rfloor})}{tn} -  \left(\Phi_{a_t}(\mathbf{p^0})\right)_k\right| \probc 0$ and hence, by \eqref{cpe0} and the dominated convergence theorem, as $n \rightarrow \infty$,
$$
\sum_{k=0}^{\infty}2^{-k}\left|\frac{D(k, T_{\lfloor nt \rfloor})}{nt} - \frac{D(k, T_{\lfloor n/ h_n \rfloor})}{n/ h_n}\right| \probc \sum_{k=0}^{\infty}2^{-k}\left|\left(\Phi_{a_t}(\mathbf{p^0})\right)_k - p_k^0\right|.
$$
As $a_t>0$ for each $t > \gamma$ and $\vp^0 \neq \vp^1$, by Lemma \ref{diffop}, $\Phi_{a_t}(\mathbf{p^0}) \neq \mathbf{p^0}$ and hence, the limit above is strictly positive. From the definition of $\hat{T}_n$ and the above, we conclude that for each $t> \gamma$,
\begin{equation}\label{cpe2}
\pr\left(\hat{T}_n \le t\right) \rightarrow 1 \ \text{ as } n \rightarrow \infty.
\end{equation}
The theorem follows from \eqref{cpe1} and \eqref{cpe2}.
\end{proof}

% \begin{proof}[Proof of Lemma \ref{lem:degree_sparsity}]
% Consider any tree of size $n$. Let $D_n(k) = $ number of vertices of degree $k$ and let $Q_n$ be the number of occupied degrees. Then there exists a positive constant $C$ such that
% $$n - 1 = \sum_{k=1}^{n} k D_n(k) = \sum_{k=1}^{n}  k D_n(k) \mathbf{1}(D_n(k) \neq 0 ) \ge \sum_{k=1}^{n}  k  \mathbf{1}(D_n(k) \neq 0 ) \ge \sum_{k=1}^{Q_n}  k \ge C Q_n^2.$$
% This shows $Q_n = O(\sqrt{n})$ thus the first claim follows. To see the second claim, let $\tilde{c}$ be an appropriate constant and construct the first $\tilde{c} \sqrt{n}$ star graphs then link them together along a path.
% %To see the second claim, construct a tree by linking the first $O(\sqrt{n})$ star graphs with a single path.
% \end{proof}
%
% \begin{proof}[Proof of Theorem \ref{thm:worst_case_runtime}]
% Computing $\sum_{k=0}^{\infty} 2^{-k} \left |\frac{D_{m}(k)}{m} - \frac{D_{n/\log(n)}(k)}{n/\log(n)} \right|$ can be done in $O\left(Q(m) + Q(n/\log(n)) \right) = O\left(\sqrt{m} + \sqrt{n/\log(n)} \right)$ by the first claim of Lemma \ref{lem:degree_sparsity}.  Therefore, computing this quantity for each size of the tree takes $O\left(\sum_{m=1}^n  \left(\sqrt{m} + \sqrt{n/\log(n)} \right) \right) = O\left(n^{3/2} \right)$. Everything else in the algorithm takes $O(n)$ steps. The second claim of Lemma \ref{lem:degree_sparsity} shows this bound is tight.
% \end{proof}

\section*{Acknowledgements}
SBh and IC were partially supported by NSF grants DMS-1613072, DMS-1606839 and ARO grant W911NF-17-1-0010. \sba{SBh is partially supported by NSF DMS-2113662.} \sba{SBa is partially supported by the NSF CAREER award DMS-2141621. SBa and SBh were also supported in part by the NSF RTG grant DMS-2134107.} \chsb{We thank three anonymous referees and an associate editor} for many suggestions that lead to a significant improvement in the original submission. 

% \bibliographystyle{plain}


\end{document}